\documentclass[11pt]{article}
\usepackage{ifpdf}
\pdfoutput=1

\usepackage{amsmath, amssymb, latexsym, multicol, bm, placeins}

\usepackage{tikz}
\usetikzlibrary{calc,intersections,through,backgrounds}
\usepackage{tkz-euclide}
\usetikzlibrary{decorations.pathreplacing,calligraphy}
\usetikzlibrary{shapes.geometric}

\usepackage{color}

\usepackage{algpseudocode,algorithm,algorithmicx}

\usepackage{enumerate}
\usepackage{geometry}
\usepackage{epstopdf}
\usepackage[percent]{overpic}
\usepackage{subcaption}

\usepackage{amsmath,amssymb,amsthm}
\usepackage{authblk}

\usepackage{booktabs}
\usepackage{tabularx}
\usepackage{multirow}

\makeatletter\@addtoreset{equation}{section}\makeatother

\definecolor{ao}{rgb}{0.0, 0.5, 0.0}

\newcommand{\jump}[1]{\lbrack\!\lbrack\,#1\,\rbrack\!\rbrack}

 \newtheorem{theorem}{Theorem}[section]
 \newtheorem{lemma}[theorem]{Lemma}
 \newtheorem{proposition}[theorem]{Proposition}
\newtheorem{remark}[theorem]{Remark}
 \newtheorem{corollary}[theorem]{Corollary}

\newtheorem{assumption}[theorem]{Assumption}
 \newtheorem{definition}[theorem]{Definition}

\newcommand{\mesh}{\mathcal{T}}

\newcommand{\meshs}{{\mesh_*}}

\newcommand{\edge}{\mathcal{E}}
\newcommand{\edgeE}{{\mathcal{E}_E}}

\newcommand{\nodes}{\mathcal{N}}
\newcommand{\nodesE}{{\mathcal{N}_E}}

\newcommand{\RC}{{\mathcal{C}}}

\newcommand{\V}{\mathbb{V}}

\newcommand{\Vmesh}{\V_\mesh}
\newcommand{\VE}{{\V_E}}

\newcommand{\VdE}{{\V_{\partial E}}}
\newcommand{\Vmeshz}{\V_\mesh^0}

\newcommand{\Vmeshs}{\V_\meshs}

\newcommand{\cB}{{\mathcal{B}}}
\newcommand{\B}{a}
\newcommand{\Bmesh}{\mathcal{B}_\mesh}
\newcommand{\Bmeshs}{\mathcal{B}_\meshs}
\newcommand{\amesh}{a_\mesh}
\newcommand{\mmesh}{m_\mesh}
\newcommand{\aE}{a_E}
\newcommand{\mE}{m_E}
\newcommand{\Smesh}{S_\mesh}

\newcommand{\sE}{s_E}
\newcommand{\SE}{S_E}

\newcommand{\Pimesh}{{\Pi^\nabla_\mesh}}
\newcommand{\PiE}{{\Pi^\nabla_E}}

\newcommand{\Imesh}{\mathcal{I}_\mesh}
\newcommand{\IE}{\mathcal{I}_E}

\newcommand{\Pimeshz}{{\Pi^0_\mesh}}
\newcommand{\PiEz}{{\Pi^0_E}}

\newcommand{\umesh}{u_\mesh}
\newcommand{\umeshs}{u_\meshs}

\newcommand{\etamesh}{\eta_\mesh}

\newcommand{\rmesh}{r_\mesh}
\newcommand{\jmesh}{j_\mesh}

\newcommand{\vvvert}{|\!|\!|}

\newcommand{\data}{\mathcal{D}}

\def\hh{{\mathsf{h}}}

\newcommand{\ACu}{\mathcal{A}}
\newcommand{\ACA}{\mathbb{A}}
\newcommand{\ACc}{\mathbb{C}}
\newcommand{\ACf}{\mathbb{F}}

\newenvironment{algotab}
{\par\begin{samepage}%
\begin{tabbing}\ttfamily%
 \hspace*{5mm}\=\hspace{3ex}\=\hspace{3ex}\=\hspace{3ex}\=\hspace{3ex}%
\=\hspace{3ex}\=\hspace{3ex}\=\hspace{3ex}\=\hspace{3ex}\kill}%
{\end{tabbing}\end{samepage}}

\author[1]{L. Beir\~ao da Veiga \thanks{lourenco.beirao@unimib.it}}

\author[2]{C. Canuto \thanks{claudio.canuto@polito.it}}

\author[3]{R. H. Nochetto \thanks{rhn@math.umd.edu}}

\author[4]{G. Vacca \thanks{giuseppe.vacca@uniba.it}}

\author[5]{M. Verani \thanks{marco.verani@polimi.it}}

\affil[1]{Dipartimento di Matematica e Applicazioni,
Universit\`a degli Studi di Milano Bicocca,
Via Roberto Cozzi 55 - 20125 Milano, Italy}

\affil[2]{Dipartimento di Scienze Matematiche G.L. Lagrange,
 Politecnico di Torino, 
 Corso Duca degli Abruzzi 24 - 10129 Torino, Italy}
  
\affil[3]{Department of Mathematics and Institute
for Physical Science and Technology, 
University of Maryland, 
College Park - 20742, MD, USA}

\affil[4]{Dipartimento di Matematica, 
Universit\`a degli Studi di Bari, 
Via Edoardo Orabona 4  - 70125 Bari, Italy}

\affil[5]{MOX-Laboratory for Modeling and Scientific Computing,  Dipartimento di Matematica, Politecnico di Milano, 
Piazza Leonardo da Vinci 32 - 20133 Milano, Italy}

\title{Adaptive VEM for variable data: convergence and optimality}

\begin{document}
\maketitle
\

\begin{abstract}
  We design an adaptive virtual element method (AVEM) of lowest order over triangular meshes with hanging nodes in
  2d, which are treated as polygons.
  AVEM hinges on the stabilization-free a posteriori error estimators recently derived in \cite{paperA}.
  The crucial property, that also plays a central role in this paper, is that the stabilization term can be made
  arbitrarily small relative to the a posteriori error estimators upon increasing the stabilization
  parameter. Our AVEM concatenates two modules, {\tt GALERKIN} and {\tt DATA}. The former deals with
  piecewise constant data and is shown in \cite{paperA} to be a contraction between consecutive iterates. The latter
  approximates general data by piecewise constants to a desired accuracy. AVEM is shown to be convergent and quasi-optimal, in terms
  of error decay versus degrees of freedom, for solutions and data belonging to appropriate approximation classes.
  Numerical experiments illustrate
  the interplay between these two modules and provide computational evidence of optimality.
  \end{abstract}

%

\section{Introduction}\label{sec:intro}

\bigskip

Virtual element methods (VEMs) are a new paradigm for the conforming discretization of partial differential equations
(PDEs)
over polytopal meshes. They were introduced a few years ago and have seen a rapid development with an increasing
number of applications ever since \cite{volley, autostop, BBMR:2016}. Virtual element functions are
continuous piecewise polynomials on the skeleton of the polytopal mesh and are extended inside the elements in a
convenient way that avoids their explicit manipulation. This flexibility allows for global regularity, say continuity
in the context of second order PDEs, but requires dealing with projection operators and stabilization of the resulting 
discrete bilinear form to be coercive (or more generally to satisfy a discrete inf-sup condition). If the PDE has
variable data $\data = (A,c,f)$, as in our prototype boundary value problem
\begin{equation}\label{eq:prototype}
- \nabla\cdot \left(A \nabla u \right) + c u = f \quad  \text{ in } \Omega\,, \qquad u = 0 \quad \text{on } \partial\Omega\,,
\end{equation}
then $\data$ has to be further approximated to formulate the discrete counterpart of \eqref{eq:prototype}.
This is well understood in the a priori
analysis of VEMs, which deliver optimal convergence rates under minimal regularity assumptions on $\data$ and for
rather simple and practical choices of the stabilization term.

The a posteriori error analysis of VEMs approximations of \eqref{eq:prototype} initiated in \cite{daVeigaManzini:2015,Cangiani:17,BerroneBorio:2017},
along with suitable upper and lower error estimates for variable data $\data$.
The stabilization term and the residual error estimator of \cite{Cangiani:17},
which is the one more relevant to us, turn
out to be of the same order but the former is not bounded above by the energy error. This is problematic to
study convergence of any adaptive VEM (AVEM for short). We have recently tackled this crucial issue in
\cite{paperA} and shown that the stabilization term can be made arbitrarily small relative to the error
estimator upon increasing the stabilization parameter. This property is valid in 2d on newest-vertex bisection
meshes made of triangles with hanging nodes and a fixed maximal {\it global index},
which limits the level of hanging nodes. Hence,
triangles with multiple nodes are viewed as polygons for the VEM approach. This severe mesh restriction
is crucial to relate the actual VEM mesh $\mesh$ with the largest conforming submesh $\mesh^0$ of $\mesh$
and their approximation properties. 
Moreover, this leads to stabilization-free a posteriori error estimates, derived in \cite{paperA},  and
facilitates the convergence analysis of AVEM, which is the ultimate objective of this paper. We are not aware
of similar studies for AVEM even though convergence is a fundamental mathematical question of practical significance.

In contrast, the convergence and optimality analyses of adaptive finite element methods (AFEMs)
constitute a mature research field
for elliptic PDEs such as \eqref{eq:prototype}; we refer to the surveys \cite{NSV:09,NochettoVeeser:12} as well as
\cite{Axioms2014} for
details. A common approach in the AFEM literature is to assume that the linear and bilinear forms associated
with \eqref{eq:prototype} can be computed exactly. The role of quadrature is not assessed a posteriori and,
as a consequence,
the resulting AFEMs are not fully practical unless data $\data$ is piecewise polynomial. This leads to the
usual one-loop AFEMs which iterate the modules
\begin{equation}\label{eq:adaptive-loop}
\texttt{SOLVE} \,\,
\longrightarrow \,\,
\texttt{ESTIMATE} \,\,
\longrightarrow \,\,
\texttt{MARK} \,\,
\longrightarrow \,\,
\texttt{REFINE}.
\end{equation}
A valid and practical alternative is to first approximate $\data$ by piecewise polynomials to a desired
accuracy, and next run \eqref{eq:adaptive-loop} for such approximate data to achieve a comparable level of precision.
This two-step AFEM was first proposed by R. Stevenson \cite{Stevenson2007}, and further explored in
\cite{BonitoDeVoreNochetto,CohenDeVoreNochetto}.

Dealing with approximate data $\data$ is inherent to the formulation of VEMs and their
basic definition. It is thus natural in this context to think of two-step AVEMs. This is precisely
our intent in this paper, in which we design an AVEM for \eqref{eq:prototype} in two stages.
We first assume that $\data$ is
piecewise constant and introduce a one-step AVEM, the so-called {\tt GALERKIN} module,
which is shown in \cite{paperA} to possess
a contraction property between consecutive
adaptive iterations.  We next consider
variable data $\data$ and design a two-step AVEM that consists of a concatenation of the modules
{\tt DATA} and {\tt GALERKIN} in the spirit of \cite{BonitoDeVoreNochetto,CohenDeVoreNochetto,Stevenson2007}.
Given an initial mesh $\mesh_0$ and parameters $\varepsilon_0,\omega >0$,
AVEM sets $k=0$ and iterates

\medskip
{
\begin{algotab}
  \>  $[\widehat\mesh_{k},\widehat\data_{k}]={\tt DATA} \, (\mesh_k, \data, \omega \, \varepsilon_k)$ \\
  \>  $[\mesh_{k+1},u_{k+1}]={\tt GALERKIN} \, (\widehat{\mesh}_{k},\widehat\data_{k},\varepsilon_k)$ \\
  \>  $\varepsilon_{k+1}=\tfrac12 {\varepsilon_k}; ~ k \leftarrow k+1$
\end{algotab}
}
\medskip

\noindent
The module {\tt DATA} approximates $\data=(A,c,f)$ in the
spaces $\big((L^\infty(\Omega))^{2\times 2}, L^\infty(\Omega), L^2(\Omega)\big)$ by piecewise constant data
$\widehat\data_k$ on an admissible refinement
$\widehat\mesh_k$ of $\mesh_k$ to accuracy $\omega \, \varepsilon_k$.
The pair $(\widehat{\mesh}_{k},\widehat\data_{k})$
is taken by {\tt GALERKIN} to run an inner loop, with piecewise constant data $\widehat\data_{k}$ and initial mesh
$\widehat{\mesh}_{k}$, that creates the next mesh-solution pair $(\mesh_{k+1},u_{k+1})$.
The module {\tt GALERKIN} stops as soon as the error tolerance $\varepsilon_k$ is reached, which takes a finite
number of iterations because {\tt GALERKIN} is a contraction between consecutive iterates.
It is worth noticing that, in the absence of
this stopping test, {\tt GALERKIN} would converge to the solution of \eqref{eq:prototype}
corresponding to the perturbed data $\widehat\data_k$, which is not the desired solution $u$ of \eqref{eq:prototype}.
The relative resolution of the modules {\tt DATA} and {\tt GALERKIN} is critical and is
governed by the parameter $\omega > 0$.
In our numerical experiments we observe that $\omega=1$ is an adequate choice.

It is clear from its definition that this two-step AVEM converges. Concerning its optimality, we show that the number of
iterations of {\tt GALERKIN} is independent of the iteration counter $k$ and its complexity is dictated by the
approximation classes of the solution $u$ and data $\data$.
This requires $\omega$ to be sufficiently small, or equivalently that the
perturbed solution of \eqref{eq:prototype} with data $\widehat\data_k$ is much closer to $u$ than the error
tolerance $\varepsilon_k$; this is in the spirit of \cite{BonitoDeVoreNochetto,Stevenson2007}.
We also prove that the complexity of {\tt DATA} is given
by suitable approximation classes of $\data=(A,c,f)$ in the spaces
$\big((L^\infty(\Omega))^{2\times 2}, L^\infty(\Omega), L^2(\Omega)\big)$.
Altogether, this yields the following optimal
decay estimate for the energy error in terms of the number of degrees of freedom $\#\mesh_k$
\begin{equation}\label{eq:error-decay}
|u-u_k|_{1,\Omega} \leq  C(u,\data) \, \big( \#\mesh_k \big)^{-s},
\end{equation}
where $s>0$ is the worse decay rate between those of the near-best approximations errors for $u$ and for $\data$; typically $s=\frac12$ in dimension 2.

This paper is organized as follows. We present the weak formulation of \eqref{eq:prototype} in
Section \ref{sec:ingr} and recall the VEM basic ingredients in Section \ref{sec:VEMprelim}. We
discuss VEM for piecewise constant data in Section \ref{sec:discrete-pb}, including the
stabilization-free a posteriori
error estimates from \cite{paperA}. In Section \ref{sec:AVEM-pcwconstant} we design {\tt GALERKIN}, and
recall its fundamental contraction property from  \cite{paperA}.  We deal with variable data in Section \ref{sec:variable-data},
which entails a perturbation estimate for \eqref{eq:prototype}, the design of {\tt DATA}, and eventually of {\tt AVEM} for general data.
Section \ref{sec:complex-Gal} analyzes the computational cost of {\tt GALERKIN}, showing that the number of sub-iterations inside a call to {\tt GALERKIN} is 
uniformly bounded. Section \ref{sec:quasi-optimality} is devoted to the study of the quasi-optimality of AVEM: approximation classes for the solution and data are introduced, 
and the rate decay of the error in the energy norm versus the number of degrees of freedom is estimated in terms of these classes. 
Section \ref{sec:approx-data} completes the analysis, with the study of the decay of data approximation errors.
We document the interplay between the modules {\tt DATA} and {\tt GALERKIN} with several
illuminating numerical experiments in Section \ref{sec:experiments}. It is important to realize that
for mesh refinement to maintain bounded global indices, and thus admissible meshes,
further refinement beyond the marked elements
might be necessary.  In Section \ref{sec:admissible} we design and study a
procedure to make meshes admissible in the sense that the global index is uniformly bounded for all $k$.
This procedure hinges on the bisection algorithm and is of somewhat intrinsic interest. We prove that it
is optimal in terms of degrees of freedom, very much in the spirit of the completion algorithm for
conforming bisection meshes by Binev, Dahmen, and DeVore \cite{BDD:04}; see also 
\cite{NSV:09,NochettoVeeser:12,Stevenson:08}. 
We finally draw conclusions in Section \ref{S:conclusions}.

\section{The continuous problem}
\label{sec:ingr}
In a polygonal domain $\Omega \subset \mathbb{R}^2$, consider the second-order Dirichlet boundary-value problem
\begin{equation}\label{eq:pde}
- \nabla\cdot \left(A \nabla u \right) + c u = f \quad  \text{ in } \Omega\,, \qquad u = 0 \quad \text{on } \partial\Omega\,,
\end{equation}
with data $\data=(A,c,f)$, where $A \in (L^\infty(\Omega))^{2\times 2}$ is symmetric and uniformly positive-definite in $\Omega$, $c \in L^\infty(\Omega)$ is non-negative in $\Omega$, and $f \in L^2(\Omega)$. 
The variational formulation of the problem is
\begin{equation}\label{eq:pde-var}
u \in \V \, : \quad   \cB(u,v) = (f,v)_\Omega  \qquad \forall v \in \V \,,
\end{equation}
with $\V := H^1_0(\Omega)$ and 
$\cB(u,v):=\B(u,v) + m(u,v)$
where 
$$
\B(u,v):= \int_\Omega ( A \nabla u) \cdot \nabla v \, ,  \quad 
m(u,v) := \int_\Omega c \, u \, v 
$$ 
are the bilinear forms associated with \eqref{eq:pde}. Let $\vvvert \cdot \vvvert=\sqrt{\mathcal{B}(\cdot,\cdot)}$ be the energy norm, which satisfies  
\begin{equation}\label{norm:equiv}
c_\cB \vert v \vert_{1,\Omega}^2\leq \vvvert v \vvvert^2\leq c^\cB \vert v \vert^2_{1,\Omega}\quad \forall v \in \V\,,
\end{equation}
for suitable constants $0 <c_\cB \leq c^\cB$.

\section{VEM preliminaries}\label{sec:VEMprelim}

In view of the adaptive discretization of the problem, let us fix an initial conforming partition $\mesh_0$ of $\overline{\Omega}$ made of triangular elements. Let us denote by $\mesh$ any refinement of $\mesh_0$ obtained by a finite number of successive {\it newest-vertex bisections} \cite{BDD:04,NSV:09,NochettoVeeser:12,Stevenson:08}; the triangulation $\mesh$ need not be conforming, since hanging nodes may be generated by the refinement. Let $\nodes$ denote the set of nodes of $\mesh$, i.e., the collection of all vertices of the triangles in $\mesh$; a node $z \in \nodes$ is {\it proper} if it is a vertex of all triangles containing it; otherwise, it is a {\it hanging node}. Thus,  ${\cal N}={\cal P}\cup{\cal H}$ is partitioned into the union of the set ${\cal P}$ of proper nodes and the set ${\cal H}$ of hanging nodes.

Given an element $E \in \mesh$, let $\nodesE$ be the set of nodes sitting on $\partial E$; it contains the three vertices and, possibly, some hanging nodes.  If the cardinality $|\nodesE|=3$, $E$ is said a {\it proper triangle} of $\mesh$; if $|\nodesE|>3$, then according to the VEM philosophy $E$ is not viewed as a triangle, but as a polygon having $|\nodesE|$ edges, some of which are placed consecutively on the same line; the set of all edges of $E$ is denoted by $\edgeE$. 
Note that if $e \subset \partial E \cap \partial E'$, then it is an edge for both elements; consequently, it is meaningful to define the {\it skeleton} of the triangulation $\mesh$ by setting $\edge = \edge_\mesh := \bigcup_{E \in \mesh} \edgeE$. Throughout the paper, we will set $h_E = |E|^{1/2}$ for an element and $h_e=|e|$ for an edge.

\medskip
The concept of {\it global index} of a hanging node, introduced in \cite{paperA}, will be crucial in the sequel. To define it, let us first observe that any hanging node $\bm{x} \in {\cal H}$ has been obtained through a newest-vertex bisection by halving an edge of a triangle in the preceding triangulation; denoting by  $\bm{x}', \bm{x}'' \in  {\cal N}$ the endpoints of such edge, let us set ${\mathbf B}(\bm{x})=\{\bm{x}', \bm{x}''\}$.

\begin{definition}[Global index of a node and a partition]\label{def:node-index}
The global index $\lambda$ of a node $\bm{x} \in {\cal N}$  is recursively defined as follows:
\begin{itemize}
\item If $\bm{x} \in {\cal P}$, then set $\lambda(\bm{x}):=0$;
\item If $\bm{x} \in {\cal H}$,  with $\bm{x}', \bm{x}'' \in  {\mathbf B}(\bm{x})$, then set $\lambda(\bm{x}):=\max\big(\lambda(\bm{x}'), \lambda(\bm{x}'')\big)+1$.
\end{itemize}
The global index of the partition $\mesh$ is defined as 
$\Lambda_{\mesh} := \displaystyle{\max_{\bm{x} \in {\cal N}} \lambda(\bm{x}) }$.
\end{definition}


\begin{definition}[$\Lambda$-admissible partitions]\label{def:Lambda-partitions}
Given a constant $\Lambda \geq 1$, a non-conforming partition $\mesh$ is said to be $\Lambda$-admissible  if 
$$
\Lambda_{\mesh} \leq \Lambda \,.
$$
\end{definition}

Starting from the initial conforming partition $\mesh_0$ (which is trivially $\Lambda$-admissible), all the subsequent non-conforming partitions generated by the module $\texttt{REFINE}$ in the sequel will remain $\Lambda$-admissible due to the algorithm $\texttt{CREATE\_ADMISSIBLE\_CHAIN}$ studied in Section \ref{sec:admissible}. We refer to
\cite{BonitoNochetto:10} for a similar algorithm in the context of dG approximations.

\begin{remark}\label{rem:bound-global-index}
{\rm
The condition that $\mesh$ is $\Lambda$-admissible has the following implications for each element $E \in \mesh$:
\begin{itemize}
\item If $L \subset \partial E$ is one of the three sides of the triangle $E$, then $L$ may contain at most $2^\Lambda-1$ hanging nodes; consequently, $|{\cal N}_E| \leq 3\cdot 2^\Lambda$.

\item If $e \subset \partial E$ is any edge, then $h_e \simeq h_E$, where the hidden constants only depend on the shape of the initial triangulation $\mesh_0$ and possibly on $\Lambda$.
\end{itemize}
}
\end{remark}

In the following $C$ will denote a generic positive constant independent of the mesh $\mesh$ but which may depend on $\Omega$, on the initial partition $\mesh_0$,
on the data $\data$ and on the constant $\Lambda$ (cf. Definition \ref{def:Lambda-partitions})
and that may change at each occurrence,
whereas the symbol $\lesssim$ will denote a bound up to $C$.

\subsection{VEM spaces and projectors}

Although the results of the present paper apply to a wider set of VEM spaces \cite{volley,projectors,BBMR:2016}, we prefer to focus the attention on the so-called {\it enhanced} VEM space.  We will be brief and refer to \cite{paperA} for a more detailed description which adopts the same notation. We start with the projector $\PiE : H^1(E) \to \mathbb{P}_1(E)$, which is is defined by the conditions
\begin{equation}\label{eq:def-PinablaE}
(\nabla (v - \PiE v), \nabla q_1)_E = 0 \quad \forall q_1 \in \mathbb{P}_1(E), \qquad  \int_{\partial E} (v-\PiE v) = 0 \, .
\end{equation} 
To introduce the space of discrete functions in $\Omega$ associated with $\mesh$,
for each element $E \in \mesh$ we define
\begin{equation}\label{vem:choice:2} 
\begin{aligned}
& \VdE := \{ v \in {\cal C}^0(\partial E) : v_{|e} \in \mathbb{P}_1(e) \ \forall e \in \edgeE \} \, ,
\\
& \VE := \big\{ v \in H^1(E) \ : \ v_{|\partial E} \in \VdE , \ \Delta v \in \mathbb{P}_1(E) \, ,
\int_E (v - \PiE v) q_1 = 0 \ \forall q_1 \in \mathbb{P}_1(E)
\big\} \, .
\end{aligned}
\end{equation}

Obviously $\mathbb{P}_1(E) \subseteq \VE$ and, if $E$ is a proper triangle, then $\VE = \mathbb{P}_1(E)$.
Once the local spaces $\VE$ are defined, we introduce the global discrete space
\begin{equation}\label{eq:def-VT}
\Vmesh := \{ v \in \V :  \ v_{|E} \in \VE \ \ \forall E \in \mesh \}\,.
\end{equation} 
Note that functions in $\Vmesh$ are piecewise affine on the skeleton $\edge$ and are globally continuous. A \emph{set of degrees of freedom} for the space $\Vmesh$ is given by the pointwise evaluation at all (internal) mesh vertices.

We also define the subspace of continuous, piecewise affine functions on $\mesh$ 
\begin{equation}\label{eq:def-VT0}
\Vmeshz := \{ v \in \V :  \ v_{|E} \in \mathbb{P}_1(E) \ \ \forall E \in \mesh \} \subseteq \Vmesh \, .
\end{equation} 
This subspace was crucial in \cite{paperA} to get a stabilization-free a posteriori error estimate,
  and will play an essential role in this paper as well to remove the stabilization term from several estimates.

The discretization of Problem \eqref{eq:pde} will involve the following global projection operators
\begin{equation}
\label{eq:def-Pinabla}
\Pimesh : \Vmesh \to \mathbb{P}_1(\mesh),
\qquad
\Imesh : \Vmesh \to \mathbb{P}_1(\mesh),
\qquad
\Pimeshz : L^2(\Omega) \to \mathbb{P}_1(\mesh),
\end{equation}
where $\mathbb{P}_1(\mesh)$ denotes the space of (discontinuous) piecewise linear polynomials over $\mesh$.
We define these operators in terms of their local counterparts. In fact, for each element $E \in \mesh$,
$\Pimesh$ restricts to the local elliptic projection operator $\PiE$ in
\eqref{eq:def-PinablaE}, $\Imesh$ restricts to the local
Lagrange interpolation operator $\IE : \VE \to \mathbb{P}_1(E)$ at the vertices of $E$, and $\PiEz$
restricts to the local $L^2$-orthogonal projection operator $\PiEz : L^2(E) \to \mathbb{P}_1(E)$.
It turns out that $\PiEz = \PiE$ on $\VE$, because of the definition \eqref{vem:choice:2} of the space $\VE$,
and that $\PiE$ is computable on $\VE$ in terms of the degrees of freedom \cite{volley,paperA}.
Furthermore, in view of the definition \eqref{eq:def-VT0} of $\Vmeshz$, $\Pimesh v = \Imesh v = v$
for all $v \in \Vmeshz$.

\section{A Virtual Element Method with piecewise constant data}\label{sec:discrete-pb}
In this section we briefly summarize the definition and certain properties of the virtual element discretization of \eqref{eq:pde-var} introduced in \cite{paperA} under the following assumption.
\begin{assumption}[coefficients and right-hand side of the equation]\label{ass:constant-coeff}
The data  $\data=(A,c,f)$ in \eqref{eq:pde} are constant in each element $E$ of $\mesh$.
\end{assumption}
For any $E \in \mesh$ we use the following notation: $A_E= A_{|E} \in \mathbb{R}^{2 \times 2}$, $c_E = c_{|E} \in \mathbb{R}$, $f_E = f_{|E} \in \mathbb{R}$.

 \subsection{The discrete problem}
 
Under the above assumption, we define  $\amesh, \mmesh : \Vmesh \times \Vmesh \to \mathbb{R}$ by
\begin{equation}\label{eq;def-aT}
\begin{alignedat}{2}
  &\amesh(v,w) : = \sum_{E \in \mesh} \aE(v,w) \, ,
  \qquad
  && \aE(v,w) := \int_E  (A_E \nabla \PiE v) \cdot \nabla \PiE w \, ,
  \\
  &\mmesh(v,w) : = \sum_{E \in \mesh} \mE(v,w) \, ,
  \qquad
  &&\mE(v,w)  :=  c_E \int_E \PiE v \, \PiE w  \, .
\end{alignedat}
\end{equation}
Next, for any $E \in \mesh$, we introduce the stabilization symmetric bilinear form $\sE : \VE \times \VE \to \mathbb{R}$ 
\begin{equation}\label{eq:stab-dofidofi}
\sE(v,w) = \sum_{i=1}^{\nodesE} v({\bm x}_i) w({\bm x}_i) \ ,
\end{equation}
with $\{ {\bm x}_i \}_{i=1}^{\nodesE}$ denoting the nodes of $E$.   This form controls the kernel of $\aE$ on $\VE / {\mathbb R}$ because it satisfies
\begin{equation}\label{eq:stab-sE}
c_s | v |_{1,E}^2 \leq \sE(v,v) \leq C_s | v |_{1,E}^2 \qquad \forall v \in \VE  / {\mathbb R} \,,
\end{equation}
for constants $C_s \geq c_s > 0$ independent of $E$; for a proof of \eqref{eq:stab-sE} we refer to \cite{BLR:2017,brenner-sung:2018}. Other choices for the stabilization form are available in the literature \cite{BLR:2017,brenner-sung:2018} and the results presented here easily extend to such cases.
With the local form $s_E$ at hand, we define the local and global stabilization forms
\begin{equation}\label{eq:def-SE}
\begin{aligned}
& \SE(v,w) := \sE(v-\IE v, w-\IE w) \qquad \forall \, v,w\in \VE \,, \\
& \Smesh(v,w) :=  \sum_{E \in \mesh}  \SE(v,w) \qquad \forall \, v,w\in \Vmesh \,.
\end{aligned}
\end{equation}
Note that from  \eqref{eq:stab-sE} we obtain
\begin{equation}\label{eq:stab-norm}
\Smesh(v,v) \simeq | v - \Imesh v|_{1,\mesh}^2 \qquad \forall v \in \Vmesh\,,
\end{equation}
where $ | \, \cdot \, |_{1,\mesh}$ denotes the broken $H^1$-seminorm over the mesh $\mesh$.

\medskip
\noindent
Finally, for all $v,w \in \Vmesh$ we define the complete bilinear form
\begin{equation}\label{eq:def-BT}
\Bmesh : \Vmesh \times \Vmesh \to \mathbb{R}\,, \qquad 
\Bmesh(v,w) := \amesh(v,w) + \mmesh(v,w ) + \gamma \Smesh(v,w) \,, 
\end{equation} 
where $\gamma \geq \gamma_0$ for some fixed $\gamma_0 >0$ is a stabilization constant independent of $\mesh$.
The following properties are an easy consequence of the definitions and bounds outlined above.

\begin{lemma}[properties of bilinear forms]\label{prop:bilin-forms}
The following properties are valid
 \begin{enumerate}[$\bullet$]
 \item For any $v \in \Vmesh$ and any $w \in \Vmeshz$, it holds
\begin{equation}\label{eq:propB1}
\amesh(v,w) = a(v,w)\,, 
\qquad \mmesh(v,w) = m(v,w)\,,
\qquad \Smesh(v,w) = 0\,.  
\end{equation}
\item The form $\Bmesh$ satisfies 
\begin{equation}\label{eq:propB2}
b\,  |v |_{1,\Omega}^2 \leq \Bmesh(v,v),  \qquad |\Bmesh(v,w)| \leq B |v|_{1,\Omega} |w|_{1,\Omega}\,, \qquad \forall v,w \in \Vmesh\,,  
\end{equation}
with continuity and coercivity constants $B \geq b >0$ independent of the triangulation $\mesh$.  
\end{enumerate}
\end{lemma}
Recalling \eqref{eq:def-BT}, direct consequence of \eqref{eq:propB1} is the following consistency result:
\begin{equation}\label{eq:consistency}
\Bmesh(v,w) = B(v, w) 
\qquad \forall v \in \Vmesh\,, \forall w \in \Vmeshz \,.  
\end{equation}
We now have all the ingredients to set the Galerkin discretization of Problem \eqref{eq:pde}:  
{\it find}
\begin{equation}\label{def-Galerkin}
{\umesh \in \Vmesh \, :} \quad \Bmesh(\umesh,v) = {\cal F}_{\mesh} (v) \qquad \forall v \in \Vmesh \, ,
\end{equation}
with discrete loading term
\begin{equation}\label{discr-rhs}
{\cal F}_{\mesh} (v) := \sum_{E \in \mesh} f_E \int_E  \PiE v \qquad \forall v \in H^1_0(\Omega) \, .
\end{equation}
{Combining \eqref{eq:propB2} with the Lax-Milgram Lemma, we obtain existence, uniqueness and stability of the solution $\umesh$ of \eqref{def-Galerkin}. Moreover, $\umesh$ satisfies the following orthogonality condition in the subspace $\V_\mesh^0$ \cite{volley,paperA}.}

\begin{lemma}[{Galerkin quasi-orthogonality}]\label{L:Galerkin-orthogonality}
The solutions $u$ of \eqref{eq:pde-var} and $\umesh$ of \eqref{def-Galerkin} satisfy
\begin{equation}\label{eq:PGO:1}
\cB(u-\umesh,v) = 0  \qquad\forall \, v\in \V_\mesh^0.
\end{equation}
\end{lemma}

\subsection{An a posteriori error estimator}

Since we are interested in building adaptive discretizations, we rely on a posteriori error control.  Hereafter we present the residual-type a posteriori estimator introduced in \cite{paperA} as a variant of the one in \cite{Cangiani:17}. To this end, recalling that $\data = (A,c,f)$ denotes the set of piecewise constant data, for any $v \in \Vmesh$ and any element $E$ let us define the internal residual over $E$  
\begin{equation}\label{eq:estim1}
\rmesh(E;v,\data) :=   f_E \,    - \, c_E \, \PiE v\,.
\end{equation}  
Similarly, for any two elements $E_1, E_2 \in \mesh$ sharing an edge $e \in \edge_{E_1}\cap \edge_{E_2}$, let us define the jump residual over $e$
\begin{equation}\label{eq:estim2}
\jmesh(e;v,\data) := \jump{A \nabla \, \Pi^\nabla_{\mesh} v}_e = (A_{E_1} \nabla \, \Pi^\nabla_{E_1} v_{|E_1}) \cdot \boldsymbol{n}_1  +  (A_{E_2} \nabla \, \Pi^\nabla_{E_2} v_{|E_2}) \cdot \boldsymbol{n}_2 \,, 
\end{equation}  
where $\boldsymbol{n}_i$ denotes the unit normal vector to $e$ pointing outward with respect to $E_i$; set $\jmesh(e;v, \data) =0$ if $e \subset \partial\Omega$. Then, taking into account Remark \ref{rem:bound-global-index}, we define the local residual estimator associated with $E$
\begin{equation}\label{eq:estim3}
\etamesh^2(E;v,\data) := h_E^2 \Vert \rmesh(E;v,\data) \Vert_{0,E}^2 \ + \ \tfrac12 \sum_{e \in \edgeE} h_E \Vert \jmesh(e;v,\data) \Vert_{0,e}^2 \;.
\end{equation}
The residual estimator localized on some subset $\mathcal{S}\subseteq \mathcal{T}$ is
\begin{equation}\label{eq:estim_loc}
\etamesh^2(\mathcal{S};v,\data) := \sum_{E \in \mathcal{S}} \etamesh^2(E;v,\data) 
\end{equation}
and the global residual estimator is
\begin{equation}\label{eq:estim4}
\etamesh^2(v,\data) := \etamesh^2(\mesh;v,\data)  = \sum_{E \in \mesh} \etamesh^2(E;v,\data) \,.
\end{equation}

{Upper and lower a posteriori bounds of the energy error are} provided by the following result, whose proof can be found in \cite[Proposition 4.1 and Corollary 4.3]{paperA}.
\begin{proposition}[{a posteriori error estimates}] \label{prop:aposteriori} 
There exist constants $C_\text{apost} > c_\text{apost}>0$  depending on $\Lambda$ and $\data$ but independent of $u$, $\mesh$, $\umesh$ and {the stabilization parameter} $\gamma$, such that
\begin{equation}\label{eq:apost1}
\begin{split}
\vert u - \umesh \vert_{1, \Omega}^2  &\leq C_\text{apost} \left( \etamesh^2(\umesh,\data) 
+ \Smesh(\umesh, \umesh) \right) \,, \\
c_\text{apost} \ \etamesh^2(\umesh,\data) & \leq \vert u - \umesh \vert_{1, \Omega}^2 + \Smesh(\umesh, \umesh) \,.
\end{split}
\end{equation} 
\end{proposition}
{The stabilization term $\Smesh(\umesh, \umesh)$ and residual estimator $\etamesh^2(\umesh,\data)$ are,
unfortunately, of the same order \cite[Section 4.1]{paperA}. However, such difficulty is handled by the following crucial result, proved in \cite[Proposition 4.4]{paperA}, which relies on the subspace $\V_\mesh^0$ and Lemma \ref{L:Galerkin-orthogonality} (Galerkin quasi-orthogonality). This shows the importance of $\V_\mesh^0$.}
\begin{proposition}[bound of the stabilization term by the residual] \label{prop:bound-ST} There exists a constant $C_B >0$, depending on $\Lambda$ but independent of $\mesh$, $\umesh$ and {the stabilization parameter} $\gamma$, such that
\begin{equation}\label{eq:bound-ST}
\gamma^2 \Smesh(\umesh, \umesh) \leq C_B \, \etamesh^2(\umesh, \data) \,.
\end{equation}
\end{proposition}

Combining \eqref{eq:apost1} and \eqref{eq:bound-ST} gives rise to the following fundamental estimate
\cite[Corollary 4.5]{paperA}.
\begin{theorem}[{stabilization-free a posteriori error estimates}]\label{Corollary:stab-free}
Assume that the {stabilization} parameter $\gamma$ is chosen to satisfy $\gamma^2 > \displaystyle{\frac{C_B}{c_\text{apost}}}$. Then it holds
\begin{equation}\label{apost:stab-free}
C_L \, \etamesh^2(\umesh,\data)    \leq    \vert u - \umesh \vert_{1, \Omega}^2  \leq C_U \, \etamesh^2(\umesh,\data)  \, ,
\end{equation} 
with $C_L=c_\text{apost} - C_B \gamma^{-2}$ and 
$C_U= C_\text{apost} \, \left(1+ C_B \gamma^{-2} \right)$.
\end{theorem}

\section{AVEM for piecewise constant data}\label{sec:AVEM-pcwconstant}
In this section,  we recall from \cite{paperA} the Adaptive Virtual Element Method (AVEM) for approximating \eqref{eq:pde-var}  under 
Assumption \ref{ass:constant-coeff},  together with  its convergence property.  In particular,  AVEM for piecewise constant data is realized by a call to the module {\tt GALERKIN} described hereafter.  Given a $\Lambda$-admissible input mesh $\widehat\mesh$, piecewise constant input data $\data$ on $\widehat\mesh$ and a tolerance $\varepsilon>0$, the module 
\begin{equation}\label{module:_GAL}
 [\mesh,u_{\mesh}]={\tt GALERKIN}(\widehat{\mesh},\data,\varepsilon)
 \end{equation}
produces  a $\Lambda$-admissible bisection refinement $\mesh$ of
$\widehat{\mesh}$ and the Galerkin approximation $u_{\mesh} \in \Vmesh$ to the solution ${u}$ of problem \eqref{eq:pde} with piecewise constant data $\data$, such that 
\begin{equation}\label{aux:AVEM:estimate}
\vvvert u -u_{\mesh} \vvvert \leq C_G \, \varepsilon \,,
\end{equation}
with $C_G=\sqrt{c^\cB C_U}$, where $c^\cB$ is defined in \eqref{norm:equiv}  and $C_U$ is defined in \eqref{apost:stab-free}.
This is obtained by iterating the classical paradigm 
\begin{equation}\label{eq:paradigm}
\texttt{SOLVE} \,\,
\longrightarrow \,\,
\texttt{ESTIMATE} \,\,
\longrightarrow \,\,
\texttt{MARK} \,\,
\longrightarrow \,\,
\texttt{REFINE}
\end{equation}
producing a sequence of $\Lambda$-admissible meshes $\{\mesh_k\}_{k\geq 0}$, with $\mesh_0=\widehat\mesh$, and associated Galerkin solutions $u_k\in\V_{\mesh_k}$ to the problem \eqref{eq:pde} with data $\mathcal{D}$. The iteration stops as soon as $\eta_{\mesh_k} (u_k, \data) \leq \varepsilon$, which is possible thanks to the convergence result stated in Theorem \ref{prop:convergence-GALERKIN} below.

The {modules} in \eqref{eq:paradigm} are defined as follows: given piecewise constant data $\data$ on $\mesh_0$\,,

\medskip
\begin{enumerate}[$\bullet$]
\item $[u_\mesh]=\texttt{SOLVE}(\mesh, \data)$ produces the Galerkin solution on the mesh $\mesh$ for data $\data$;
\item $[\{\eta_\mesh(\, \cdot \, ;u_\mesh,\data)\}]=\texttt{ESTIMATE}(\mesh, u_\mesh)$ computes the local residual estimators \eqref{eq:estim3} on the mesh $\mesh$,  {which depend on the Galerkin solution $u_\mesh$ and data $\data$;}
\item $[\mathcal{M}] = \texttt{MARK}(\mesh, \{\eta_\mesh(\, \cdot \, ;u_\mesh,\data)\}, \theta)$  implements  the D{\"o}rfler criterion \cite{dorfler}, precisely for  {a given parameter $\theta \in (0,1)$ an almost minimal} set $\mathcal{M} \subset \mesh$ is found such that
\begin{equation}
\label{eq:dorfler}
\theta \, \eta_{\mesh}^2 (u_\mesh, \data) \leq \eta_{\mesh}^2( \mathcal{M}; u_\mesh, \data) \,;
\end{equation} 
\item $[\meshs]=\texttt{REFINE}(\mesh, \mathcal{M})$ produces a $\Lambda$-admissible refinement $\meshs$ of $\mesh$, obtained by  newest-vertex bisection of all the elements in $\mathcal{M}$ and, possibly, some other elements.
\end{enumerate}

\medskip
In the procedure \texttt{REFINE}, non-admissible hanging nodes,  i.e., hanging nodes with global index larger than $\Lambda$, might be created while refining elements in $\mathcal{M}$ through newest-vertex bisection.
Thus, in order to obtain a  $\Lambda$-admissible partition $\meshs$, \texttt{REFINE} possibly refines other elements in $\mesh$. This is accomplished by applying to each $E \in  \mathcal{M} $ a procedure, termed $\texttt{CREATE\_ADMISSIBLE\_CHAIN}(\mesh, E)$, 
which identifies and refines a chain of elements starting at $E$, thereby creating a $\Lambda$-admissible partition. The loop is as follows: 

\medskip
\begin{algotab}
\label{eq:refine}
  \>  $[\meshs]=\texttt{REFINE}(\mesh, \mathcal{M})$
  \\
  \>  $\quad  \text{for } E \in \mathcal{M} \cap \mesh$ 
  \\
  \>  $\quad \quad  [\mesh] = \texttt{CREATE\_ADMISSIBLE\_CHAIN}(\mesh, E)$ 
  \\
  \>  $\quad \text{end for }$
  \\
  \>  $\quad  \text{return}(\mesh)$
\end{algotab}
\medskip

 Due to the technical nature of the procedure $\texttt{CREATE\_ADMISSIBLE\_CHAIN}$, we postpone its description and analysis to Section \ref{sec:admissible-construction}. We state now a complexity estimate for $\texttt{REFINE}$, whose proof is given at the end of that section. This result is fundamental for our optimality analysis of AVEM in Section \ref{sec:quasi-optimality} and is similar in spirit to the original estimate for the bisection method by Binev, Dahmen, and DeVore \cite{BDD:04}; see also \cite{NSV:09,NochettoVeeser:12,Stevenson:08}.

\begin{theorem}[complexity of $\texttt{REFINE}$]\label{T:complexity-REFINE}
  Let $\mesh_0$ be an initial mesh with suitable initial labeling. Let $\mesh_k$ be a 
  $\Lambda$-admissible refinement of $\mesh_0$ by newest-vertex bisection created by successive calls
  $\mesh_{j+1} =$ {\tt REFINE($\mesh_j,\mathcal{M}_j)$} for $0\leq j \leq k-1$. Then there exists a universal constant
  $C_0>0$, solely depending on $\mesh_0$ and its labeling, such that
  \begin{equation}\label{eq:complexity-REFINE}
   \# \mesh_k - \#\mesh_0 \leq C_0 \sum_{j=0}^{k-1} \#\mathcal{M}_j.
  \end{equation}
\end{theorem}    
We point out that  a different procedure, termed $\texttt{MAKE\_ADMISSIBLE}$, was used in \cite{paperA} to generate a $\Lambda$-admissible refinement. While the implementation of this procedure is simpler than the one in $\texttt{CREATE\_ADMISSIBLE\_CHAIN}$, and works well in practice, only the latter guarantees the validity of the bound \eqref{eq:complexity-REFINE}.

At last, we state the following convergence result for {\tt GALERKIN} (cf. \cite[Theorem 5.1]{paperA}) with piecewise constant data.
\begin{theorem}[{convergence of {\tt GALERKIN}}]\label{prop:convergence-GALERKIN}
There exist constants $\beta>0$  {and $\alpha\in (0,1)$} such that, choosing  {the stabilization parameter $\gamma>0$ sufficiently large} in the Definition \ref{eq:def-BT}, the approximations $u_k \in {\mathbb V}_{\mesh_k}$ defined in {\tt GALERKIN} satisfy
\begin{eqnarray}
\vvvert u - u_{k} \vvvert^2+\beta \, \eta_{\mesh_{k}}^2(u_k, \data) \lesssim \alpha^k \,, \qquad k \geq 0\,.
\end{eqnarray}
\end{theorem}

\section{AVEM for general data}\label{sec:variable-data}
In this section we describe  {the two-step AVEM for general (non-piecewise constant) data and discuss
its convergence properties. We first state the regularity of data.

\begin{assumption}[regularity of data]\label{A:assumption-data}
The data $\data = (A,c,f)$ satisfies
\[
\data \in C^0(\mesh_0;\mathbb{R}^{2\times2}) \times L^\infty(\Omega) \times L^2(\Omega)  \, ,
\]
where $C^0(\mesh_0;\mathbb{R}^{2\times2})$ denotes the space of piecewise uniformly continuous tensor fields
over $\mesh_0$.
\end{assumption}
We will see below that the regularity of $c$ and $f$ can be weakened, but not that of $A$ unless we proceed as
in \cite{BonitoDeVoreNochetto}. We could assume $c\in L^q(\Omega)$ for $1 < q < \infty$ and
$f \in H^{-1}(\Omega)$, but we will not pursue this regularity much further.
We begin with a perturbation result for the solution of the exact problem.

\subsection{Data perturbation}
Let $\widehat{\mathcal{D}}=(\widehat{A},\widehat{c},\widehat{f})$ be the element-by-element average of $\mathcal{D}=(A,c,f)$ over a partition $\mesh$ of $\Omega$,  namely
\begin{equation}\label{eq:averages}
\widehat{A}\vert_E:=A_E=\frac{1}{\vert E\vert} \int_E A \qquad 
\widehat{c}\vert_E:=c_E=\frac{1}{\vert E\vert} \int_E c \qquad 
\widehat{f}\vert_E:=f_E=\frac{1}{\vert E\vert} \int_E f \qquad \forall \,E \in\mesh.
\end{equation}
If $\alpha>0$ is the smallest eigenvalue of $A=A({\bm x})$ for all ${\bm x}\in \Omega$, then for any $\xi \in \mathbb{R}^2$
\[
\xi\cdot A({\bm x})\xi \ge \alpha|\xi|^2 \quad\forall \, {\bm x}\in \Omega
\quad\Rightarrow\quad
\xi\cdot A_E \, \xi \ge \alpha|\xi|^2 \quad\forall \,E \in\mesh,
\]    
whence the smallest eigenvalue $\widehat{\alpha}$ of $\widehat{A}$ satisfies $\widehat{\alpha}\ge\alpha$; thus $\widehat{A}$ is uniformly SPD in $\Omega$.
We view $\widehat{\data}$ as a perturbation of $\data$ and consider the corresponding bilinear form $\widehat{\mathcal{B}}(\cdot,\cdot)=\widehat{a}(\cdot,\cdot)+\widehat{m}(\cdot,\cdot)$, with
$$
\widehat{a}(u,v)=\int_{\Omega} \widehat{A} \nabla u \cdot \nabla v,
\quad \widehat{m}(u,v)=\int_{\Omega} \widehat{c} u v \quad\forall \, u,v \in \V \,,
$$
and perturbed problem 
\begin{equation}\label{eq:pde-var:pert}
\widehat{u}\in \V \, : \quad
\widehat{\mathcal{B}}(\widehat{u},v)=(\widehat{f},v)\quad \forall v\in \V.
\end{equation} 
}

\begin{lemma}[continuous dependence on data]\label{lemma-data-approx}
There exists a constant 
$C>0$, depending on $\Omega$ and the mesh shape-regularity, such that for any $1 < q \leq\infty$ and $s \in [0,1]$ satisfying $s < 2(q-1)/q$ it holds
\begin{equation}\label{pert_estimate}
\vert u-\widehat{u}\vert_{1,\Omega}  {\leq \frac{1}{\alpha} \vert u \vert_{1,\Omega}} \left( \| A-\widehat{A}\|_{L^\infty(\Omega)} + \frac{C q}{2q-2-sq} \| \hh^s (c-\widehat{c}) \|_{L^{q}(\Omega)} \right) + C \| \hh ( f-\widehat{f}) \|_{L^2(\Omega)}\,,
\end{equation}
where the mesh density $\hh$ of $\mesh$ is the piecewise constant function satisfying $\hh|_E = h_E$ for all $E \in\mesh$.
\end{lemma}
\begin{proof}
We write the difference between \eqref{eq:pde-var} and \eqref{eq:pde-var:pert} as follows:
 {
\begin{equation*}
\int_{\Omega} \left[(\widehat{A} \nabla(u-\widehat{u})) \cdot \nabla v + \widehat{c}(u-\widehat{u})v \right] =
\int_{\Omega} \left[((\widehat{A}-A) \nabla u) \cdot \nabla v + (\widehat{c}-c)u v\right] + \int_{\Omega} (f-\widehat{f})v \, .
\end{equation*}
Since $(\widehat{c},\widehat{f})$ are $L^2$-projections of $(c,f)$ on piecewise constants over $\mesh$, we readily obtain
\[
\int_{\Omega} \left[(\widehat{A} \nabla(u-\widehat{u})) \cdot \nabla v + \widehat{c}(u-\widehat{u})v \right]
= \int_{\Omega} \left[((\widehat{A}-A) \nabla u ) \cdot\nabla v + (\widehat{c}-c) (u v - \overline{uv})\right] + \int_{\Omega} (f-\widehat{f})(v-\overline{v}) \,,
\]
where the overbars denote the piecewise constant averages over $\mesh$.
Taking the test function $v=u-\widehat{u}\in\V$, and using the relation $\widehat{\alpha}\ge\alpha>0$ for
the smallest eigenvalues of $\widehat{A}$ and $A$,} standard arguments yield
\begin{equation}\label{newq-0}
\begin{aligned}
\alpha {\|\nabla v \|^2_{L^2(\Omega)}} & \leq 
\|(\widehat{A}-A)\nabla u\|_{L^2(\Omega)}\|\nabla v\|_{L^2(\Omega)} 
\\
& + \sum_{E\in\mesh} h_E^s \| c-\widehat{c} \|_{L^{q}(E)} |uv|_{W^{s}_{q'}(E)}
+ \sum_{E\in\mesh}  h_E \| f-\widehat{f} \|_{L^2(E)} \|\nabla v\|_{L^2(E)} \,,
\end{aligned}
\end{equation}
where we set $q'=q/(q-1)$. We now focus on the more involved mass term.
We start by observing that, by a standard H\"older inequality on sequences and Sobolev embeddings,
\begin{equation}\label{newq-1}
\sum_{E\in\mesh} h_E^s \| c-\widehat{c} \|_{L^{q}(E)} |uv|_{W^{s}_{q'}(E)} 
\leq\| \hh^s (c-\widehat{c}) \|_{L^{q}(\Omega)} |uv|_{W^{s}_{q'}(\Omega)} 
\leq C \| \hh^s (c-\widehat{c}) \|_{L^{q}(\Omega)} |uv|_{W^{1}_{p'}(\Omega)}  \, ,
\end{equation}
where $1/p' = 1/q' - s/2$ and $C$ depends on $\Omega$.
Consequently, for $r$ satisfying $1/r +1/2 = 1/p'$, we get
$$
|uv|_{W^{1}_{p'}(\Omega)} = |u \nabla v + v \nabla u |_{L^{p'}(\Omega)}
\leq\| u \|_{L^r(\Omega)} \| v \|_{H^1(\Omega)} +  \| v \|_{L^r(\Omega)} \| u \|_{H^1(\Omega)}.
$$
Combining the definitions of $r$ and $p'$ we easily obtain the explicit expression $r = 2 q' / (2 - sq') = {2q}/({2q-2-sq})$. 
Since $r \in [1,\infty)$, the Sobolev embedding $H^1(\Omega) \subseteq L^r(\Omega)$ and previous bound yield
\begin{equation}\label{newq-2}
|uv|_{W^{1}_{p'}(\Omega)} \leq C r \| \nabla u \|_{L^2(\Omega)} \| \nabla v \|_{L^2(\Omega)}
\end{equation}
with $C$ depending on $\Omega$.
Inequalities \eqref{newq-0}, \eqref{newq-1}, \eqref{newq-2} give
$$
\begin{aligned}
\alpha {\|\nabla v \|^2_{L^2(\Omega)}} & \leq \|(\widehat{A}-A)\nabla u\|_{L^2(\Omega)}\|\nabla v\|_{L^2(\Omega)} \\
& + C r \| \hh^s (c-\widehat{c}) \|_{L^{q}(\Omega)} \| \nabla u \|_{L^2(\Omega)} \| \nabla v \|_{L^2(\Omega)}
+ C \| \hh (f-\widehat{f}) \|_{L^{2}(\Omega)} \| \nabla v \|_{L^2(\Omega)} \,,
\end{aligned}
$$
from which we immediately deduce the asserted estimate \eqref{pert_estimate}.
\end{proof}

\begin{remark} {\rm 
The bound
$ \vert u \vert_{1,\Omega}\leq \frac{1}{\alpha}\| f\|_{H^{-1}(\Omega)}$
allows us to rewrite \eqref{pert_estimate} in terms of data
\begin{equation}\label{pert_estimate:2}
\vert u-\widehat{u}\vert_{1,\Omega} \leq \frac{1}{{\alpha}^2} \| f\|_{H^{-1}(\Omega)} \left( \| A-\widehat{A}\|_{L^\infty(\Omega)}
+ \frac{C q}{2q-2-sq} \| \hh^s (c-\widehat{c}) \|_{L^{q}(\Omega)} \right) + C \| \hh  (f-\widehat{f})  \|_{L^2(\Omega)}.
\end{equation}
}
\end{remark}

\begin{remark}{\rm 
A pair of relevant choices for $q$ in Lemma \ref{lemma-data-approx} are $q=\infty$, which allows us to take $s=1$, and $q=2$, which allows us to take any value of $s$ strictly smaller than one.  {Values of $q\le2$ can be taken, but note that the largest possible exponent $s$ tends to zero as $q \rightarrow 1$.}
}
\end{remark}

 {
\subsection{The module {\tt DATA}: piecewise constant approximation of data} \label{sec:data}
 Given $\data=(A,c,f)$ satisfying Assumption \ref{A:assumption-data}, a mesh $\mesh$ 
 and a tolerance $\varepsilon$, the module
\begin{equation}\label{eq:data-module}
[\widehat\mesh,\widehat\data]={\tt DATA}(\mesh, \data,\varepsilon)
\end{equation}
 produces a $\Lambda$-admissible bisection refinement $\widehat\mesh$ of
 $\mesh$  and  a piecewise constant approximation  $\widehat\data=(\widehat{A},\widehat{c},\widehat{f})$ of $\data$
 over $\widehat\mesh$ such that
 \begin{equation}\label{data:approximation}
 \| A-\widehat{A}\|_{L^\infty(\Omega)}+  \| \hh (c-\widehat{c}) \|_{L^\infty(\Omega)} + \| \hh (f-\widehat{f}) \|_{L^2(\Omega)} \leq \varepsilon \,,
\end{equation}
which controls the perturbation error according to Lemma \ref{lemma-data-approx} (continuous dependence on data).
In view of \eqref{data:approximation},} for any $E\in \widehat\mesh$ we introduce the following  local data error estimators
\begin{equation}
\label{eq:zeta-Acf-loc}
\begin{aligned}
\zeta_{\widehat \mesh}(E; A) := 
\Vert A - \widehat A\Vert_{L^\infty(E)} \,,
\quad
\zeta_{\widehat \mesh}(E; c)  := 
h_E \|  c-\widehat{c} \|_{L^{\infty}(E)}\,,
\quad
\zeta_{\widehat \mesh}(E; f)  := 
h_E \| f-\widehat{f} \|_{L^2(E)} \,,
\end{aligned}
\end{equation}
and the global data error estimators
\begin{equation}
\label{eq:zeta-Acf}
\begin{aligned}
\zeta_{\widehat \mesh}(A) :=
\Vert A - \widehat A\Vert_{L^\infty(\Omega)} \,,
\quad
\zeta_{\widehat \mesh}(c)  := 
\| \hh (c-\widehat{c}) \|_{L^{\infty}(\Omega)}\,,
\quad
 {\zeta_{\widehat \mesh}(f)  :=
\| \hh (f-\widehat{f}) \|_{L^2(\Omega)}} \,,
\end{aligned}
\end{equation}
and
\begin{equation}
\label{eq:zeta-data}
\zeta_{\widehat \mesh}(\data) := 
\zeta_{\widehat \mesh}(A) +
\zeta_{\widehat \mesh}(c) +
\zeta_{\widehat \mesh}(f) \,.
\end{equation}

The data error reduction is obtained by iterating the following  {loop}
\begin{equation}\label{eq:paradigm_data}
\texttt{PROJECT} \,\,
\longrightarrow \,\,
\texttt{ESTIMATE\_DATA} \,\,
\longrightarrow \,\,
\texttt{MARK\_DATA} \,\,
\longrightarrow \,\,
\texttt{REFINE} \, ,
\end{equation}
which produces a sequence of $\Lambda$-admissible meshes $\{\widehat\mesh_j\}_{j\geq 0}$, with $\widehat\mesh_0=\mesh$, and associated piecewise constant data $\widehat\data_j = (\widehat{A}_j,\widehat{c}_j,\widehat{f}_j)$ w.r.t. $\widehat\mesh_j$, that approximates the exact 
 data $\mathcal{D}$ until a $k \geq 0$ is found that satisfies $\zeta_{\widehat\mesh_k}(\data) \leq \varepsilon$. 

The  {modules} in \eqref{eq:paradigm_data} are defined as follows:

\medskip
\begin{enumerate}[$\bullet$]
\item $[\widehat\data]=\texttt{PROJECT}(\mesh, \data)$ computes the element-by-element average   $\widehat \data = (\widehat{A},\widehat{c},\widehat{f})$ of $\data$ over $\mesh$;
\item $[\{\zeta_\mesh(\, \cdot \, ;A)\}, \, 
\{\zeta_\mesh(\, \cdot \, ;c)\}, \,
\{\zeta_\mesh(\, \cdot \, ;f)\}]=\texttt{ESTIMATE\_DATA}(\mesh, \data, \widehat\data)$ computes the local data error estimators \eqref{eq:zeta-Acf-loc} on the mesh $\mesh$;
\item $[\mathcal{M}_{\data}] = \texttt{MARK\_DATA}(\mesh, \{\zeta_\mesh(\, \cdot \, ;A)\}, \, 
\{\zeta_\mesh(\, \cdot \, ;c)\}, \,
\{\zeta_\mesh(\, \cdot \, ;f)\},  \theta, \varepsilon)$  implements  the following marking criteria.
For the diffusion and the reaction terms $A$ and $c$ we apply the {\it greedy} strategy that selects
\[
\mathcal{M}_A := \{E \in \mesh \, :\,  \zeta_\mesh(E;A) \geq \tfrac13 \varepsilon \} \,,
\qquad
\mathcal{M}_c := \{E \in \mesh \, :\,  \zeta_\mesh(E;c) \geq \tfrac13 \varepsilon \} \,.
\]
For the load term  $f$, which accumulates in $\ell^2$ rather than $\ell^\infty$, we first check if $\zeta_\mesh(f) \geq \tfrac13 \varepsilon$, and if so we apply a {\it pseudo-greedy} stategy that, given a parameter $\theta \in (0,1)$,  selects
\begin{equation}
\label{eq:pseudogreedy-data}
\mathcal{M}_f := \{E \in \mesh \, :\,  \zeta_\mesh(E;f) \geq \theta \, \max_{E' \in \mesh}  \zeta_\mesh(E';f) \,.
\end{equation} 
 Finally, we let the marked set be $\mathcal{M}_\data := \mathcal{M}_A \cup \mathcal{M}_c \cup \mathcal{M}_f$. In Sect. \ref{sec:approx-data}, the optimality properties of the greedy and pseudo-greedy strategies will be assessed.
 
\item $[\widehat\mesh]=\texttt{REFINE}(\mesh, \mathcal{M}_\data)$ produces a $\Lambda$-admissible refinement $\widehat\mesh$ of $\mesh$, obtained by newest-vertex bisection of all the elements in $\mathcal{M}_\data$ and, possibly, some other elements.  This is the same procedure described in Section \ref{sec:AVEM-pcwconstant}, applied with $\mathcal{M}$ replaced by $\mathcal{M}_\data$.
\end{enumerate}

\medskip
 {Altogether, if $\widehat{u}$ denotes the exact solution of the perturbed problem \eqref{eq:pde-var:pert} with
  the output data $\widehat{\data}$ from \eqref{eq:data-module}, in view of \eqref{norm:equiv}
  and \eqref{pert_estimate:2} with $q=\infty$ there exists a constant $C_D$ depending on $\Omega$, data $\data$, and the shape-regularity constant of $\mesh_0$ such that {\tt DATA} delivers the error estimate
\begin{equation}\label{eq:pert-estimate}
\vvvert u - \widehat{u} \vvvert \leq C_D \, \varepsilon.
\end{equation}
}

\subsection{Realization of AVEM}

Hereafter, we propose an  {adaptive VEM (or AVEM) that concatenates the modules {\tt DATA} and {\tt GALERKIN}
introduced in \eqref{eq:data-module} and \eqref{module:_GAL}, respectively.
Concerning the latter module, its input now is a mesh $\widehat\mesh$ and piecewise constant data $\widehat\data$ on $\widehat\mesh$, while its output is a bisection refinement $\mesh$ of
$\widehat\mesh$ and the corresponding Galerkin approximation ${u}_{\mesh}$ 
to the exact solution $\widehat{u}$ of problem \eqref{eq:pde} with piecewise constant data $\widehat\data$. They satisfy \eqref{aux:AVEM:estimate}, namely
\begin{equation}\label{eq:AVEM:estimate}
   \vvvert \widehat{u} - {u}_{\mesh} \vvvert \leq C_G \, \varepsilon.
\end{equation}

\bigskip 
\noindent {\bf The module {\tt AVEM}}. Given an initial tolerance $\varepsilon_0>0$, a target tolerance ${\tt tol}$ and initial mesh $\mesh_0$, as well as a safety parameter $\omega \in (0,1]$, AVEM consists of the two-step algorithm:

\medskip
\begin{algotab}
  \>  $[\mesh, \umesh]=\texttt{AVEM}(\mesh_0, \varepsilon_0, \omega, {\tt tol})$
  \\
  \>  $\quad  k= 0$ 
  \\
  \>  $\quad \text{while } \varepsilon_k >  \tfrac12 {\tt tol}$
  \\
  \>  $\quad \quad  [\widehat\mesh_{k},\widehat{\data}_{k}]={\tt DATA}(\mesh_k, \data, \omega \, \varepsilon_k) $ 
  \\
  \>  $\quad \quad [\mesh_{k+1},u_{k+1}]={\tt GALERKIN}(\widehat{\mesh}_{k},\widehat{\data}_{k},\varepsilon_k)$ 
  \\
  \>  $\quad \quad  \varepsilon_{k+1}=\tfrac12 {\varepsilon_k}$
  \\
  \>  $\quad \quad  k \leftarrow k+1$
  \\
  \>  $\quad  \text{end while }$
  \\
  \>  $\quad  \text{return}(\mesh_k, u_k)$
\end{algotab}
\medskip

\begin{proposition}[convergence of \texttt{AVEM}]\label{P:convergence-AVEM}
For each $k\ge0$ the modules {\tt DATA} and {\tt GALERKIN} converge in a finite number of iterations.  
Moreover, there exists a constant $C_*$ depending solely on $\mesh_0$ such that the output of $[\mesh_{k+1},u_{k+1}]={\tt GALERKIN}(\widehat{\mesh}_{k},\widehat{\data}_{k},\varepsilon_k)$ satisfies
$\vvvert {u} -u_{k+1} \vvvert \leq C_* \varepsilon_k$ for all $k \geq 0$.    
Therefore, {\tt AVEM} stops after $K$ iterations, and delivers the estimate
\begin{equation*}
  \vvvert u -u_{K+1} \vvvert \leq C_* {\tt tol}.
\end{equation*}   
\end{proposition}
\begin{proof}
We recall that Assumption \ref{A:assumption-data} guarantees that $A$
is uniformly continuous in each element of the initial mesh $\mesh_0$.
Consequently, $\|A-\widehat{A}\|_{L^\infty(E)}$ can be made arbitrarily small upon
reducing $h_E$ for all $E\in\widehat{\mesh}_k$. Moreover, since $c\in L^\infty(\Omega)$ and
$f\in L^2(\Omega)$ in view of Assumption \ref{A:assumption-data}, the errors $ \| \hh (c-\widehat{c} )\|_{L^\infty(\Omega)}$ and
$ \|\hh (f-\widehat{f}) \|_{L^2(\Omega)}$ can also be made arbitrarily small because of
the factor $\hh$. This implies that {\tt DATA} converges to tolerance $\omega\varepsilon_k$ for every $k\ge0$
in a finite number of steps. The same is valid for {\tt GALERKIN}, this time due to Theorem \ref{prop:convergence-GALERKIN} (convergence of {\tt GALERKIN}), whence we deduce that each loop of {\tt AVEM} requires finite
iterations. Thus, the output $u_{k+1}$ satisfies
\[
\vvvert u -u_{k+1} \vvvert \leq\vvvert u - \widehat{u}_k \vvvert + \vvvert \widehat{u}_k -u_{k+1} \vvvert \le
\big( C_D + C_G \big) \varepsilon_k \quad\forall \, k \geq 0 \,,
\]
according to \eqref{eq:pert-estimate} with $\omega\varepsilon_k\le\varepsilon_k$ and \eqref{eq:AVEM:estimate}.
Finally, {\tt AVEM} terminates after $K$
loops, where $K$ satisfies $\frac{1}{2}\texttt{tol} < \varepsilon_{K}  \leq \texttt{tol}$, and the asserted estimate holds with $C_* = C_D + C_G$.
\end{proof}

This elementary proof gives neither information about the dependence of the number of sub-iterations within each
loop of {\tt AVEM} upon the iteration counter $k$,  nor insight whether  the error decays optimally in terms of degrees of freedom. Answers to these two questions will be provided in Section \ref{sec:complex-Gal} and Sections \ref{sec:quasi-optimality} and \ref{sec:approx-data}, respectively.  

}

{


\section{Computational cost of {\tt GALERKIN}} \label{sec:complex-Gal}

In the sequel, we aim at investigating the complexity of {\tt GALERKIN} within the {\tt AVEM} loops. To this end, we need some preparatory results. In order to facilitate the reader, we shall use the notation 
\begin{itemize}
\item $\mathsf{exact.sol}(\, \cdot \, )$, to indicate the exact solution to the boundary-value problem \eqref{eq:pde} with data prescribed by the argument,
\item $\mathsf{galerkin.sol}(\, \cdot \, , \, \cdot \, )$, to indicate the solution to the Galerkin problem \eqref{def-Galerkin} on the partition prescribed by the first argument, with data prescribed by the second argument.
\end{itemize}
Furthermore, for any $k \in \mathbb{N}$, let  
$(\widehat\mesh_{k},\widehat\data_{k})$ and
$(\mesh_{k+1},u_{k+1})$, respectively,  be the outputs of the module {\tt DATA}
and module {\tt GALERKIN} at iteration $k$ of {\tt AVEM}.
Then referring to \eqref{eq:def-VT}, \eqref{eq:def-Pinabla}, 
\eqref{eq:estim4},
\eqref{eq:averages}, 
we set the following notations:

\begin{table}[!h]
\centering
\begin{tabular}{cccccc}
\texttt{mesh}  
& \texttt{VEM space}
& \texttt{projection}
& \texttt{estimator} 
& \quad
& \texttt{piecewise constant data} 
\\
\vspace{1.25ex}
$\widehat{\mesh}_k$ 
& $\widehat{\V}_k := \V_{\widehat{\mesh}_k}$
& $\widehat{\Pi}^\nabla_k := \Pi^{\nabla}_{\widehat{\mesh}_k}$  
& $\widehat{\eta}_k := \eta_{\widehat{\mesh}_k}$
& \quad
& $\widehat{\data}_k  := (\widehat{A}_k, \widehat{c}_k, \widehat{f}_k)$
\\
\vspace{0.75ex}
$\mesh_k$ 
& ${\V}_k := \V_{{\mesh}_k}$
& ${\Pi}^\nabla_k := \Pi^{\nabla}_{\mesh_k}$  
& $\eta_k := \eta_{\mesh_k}$
& \quad
& $\widehat{\data}_{k-1}  := (\widehat{A}_{k-1}, \widehat{c}_{k-1}, \widehat{f}_{k-1}).$
\end{tabular}
\end{table}

\begin{lemma}[uniform boundedness of $u_k$]\label{unif-bound-gal}
For any $k \geq 1$, let $u_{k}=\mathsf{galerkin.sol}(\mesh_{k}, \widehat{\data}_{k-1})$ be the output of the module {\tt GALERKIN} at iteration $k-1$. Then it holds
\begin{equation}
\vert u_k\vert_{1,\Omega} \leq c_0 \, \Vert f \Vert_{0,\Omega}
\end{equation}
for a constant $c_0>0$ independent of $k$.
\end{lemma}
\begin{proof}
Choosing $v=u_k=u_{\mesh_k}$ in \eqref{def-Galerkin} and noting that 
$\| \widehat{f}_{k-1} \|_{0,\Omega} \leq \| f \|_{0,\Omega}$, we get
\begin{equation*}\label{coerc-Galerkin-k}
 {\cal B}_{\mesh_k}(u_{k},u_{k}) = {\cal F}_{\mesh_k} (u_{k}) \leq \Vert f \Vert_{0,\Omega}  \Vert \Pi^\nabla_{k} u_{k} \Vert_{0,\Omega}  \,.
\end{equation*}
The result follows from the uniform $H^1$-coercivity of the form ${\cal B}_{\mesh_k}$ and the $H^1$-stability of the $\Pi^\nabla_k$ operator.
\end{proof}

\begin{lemma}[data perturbation of the error estimators]\label{lemma:pert_estimat}
For any $k \geq 1$, let $(\mesh_k,u_k)$ be the output of the module {\tt GALERKIN} at iteration $k-1$ of {\tt AVEM}, i.e.
$u_{k} = \mathsf{galerkin.sol}(\mesh_k, \widehat\data_{k-1})$. 
Let $(\widehat\mesh_{k},\widehat\data_{k})$ be the output of the module {\tt DATA} at iteration $k$ of {\tt AVEM}, and $u_{k,0} =\mathsf{galerkin.sol}(\widehat\mesh_{k}, \widehat\data_{k})$. 
Then it holds
\begin{equation}\label{data_pert_estimator}
 \widehat{\eta}_{k}^2(u_{k,0},\widehat\data_{k})
 \leq c_1 \eta_{k}^2(u_k,\widehat\data_{k-1}) + c_2  \epsilon_k^2 + c_3 \vert u_{k,0}-u_k\vert^2_{1,\Omega}
\end{equation}
for suitable positive constants $c_1,c_2,c_3$.
\end{lemma}

\begin{proof}
\newcommand{\Pimeshk}{{\Pi^\nabla_{k}}}
\newcommand{\Pimeshh}{{\widehat{\Pi}^\nabla_{k}}}
\newcommand{\PiEmeshk}{{\Pi^{\nabla,E}_{\mesh_k}}}
\newcommand{\PiEmeshh}{{\Pi^{\nabla,E}_{\hat{\mesh}_{k+1}}}}

We introduce the following notation
\[
\begin{aligned}
\widehat{r}_k &:= \widehat{f}_k - \widehat{c}_k \Pimeshh u_k\, 
&\qquad \widehat{j}_{k} &:= \jump{\widehat{A}_{k} \nabla \,  \Pimeshh u_k}_{\widehat{\cal{E}}_k}\, ,\\
r_{k} &:= \widehat{f}_{k-1}- \widehat{c}_{k-1} \Pimeshk u_k\, 
&\qquad 
{j}_{k} &:= \jump{\widehat{A}_{k-1} \nabla \,  \Pimeshk u_k}_{\widehat{\cal{E}}_k}\,,
\end{aligned}
\]
and we observe that it holds
\begin{align}
\widehat{r}_{k} &= r_k + \widehat{c}_{k} (\Pimeshk u_k - \Pimeshh u_k)+  (\widehat{f}_{k}-\widehat{f}_{k-1}) + (\widehat{c}_{k-1}- \widehat{c}_{k})\Pimeshk u_k  \, , \label{intermed:1}\\
\widehat{j}_{k} &= j_k +  \jump{\widehat{A}_{k} \nabla \,  (\Pimeshh - \Pimeshk) u_k}_{\widehat{\cal{E}}_k} + \jump{(\widehat{A}_{k}-\widehat{A}_{k-1})\nabla \Pimeshk u_k}_{\widehat{\cal{E}}_k}\,.\label{intermed:2} 
\end{align}
We distinguish between refined and unrefined elements. 
Let us start from refined elements and let $E$ be an  element of $\mesh_k$ which is split into  $E_1,\ldots,E_{n} \in \widehat{\mesh}_{k}$, where $n$ depends on $E$ and it holds $\min_{1\leq i \leq n} h_{E_i} \leq h_E/2$. 
Hence, we have
\[
\begin{aligned}
\sum_{i=1}^n h_{E_i}^2 \|\widehat{r}_{k}\|_{E_i}^2 &\lesssim
\sum_{i=1}^n h_{E_i}^2 \left( \|{r}_{k}\|_{E_i}^2 + \|\widehat{c}_{k} (\Pimeshk u_k - \Pimeshh u_k) \|_{E_i}^2
\right)
\\
&\quad 
+ \sum_{i=1}^n h_{E_i}^2 
\left(\| \widehat{f}_{k}-\widehat{f}_{k-1}\|_{E_i}^2 +\| (\widehat{c}_{k-1}-\widehat{c}_{k})\Pimeshk u_k\|_{E_i}^2 \right)=: I +II\,  ,
\\
\sum_{i=1}^n \sum_{e\in\mathcal{E}_{E_i}} h_{E_i}
\| \widehat{j}_{k}\|^2_e &\lesssim
\sum_{i=1}^n \sum_{e\in\mathcal{E}_{E_i}}\left( 
h_{E_i} \|j_k\|_e^2 + 
h_{E_i} \|\jump{\widehat{A}_{k} \nabla \,  (\Pimeshh - \Pimeshk) u_k}\|_e^2 \right)
\\
&\quad +
\sum_{i=1}^n \sum_{e\in\mathcal{E}_{E_i}}\left ( h_{E_i} \| \jump{(\widehat{A}_{k}-\widehat{A}_{k-1})\nabla \Pimeshk u_k}\|_e^2\right)=: III +IV.
\end{aligned}
\]
Adapting to $I$ and $III$ the same reasoning as in the proof of \cite[Lemma 5.2]{paperA} we get 
\begin{equation}
 \widehat{\eta}_{k}^2(E; u_{k},\widehat{\data}_{k})
 \lesssim  \eta_{k}^2(E;{u}_{k}, \widehat{\data}_{k-1})
 + S_{\mesh_k(E)}(u_k,u_k) + II + IV\,.
\end{equation}
By employing  \cite[Lemma 5.3]{paperA} 
we get 
\begin{equation}
 \widehat{\eta}_{k}^2(E;u_{k,0},\widehat{\data}_{k})
 \lesssim  \eta_{k}^2(E;{u}_{k},\widehat{\data}_{k-1})
 + S_{\mesh_{k}(E)}(u_k,u_k) + \vert u_{k,0} - u_k\vert_{1,{\mesh(E)}}^2 + II+IV \,.
\end{equation}
The sum  $II+IV$ can be bounded using  H\"older's inequality, the trace inequality together with 
\eqref{data:approximation}, the stability property of $\Pi^\nabla$ and Lemma \ref{unif-bound-gal}, obtaining
$$
II+IV \lesssim \varepsilon_k^2 \,.
$$
\\
On unrefined elements $E$,  we note that  $\Pi^{\nabla, E}_k u_k = \widehat{\Pi}^{\nabla, E}_k u_k$. 
Hence, employing \eqref{intermed:1}-\eqref{intermed:2} together with \cite[Lemma 5.3]{paperA}, and estimating the terms $II + IV$ as before,
we have 
\begin{equation}
 \widehat{\eta}_{k}^2(E; u_{k,0},\widehat{\data}_{k})
 \lesssim  \eta_{k}^2(E;{u}_{k},\widehat{\data}_{k-1}) + \varepsilon_k^2 
 +  \vert u_{k,0} - u_k\vert_{1,E}^2.
\end{equation}
Finally, summing over $E$ and employing \eqref{eq:bound-ST}, we have
\begin{equation}
 \widehat{\eta}_{k}^2(u_{k,0},\widehat{\data}_{k})
 \lesssim \eta_{k}^2(u_k,\widehat{\data}_{k-1}) +  \varepsilon_k^2 + \vert u_{k,0} -u_k\vert^2_{1,\Omega}.
\end{equation}

\end{proof}

\begin{proposition}[computational cost of  {\tt{GALERKIN}}]\label{prop:compl_gal} For any $k \in \mathbb{N}$, the number $J_k$ of sub-iterations inside the call to {\tt{GALERKIN}} at iteration $k$ of {\tt AVEM} is bounded independently of $k$.
\end{proposition}

\begin{proof}

\newcommand{\Pimeshk}{{\Pi^\nabla_{k}}}
\newcommand{\Pimeshh}{{\widehat{\Pi}^\nabla_{k}}}
\newcommand{\PiEmeshk}{{\Pi^{\nabla,E}_{\mesh_k}}}
\newcommand{\PiEmeshh}{{\Pi^{\nabla,E}_{\hat{\mesh}_{k+1}}}}

We proceed in several steps. 
For any $k \in \mathbb{N}$, let  
$(\widehat\mesh_{k},\widehat\data_{k})$ and
$(\mesh_{k+1},u_{k+1})$ be the output respectively of the module {\tt DATA}
and module {\tt GALERKIN} at iteration $k$ of {\tt AVEM}.
We will use the following functions:
\begin{equation}\label{eq:def-many-functions}
\begin{aligned}
\widehat{u}_{k-1} &= \mathsf{exact.sol}(\widehat{\data}_{k-1}) \in \V 
& \qquad
\widehat{u}_{k} &= \mathsf{exact.sol}(\widehat{\data}_{k}) \in \V 
\\
{u}_{k}&=\mathsf{galerkin.sol}(\mesh_{k}, \widehat{\data}_{k-1}) 
\in \V_{k}
& \qquad  
{u}_{k+1}&=\mathsf{galerkin.sol}(\mesh_{k+1}, \widehat{\data}_{k}) 
\in \V_{k+1}
\\
{u}_{k}^{\rm en} &= \mathsf{galerkin.sol}(\widehat{\mesh}_{k}, \widehat{\data}_{k-1}) \in \widehat{\V}_k 
& \qquad  
u_{k,0} &=\mathsf{galerkin.sol}(\widehat{\mesh}_{k}, \widehat{\data}_{k}) 
\in \widehat{\V}_k \,,
\end{aligned}
\end{equation}
where the suffix ``en" stands for ``enhanced" (i.e., ${u}_{k}^{\rm en}$ is computed with the same data as $u_{k}$, but on a finer mesh).

\bigskip
\noindent{\bf{Step 1}}. Estimate of 
$\vert \widehat{u}_{k}  -  u_{k,0} \vert_{1,\Omega}$. This a consequence of the a posteriori error upper bound
\begin{equation*}
\vert  \widehat{u}_{k}  -  u_{k,0} \vert_{1,\Omega}^2\leq C_U \, \widehat{\eta}^2_{k}(u_{k,0},\widehat{\data}_{k})
\end{equation*}
given in Theorem \ref{Corollary:stab-free}.\\

\noindent{\bf Step 2}. Estimate of $\widehat{\eta}^2_{k}(u_{k,0},\widehat{\data}_{k})$.  Lemma \ref{lemma:pert_estimat} gives
 \begin{equation*}
 \widehat{\eta}^2_{k}(u_{k,0},\widehat{\data}_{k})
 \leq c_1 \eta_{k}^2(u_k, \widehat{\data}_{k-1}) + c_2  \varepsilon_k^2 + c_3 \vert u_{k,0}-u_k\vert^2_{1,\Omega}
\end{equation*}
which, in view of the input tolerance $\varepsilon_k$ appearing in the module {\tt{GALERKIN}}, implies
\begin{equation}\label{aux:eta-bound}
 \widehat{\eta}^2_{k}(u_{k,0},\widehat{\data}_{k})
\leq  c_4 \varepsilon_k^2 + c_3  \vert u_{k,0}-u_k\vert^2_{1,\Omega}
 \end{equation}
for some $c_4>0$. It remains to estimate $\vert u_{k,0}-u_k\vert^2_{1,\Omega}$ which, invoking the triangle inequality,  reduces to 
 \begin{equation}\label{aux:triang-ineq}
  \vert u_{k,0}-u_k\vert_{1,\Omega} \lesssim
  \vert u_{k,0}-u_k^{\rm en}\vert_{1,\Omega} +  \vert u_k^{\rm en} - u_k\vert_{1,\Omega} \,.
 \end{equation}
We observe that the difference $u_{k,0}-u_{k}^{\rm en}$ between the two Galerkin solutions in  $\widehat{\V}_{k}$ is the solution of the following variational problem: for any $v \in \widehat{\V}_{k}$ it holds
\[
\begin{aligned}
\int_{\Omega} &\widehat{A}_{k-1} \nabla\Pimeshh (u_{k,0}-u_{k}^{\rm en}) \cdot \nabla \Pimeshh v +
\int_{\Omega} \widehat{c}_{k-1} \Pimeshh (u_{k,0}-u_{k}^{\rm en})  \Pimeshh v +
S_{\widehat{\mesh}_k}(u_{k,0}-u_{k}^{\rm en}, v) =
\\
&=
\int_{\Omega} (\widehat{A}_{k-1} - \widehat{A}_k) \nabla\Pimeshh u_{k,0} \cdot \nabla \Pimeshh v +
\int_{\Omega} (\widehat{c}_{k-1}  - \widehat{c}_k) \Pimeshh u_{k,0}  \Pimeshh v +
\int_{\Omega} (\widehat{f}_{k}  - \widehat{f}_{k-1})  \Pimeshh v \,.
\end{aligned}
\]
 Taking $v=u_{k,0}- u_{k}^{\rm en}$,  employing on the left-hand side the uniform coercivity of the discrete bilinear term, and using on the right-hand side the triangle inequality, the Cauchy-Schwarz inequality together with \eqref{data:approximation}, and Lemma \ref{unif-bound-gal}, we get
\begin{equation}\label{aux:1:complexity}
 \vert u_{k,0}-u_{k}^{\rm en} \vert_{1,\Omega} \leq c_5 (\epsilon_{k}+\varepsilon_{k-1})=3 c_5 \varepsilon_{k}
 \end{equation}
for a proper choice of $c_5>0$.
In order to estimate $\vert u_{k}^{\rm en} - u_k\vert_{1,\Omega}$,  
we preliminary note that  $\widehat{\mesh}_{k}$ is a refinement of $\mesh_k$.  Hence, invoking \cite[Corollary 5.8]{paperA} 
we have 
$$ 
\vvvert \widehat{u}_{k-1} -   u_{k}^{\rm en} \vvvert^2 + \vvvert  u_{k}^{\rm en} - u_k \vvvert^2 \leq (1+4\delta)\vvvert \widehat{u}_{k-1} -u_k \vvvert^2
$$
which, in view of \eqref{norm:equiv}, yields
\begin{equation}\label{eq:w-u-k1}
  |u_{k}^{\rm en}- u_k|_{1,\Omega} \leq c_6  |\widehat{u}_{k-1} -u_k |_{1,\Omega}
\end{equation}
for some $c_6>0$. On the other hand, from Theorem \ref{Corollary:stab-free}, we have 
\begin{equation}\label{eq:w-u-k2}
|\widehat{u}_{k-1} - u_k |_{1,\Omega} \leq \sqrt{C_U} \, \eta_{k}(u_k,\widehat{\data}_{k-1})\leq \sqrt{C_U} \varepsilon_{k-1} =  2 \sqrt{C_U} \varepsilon_{k}\,.
\end{equation}
Thus, from eqs. \eqref{aux:triang-ineq}-\eqref{eq:w-u-k2},    we obtain 
\begin{equation}
 \vert u_{k,0} -u_k\vert_{1,\Omega}\leq (3c_5+ 2\sqrt{C_U} {c_6})\varepsilon_{k}
\end{equation}
and, employing \eqref{aux:eta-bound}, we arrive at 
$$ 
\widehat{\eta}_{k}^2(u_{k,0},\widehat{\data}_k) \leq c_7 \varepsilon_k^2
$$
for some $c_7>0$.

\noindent{\bf Step 3}. Estimate of the total error $\xi^2_{\widehat{\mesh}_k}(u_{k,0})$, where, referring to Theorem \ref{prop:convergence-GALERKIN}, for any refinement $\mesh_*$ of $\widehat{\mesh}_k$ and for any $v \in \V_{\mesh_*}$ we set
\[
\xi^2_{\mesh_*}(v):=\vert \widehat{u}_{k} - v\vert^2_{1, \Omega} +
\beta  \, \eta^2_{\mesh_*}(v, \widehat{\data}_k) \,.
\] 
Because of Steps $1$ and $2$ we have 
\[
\xi^2_{\widehat{\mesh}_k}(u_{k,0})
\leq c_7(C_U + \beta)\varepsilon_k^2=: c_8 \, \varepsilon_k^2.
\] 

\noindent{\bf Step 4}. Bound on $J_k$. Each consecutive iterate $(\mesh_{k,j},u_{k,j})$ inside {\tt GALERKIN} starting with 
$(\mesh_{k,0},u_{k,0}) = (\widehat{\mesh}_{k}, u_{k,0})$
satisfies the contraction property in Theorem \ref{prop:convergence-GALERKIN} (cf. \cite[Theorem 5.1]{paperA}).
Therefore
\[
\xi^2_{{\mesh}_{k,j}}({u}_{k,j}) \lesssim 
\alpha^j \, \xi^2_{\widehat{\mesh}_k}(u_{k,0}) \leq 
\alpha^j \, c_9 \, \varepsilon_k^2 \,,
\]
for some $c_9 >0$.
Since $J_{k}$ is the smallest value for which 
\[
\eta_{{\mesh}_{k,J_k}}({u}_{k,J_k},\widehat{\data}_{k})\leq \varepsilon_k 
\]
we have  
\[
\eta_{{\mesh}_{k,J_k-1}}({u}_{k,J_k-1},\widehat{\data}_{k})> \varepsilon_k\,.
\]
Concatenating the last two ingredients gives
\[ 
\varepsilon_k^2 \leq  \frac1\beta \, \xi^2_{{\mesh}_{k,J_k-1}}({u}_{k,J_k-1})\leq \alpha^{J_ {k}-1} \frac{c_9}{\beta} \varepsilon_k^2 \,.
\]
This in turn implies 
\[
\Big(\frac{1}{\alpha}\Big)^{J_{k}-1}\leq \frac{c_9}{\beta} \quad \Rightarrow \quad J_{k} \leq 1 + \frac{ \log ({c_9}/{\beta})}{\log(1/{\alpha})}=: J \,.
\]
We see that the upper bound $J$ of $J_{k}$ is independent of $k$. This concludes the proof.

\end{proof}


%

\section{Quasi-optimal cardinality  of AVEM} 
\label{sec:quasi-optimality}
\def\qnt{{\mathbb E}} 
\def\err{\varepsilon}

The main purpose of this section is to prove, under suitable assumptions on the solution $u$ and data $\data$, the bound \eqref{eq:error-decay} announced in the Introduction, namely the existence of constants $C(u,\data) >0$ and $s \in (0, \frac12]$ such that
\begin{equation}\label{eq:optimality bound}
|u-u_k|_{1,\Omega} \leq C(u,\data) \, \big( \#\mesh_k \big)^{-s} \,.
\end{equation}
To this end, we introduce in Sect. \ref{sec:approx-classes} certain approximation classes for functions in $\V$ and for data, tailored on the decomposition of $\Omega$ into $\Lambda$-admissible non-conforming partitions, and we assume that the solution and the data of Problem \eqref{eq:pde} belong to some of these classes. In Sect. \ref{sec: eps-approx}, we investigate the approximability properties of certain perturbations of the exact solution, namely exact solutions of \eqref{eq:pde} with perturbed coefficients. Next, in Sect. \ref{sec:optim-mark}, we consider a refinement $\meshs$ of a partition $\mesh$, and give conditions under which an optimal D\"orfler marking property holds. This allows us to prove in Sect. \ref{sec:compl-gal} an optimal estimate of the cardinality of the marked set in a call to {\tt GALERKIN}. At last, in Sect. \ref{subsec:optimality-AVEM}, we apply these results to establish the desired estimate on the rate of decay of the error produced by AVEM.

\subsection{Approximation classes}\label{sec:approx-classes}
We first introduce two families of approximation classes for a function $v \in \V$, and we show they coincide. Subsequently, we define approximation classes for the operator coefficients $A\in (L^\infty(\Omega))^{2\times 2}$ and $c \in L^\infty(\Omega)$, and for the forcing $f \in L^2(\Omega)$.

\subsubsection{Approximation classes for $v \in \V$}
We start by defining the following quantity for $v \in \V$ and $v_\mesh \in \V_\mesh$
\begin{equation}\label{eq:quant}
\qnt_\mesh^2(v,v_\mesh) :=  \vvvert v - v_\mesh \vvvert^2 + |v_\mesh - {\cal I}_\mesh v_\mesh |_{1,\mesh}^2 \, .
\end{equation}
It is worthy to observe that for $v_\mesh^0 \in \V_\mesh^0$ it obviously holds
\begin{equation}\label{eq:obs:zero}
\qnt_\mesh^2(v,v_\mesh^0) = \vvvert v - v_\mesh^0 \vvvert^2 \, .
\end{equation}


\begin{lemma}[quasi-best approximation]\label{lemma:q-opt}
Let $u$ and $u_\mesh$ be the solutions of problem \eqref{eq:pde-var} and problem \eqref{def-Galerkin}, respectively, with piecewise constant data.
There exists a constant $C^\dagger >0$, independent of $u$ and the mesh $\mesh$, such that
\begin{equation}
\qnt_\mesh^2(u,u_\mesh) \leq C^\dagger \qnt_\mesh^2(u,v_\mesh) \qquad \forall v_\mesh \in \V_\mesh \, .
\end{equation}
\end{lemma}
\begin{proof}
Let $\err_\mesh = u_\mesh - v_\mesh$. By the triangle inequality
\begin{equation}\label{eq:fin-1}
\qnt_\mesh^2(u,u_\mesh) \leq 2 \left(\vvvert u - v_\mesh \vvvert^2 + |v_\mesh - {\cal I}_\mesh v_\mesh |_{1,\mesh}^2
+ \vvvert \err_\mesh \vvvert^2 + |\err_\mesh - {\cal I}_\mesh \err_\mesh |_{1,\mesh}^2 \right) \, ,
\end{equation}
so that we only need to bound the last two terms.
First by the coercivity of the discrete bilinear form, then by recalling the discrete \eqref{def-Galerkin} and continuous \eqref{eq:pde-var} weak problems, we obtain
$$
\vvvert \err_\mesh \vvvert^2 + |\err_\mesh - {\cal I}_\mesh \err_\mesh |_{1,\mesh}^2 \le
C \Bmesh(\err_\mesh,\err_\mesh) = C \big( {\cal F}_{\mesh} (\err_\mesh) - \Bmesh(v_\mesh,\err_\mesh) \big)
= C \big( \cB(u,\err_\mesh) - \Bmesh(v_\mesh,\err_\mesh) \big) \, ,
$$
where we also used that ${\cal F}_{\mesh} (v) = (f,v)_\Omega$ since in this section we are working under a piecewise constant data assumption.
We can split the above right hand side into two terms, obtaining
\begin{equation}\label{eq:fin-2}
\vvvert \err_\mesh \vvvert^2 + |\err_\mesh - {\cal I}_\mesh \err_\mesh |_{1,\mesh}^2 
\leq T_1 + T_2
\end{equation}
with
$$
T_1 :=  \cB(u - v_\mesh,\err_\mesh) \ , \quad
T_2 :=  \cB(v_\mesh,\err_\mesh) - \Bmesh(v_\mesh,\err_\mesh) \, .
$$
The bound for the first term is trivial
\begin{equation}\label{eq:fin-3}
T_1 \leq \vvvert u - v_\mesh  \vvvert \cdot \vvvert \err_\mesh  \vvvert \, .
\end{equation}
The second term is first written explicitly recalling the expression for $\Bmesh(\cdot,\cdot)$, see \eqref{eq:def-BT}, and using the orthogonality properties of the projectors  
$$
\begin{aligned}
T_2 & = \sum_{E \in \mesh}  \int_E  \big(A_E \nabla (v_\mesh - \PiE v_\mesh) \big) \cdot \nabla \err_\mesh  \\
& \qquad \qquad \qquad \qquad +   \sum_{E \in \mesh}  \int_E  \big(c_E (v_\mesh - \PiE v_\mesh) \big) \err_\mesh
- \sE(v_\mesh-\IE v_\mesh, \err_\mesh-\IE \err_\mesh)  \\
& \leq C \big( |v_\mesh - \PiE v_\mesh|_{1,\mesh} + \| v_\mesh - \PiE v_\mesh \|_{0,\Omega} + | v_\mesh-\Imesh v_\mesh |_{1,\mesh} \big) 
\, \big( \vvvert   \err_\mesh \vvvert  + | \err_\mesh-\Imesh \err_\mesh |_{1,\mesh} \big).
\end{aligned}
$$
Since the projector $\PiE$ minimizes the distance from (discontinuous) piecewise linear functions both in the broken $H^1$ semi-norm and in the $L^2$ norm, the above bound easily yields
\begin{equation}\label{eq:fin-4}
T_2 \leq C  | v_\mesh-\Imesh v_\mesh |_{1,\mesh} \,
\big( \vvvert   \err_\mesh \vvvert  + | \err_\mesh-\Imesh \err_\mesh |_{1,\mesh} \big) \, .
\end{equation}
The result follows first combining bounds \eqref{eq:fin-2}, \eqref{eq:fin-3}, \eqref{eq:fin-4} and recalling \eqref{eq:fin-1}. 
\end{proof}

\begin{remark}{\rm 
Note that Lemma \ref{lemma:q-opt} would be false in the norm $\vvvert \cdot \vvvert$, that is without the second term in definition \eqref{eq:quant}. Indeed this would imply that if $u \in \V_\mesh$ then $u_\mesh=u$, which is well known to be false in the VE method due to the approximation of the bilinear form.
} \endproof
\end{remark}
 

We now introduce two different approximation classes, one based on the full Virtual Element space, and the other one based on the underlying piecewise linear conforming Finite Element space. Afterwards we will prove that, under the assumption of $\Lambda$-admissible partitions (cf.  Definition \ref{def:Lambda-partitions}),  such classes are equivalent.

For any $N \in {\mathbb N}$, we define the following collection of partitions:
$$
\mathbb{T}_N = \big\{ \mesh :  \mesh \text{ is $\Lambda$-admissible and satisfies }  \  \# \mesh  \leq N \big\} \ .   
$$
\begin{definition}[approximation classes of $v$] \label{def:classes} 
Given any $s \in {\mathbb R}$, $s > 0$, we define the following approximation classes
$$
\begin{aligned}
& {{\ACu}}_{s} = \big\{ v \in H^1_0(\Omega) \ : \ \exists C \in {\mathbb R} \textrm{ s.t. } \sigma_N(v) := \inf_{\mesh \in \mathbb{T}_N} \inf_{v_\mesh \in \V_\mesh} \qnt_\mesh(v,v_\mesh) \leq C N^{-s} \ \ \forall N \geq \# \mesh_0
\big\} \,, \\
& {{\ACu}}_s^0 = \big\{ v \in H^1_0(\Omega) \ : \ \exists C \in {\mathbb R} \textrm{ s.t. } \sigma_N^0(v) := \inf_{\mesh \in \mathbb{T}_N} \inf_{v_\mesh^0 \in \V_\mesh^0} \qnt_\mesh(v,v_\mesh^0) \leq C N^{-s} \ \ \forall N \geq \# \mesh_0
\big\} \,.
\end{aligned}
$$ 
and denote 
\begin{equation}
\vert v \vert_{{\ACu}_{s}}:=
\sup_{N \geq \# \mesh_0}  N^s \sigma_N(v) \,.
\end{equation}
\end{definition}

We now prove the following result on the equivalence of the approximation classes (see \cite[Proposition 5.2]{BonitoNochetto:10}).
\begin{proposition}[equivalence of classes]
The two classes in Definition \ref{def:classes} coincide, i.e.
$$
{{\ACu}}_s = {{\ACu}}_s^0 \qquad \forall s \in {\mathbb R} , \,  s > 0 \, .
$$ 
\end{proposition}
\begin{proof}
Let $s \in {\mathbb R}$,  $s > 0$. The inclusion ${{\ACu}}_s^0 \subseteq {{\ACu}}_s$
is immediate since $\V_\mesh^0 \subseteq \V_\mesh$ and thus
$$
\inf_{\mesh \in \mathbb{T}_N} \inf_{v_\mesh \in \V_\mesh} \qnt_\mesh^2(v,v_\mesh) \leq
\inf_{\mesh \in \mathbb{T}_N} \inf_{v_\mesh^0 \in \V_\mesh^0} \qnt_\mesh^2(v,v_\mesh^0) 
\qquad \forall v \in H^1_0(\Omega) \, .
$$
We now show the converse inclusion. We take a generic $v \in {{\ACu}}_s$. 
Let  $N \geq \# \mesh_0$  
, then it exists $\mesh \in \mathbb{T}_N$ and $v_\mesh \in \V_\mesh$ such that 
$$
\qnt_\mesh^2(v,v_\mesh) =  \vvvert v - v_\mesh \vvvert^2 + |v_\mesh - {\cal I}_\mesh v_\mesh |_{1,\mesh}^2
\leq C N^{-s} \, ,
$$
with $C=C(v)$ but independent of $N$.
We will exhibit an approximant in $\V_\mesh^0$ that satisfies the same bound, possibly with a different constant. We choose ${\cal I}_\mesh^0 v_\mesh \in \V_\mesh^0$, the Lagrange interpolant of $v_\mesh$ at the proper nodes of $\mesh$. Recalling observation \eqref{eq:obs:zero} and by the triangle inequality
$$
\begin{aligned}
\qnt_\mesh^2(v, {\cal I}_\mesh^0 v_\mesh) & = \vvvert v - {\cal I}_\mesh^0 v_\mesh \vvvert^2
\leq2 ( \vvvert v - v_\mesh \vvvert^2 + 
\vvvert v_\mesh - {\cal I}_\mesh^0 v_\mesh \vvvert^2 ) \\
& \leq C' ( \vvvert v - v_\mesh \vvvert^2 + 
| v_\mesh - {\cal I}_\mesh^0 v_\mesh |_{1,\Omega}^2 ) \, ,
\end{aligned}
$$
where in the current proof $C'$ denotes a generic positive constant that may change at each occurrence.
Applying \cite[Prop. 3.2]{paperA} the above bound yields
$$
\qnt_\mesh^2(v, {\cal I}_\mesh^0 v_\mesh)  \leq C' ( \vvvert v - v_\mesh \vvvert^2 + 
| v_\mesh - {\cal I}_\mesh v_\mesh |_{1,\mesh}^2 )
\leq C' \qnt_\mesh^2(v,v_\mesh) \leq C' N^{-s} \, .
$$
Since the constant $C'$ does not depend on $N$, we have shown that $v \in {{\ACu}}_s^0$. 
Therefore ${{\ACu}}_s \subseteq {{\ACu}}_s^0$, and the proof is concluded.
\end{proof} 

In the rest of the paper, we make the following assumption.
\begin{assumption}[approximability of $u$] \label{ass:approx-u}
The solution $u$ of Problem \eqref{eq:pde} belongs to ${{\ACu}}_s$ for some $s=s_u \in (0, \frac12]$.
\end{assumption}

\begin{remark}[equivalence with approximation classes on conforming partitions] {\rm
\label{rem:avem-afem} 
It is easily seen that the class ${{\ACu}}_s^0$, hence ${{\ACu}}_s$, coincides with the class ${{\ACu}}_s^c$ defined by replacing $\mathbb{T}_N$ by $\mathbb{T}_N^c= \big\{ \mesh :  \mesh \text{ is conforming and satisfies }  \  \# \mesh  \leq N \big\}$. Indeed, any $\mesh \in \mathbb{T}_N$ can be refined to produce a conforming partition $\mesh^c$, such that $\#\mesh^c \leq K\, \# \mesh$ for a positive constant $K=K_\Lambda$ solely depending on $\Lambda$. 
As a consequence, one can apply e.g.  \cite[Theorem 9.1]{BDD:04} and deduce that $u\in {\ACu}_{\frac12}$ provided $u \in W^2_p(\Omega)$ for some $p > 1$.

It must be finally observed that the important result above does not exclude that AVEM, which contains AFEM and allows more flexibility in terms of hanging nodes, could obtain a better efficiency in terms of the involved constants (in this respect, see also Section \ref{sec:experiments}).
}
\end{remark}

 \subsubsection{Approximation classes for data}
%

Given a partition $\mesh$ and piecewise constant data $\widehat{\data}=(\widehat{A}_\mesh,\widehat{c}_\mesh,\widehat{f}_\mesh)$ defined as in \eqref{eq:averages},  let us set (cf.  \eqref{eq:zeta-Acf}) 
\begin{equation}
\zeta_\mesh(A)=\| A-\widehat{A}_\mesh \|_{L^\infty(\Omega)}\,, \qquad 
 \zeta_\mesh(c)=\| \hh (c-\widehat{c}_\mesh) \|_{L^\infty(\Omega)}\,, \qquad  \zeta_\mesh(f)=\| \hh (f-\widehat{f}_\mesh) \|_{L^2(\Omega)}.
\end{equation}

\begin{definition}[approximation classes of $A$] Let 
\begin{equation}
\ACA_{s}=\{ A\in (L^\infty(\Omega))^{2\times 2}\ : \ \exists C \in {\mathbb R} \textrm{ s.t. } \inf_{\mesh\in\mathbb{T}_N} \zeta_{\mesh}(A)\leq C N^{-s} \ \forall 
N \geq \# \mesh_0 \}
\end{equation}
and denote 
\begin{equation}
\vert A \vert_{{\ACA}_{s}}:=
\sup_{N \geq \# \mesh_0} 
\left( N^s\inf _{\mesh\in \mathbb{T}_N} \zeta_\mesh (A) \right).
\end{equation}
\end{definition}

\begin{definition}[approximation classes of $c$] Let 
\begin{equation}
\ACc_{s}=\{ c\in L^\infty(\Omega)\ : \ \exists C \in {\mathbb R} \textrm{ s.t. } \inf_{\mesh\in\mathbb{T}_N} \zeta_{\mesh}(c)\leq C N^{-s} \ \forall 
N \geq \# \mesh_0 \} 
\end{equation}
and denote 
\begin{equation}
\vert c \vert_{{\ACc}_{s}}:=
\sup_{N \geq \# \mesh_0} 
\left( N^s \inf _{\mesh\in \mathbb{T}_N} \zeta_\mesh (c) \right).
\end{equation}
\end{definition}
\begin{definition}[approximation classes of $f$] Let 
\begin{equation}
\ACf_{s}=\{ f\in L^2(\Omega)\ : \ \exists C \in {\mathbb R} \textrm{ s.t. }  \inf_{\mesh\in\mathbb{T}_N} \zeta_{\mesh}(f)\leq C N^{-s} \ \
N \geq \# \mesh_0 \} 
\end{equation}
and denote 
\begin{equation}
\vert f \vert_{{\ACf}_{s}}:=
\sup_{N \geq \# \mesh_0} 
\left( N^s \inf _{\mesh\in \mathbb{T}_N} \zeta_\mesh (f) \right).
\end{equation}
\end{definition}

In the rest of the paper, we make the following assumptions concerning the data of our problem and their piecewise-linear approximation.

\begin{assumption}[approximability of data] \label{ass:approx-data} There exist $s_A, s_c, s_f \in  (0, \frac12]$ such that the data of Problem \eqref{eq:pde} satisfy 
$A \in \ACA_{\, s_{\! A}}$, $c \in \ACc_{s_c}$, $f \in \ACf_{s_f}$.
\end{assumption}

\begin{assumption}[quasi-optimality of the module \tt{DATA}] \label{ass:optim-data}
The procedure {\tt{MARK\_DATA}} introduced in Sect. \ref{sec:data} is quasi-optimal, namely the cardinalities of the marked sets ${\cal M}_A, {\cal M}_c, {\cal M}_f$ for $A, c, f$ resp., satisfy
\begin{equation}\label{eq:decay-data}
\# {\cal M}_A \lesssim  \vert A \vert_{\ACA_{\, s_{\! A}}}^{\frac{1}{s_A}} \varepsilon^{-\frac{1}{s_A}} \,, \qquad 
\# {\cal M}_c \lesssim  \vert c \vert_{\mathbb{C}_{s_c}}^{\frac{1}{s_c}} \varepsilon^{-\frac{1}{s_c}} \,, \qquad 
\# {\cal M}_f \lesssim  \vert f \vert_{\mathbb{F}_{s_f}}^{\frac{1}{s_f}} \varepsilon^{-\frac{1}{s_f}} \,.
\end{equation}
\end{assumption}

Under this assumption, setting $s_\data = \min(s_A, s_c, s_f)$, 
the cardinality of the marked set $\mathcal{M}_\data = \mathcal{M}_A \cup \mathcal{M}_c \cup \mathcal{M}_f$ satisfies
\begin{equation}
\# {\cal M}_\data \lesssim  \big(\vert A \vert_{\ACA_{\, s_{\! A}}}^{\frac{1}{s_A}} +
\vert c \vert_{\mathbb{C}_{s_c}}^{\frac{1}{s_c}} +
 \vert f \vert_{\mathbb{F}_{s_f}}^{\frac{1}{s_f}}  \big) \varepsilon^{-\frac{1}{s_\data}} =: \vert \data \vert_{\mathbb{A}_\data}^{\frac{1}{s_\data}} \, \varepsilon^{-\frac{1}{s_\data}} 
\,.
\end{equation}

\begin{remark}{\rm
In Sect. \ref{sec:approx-data} we will give regularity conditions on the data such that Assumption \ref{ass:approx-data} is satisfied. In particular, we will prove that $s_A=s_c=s_f = \frac12$ if $A \in (W^1_p(\Omega))^{2 \times 2}$ with $p>1$, $c\in L^\infty(\Omega)$ and $f \in L^2(\Omega)$. Furthermore, we will show that the implementation of  {\tt{MARK\_DATA}} described in Sect. \ref{sec:data} guarantees the validity of Assumption \ref{ass:optim-data}.
}
\end{remark}

\subsection{$\varepsilon$-approximation of order $s$}\label{sec: eps-approx}
Since the data $\widehat{\mathcal{D}}_k$ is fixed inside {\tt GALERKIN}, the performance of this module is dictated by the regularity of $\widehat{u}_{k}=\mathsf{exact.sol}(\widehat{\mathcal{D}}_k)$, which is the exact solution with data $\widehat{\mathcal{D}}_k$, rather than $u$.  We know that $u\in {\ACu}_s$ and wonder what regularity is inherited by $\widehat{u}_{k}$. This leads to the following concept introduced in \cite[Def.  3.1 and Lemma 3.2]{BonitoDeVoreNochetto}.
\begin{definition}[$\varepsilon$-approximation of order $s$] Given $u\in {\ACu}_s$ and $\varepsilon>0$,  a function $v\in H^1_0(\Omega)$
is said to be an $\varepsilon$-approximation of order $s$ to $u$ if $\vvvert u - v\vvvert\leq \varepsilon$ and there exists a constant $C>0$ independent of $\varepsilon$, $u$ and $v$ such that for all $\delta\geq \varepsilon$ there exists $N \geq \# \mesh_0$ satisfying
$$
\sigma_N(v) \leq \delta\qquad N\leq C \vert u \vert_{{\ACu}_s}^{\frac{1}{s}}\delta^{-\frac{1}{s}}+1.
$$ 
\end{definition}
\begin{remark}{\rm
In view of the definition of $\sigma_N(v)$, there exists $\mathcal{T}\in \mathbb{T}_N$ and $v_\mathcal{T}\in \mathbb{V}_\mathcal{T}$ such that 
$$
\sigma_N(v)=\qnt_\mesh(v,v_\mathcal{T})\leq \delta. 
$$
}
\end{remark}
\begin{lemma}[$\varepsilon$-approximation of $u$ of order $s$]
Let $u\in \mathcal{A}_s$ and $v\in H^1_0(\Omega)$ satisfying $\vvvert u-v\vvvert \leq \varepsilon$ for some $\varepsilon>0$.  Then $v$ is a $2\varepsilon$-approximation of order $s$ to $u$.  
\end{lemma}
\begin{proof}
Let $\delta\geq 2\varepsilon$. By definition of $\sigma_N(u)$, there exists $N \geq \# \mesh_0$, $\mathcal{T}\in \mathbb{T}_N$ and $w_\mathcal{T}\in \mathbb{V}_\mathcal{T}$ such that 
$$
\sigma_N(u)=\qnt_\mesh (u,w_\mathcal{T})\leq \frac\delta4 \qquad 
N\leq \vert u \vert_{{\ACu}_s}^{\frac{1}{s}}\big(\frac\delta4\big)^{-\frac{1}{s}}+1.
$$ 
The triangle and Young inequalities yield 
\begin{equation*}
\begin{split}
\sigma_N(v) & \leq \qnt_\mesh(v,w_\mesh) \leq 
\vvvert v - w_\mathcal{T}\vvvert + \vert w_\mathcal{T} - \mathcal{I}_\mathcal{T} w_\mathcal{T} \vert_{1,\Omega} \\
& \leq 
\vvvert v - u \vvvert + \vvvert u- w_\mathcal{T}\vvvert +  \vert w_\mathcal{T} - \mathcal{I}_\mathcal{T} w_\mathcal{T}\vert_{1,\Omega} \leq \vvvert v - u \vvvert +
\sqrt{2} \, \qnt_\mesh(u, w_\mesh) \\
& \leq \varepsilon+ \sqrt{2} \, \frac\delta2  \leq \left(\frac12 + \frac{\sqrt{2}}2\right)\delta < \delta\,.
\end{split}
\end{equation*}
Moreover, there holds
$$
N\leq 4^{\frac{1}{s}} \vert u \vert_{{\ACu}_s}^{\frac{1}{s}}\delta^{-\frac{1}{s}}+1.
$$
This concludes the proof with constant $C=4^{\frac{1}{s}}$.
\end{proof}
 \subsection{Optimality of mesh refinement}\label{sec:optim-mark}
 Hereafter, we consider two $\Lambda$-admissible partitions $\mesh$ and $\meshs$, the latter being a refinement of the former obtained by applying a newest-vertex bisection to some of the elements of $\mesh$. Considering the corresponding Galerkin solutions $\umesh$ and $\umeshs$ of problem \eqref{def-Galerkin} with piecewise constant data, we first prove that the difference in energy norm between $\umesh$ and the orthogonal projection of $\umeshs$ upon $\Vmeshs^0$ can be essentially bounded by the contribution to the error estimator coming from a neighborhood of the refined elements.  Next, we give conditions under which this portion of the error estimator satisfies a D\"orfler property with respect to the full estimator.

 \subsubsection{Localized upper bound of the difference between Galerkin solutions}
Consider an element $E \in \mesh$ which has been split into two elements $E_1, E_2 \in \meshs$. If $v \in \Vmesh$, then $v$ is known on $\partial E$, hence in particular at the new vertex of $E_1, E_2$ created by bisection. Thus, $v$ is known at all nodes (vertices and possibly hanging nodes) sitting on $\partial E_1$ and $\partial E_2$, since the new edge $e=E_1\cap E_2$ does not contain internal nodes. This uniquely identifies a function in $\V_{E_1}$ and a function in $\V_{E_2}$, which are continuous {across} $e$.
In this manner, we associate to any $v \in \Vmesh$ a unique function $v_* \in \Vmeshs$, that coincides with $v$ on the skeleton $\edge$. We will actually write $v$ for $v_*$ whenever no confusion is possible.

We introduce the following orthogonal decomposition of $\Vmesh$
\begin{equation}\label{orth-decomp}
\Vmesh= \Vmesh^0 \oplus \Vmesh^{\perp} \,,
\end{equation}
where $\Vmesh^{\perp}$ is the orthogonal complement of
$\Vmesh^0$ in $\Vmesh$ with respect to the (discrete) scalar product $\mathcal{B}_\mesh(\cdot,\cdot)$, and we prove a localized estimate (cf. \cite[Lemma 3.5]{BonitoNochetto:10}) that is crucial in the discussion of the quasi-optimal cardinality of our adaptive algorithm.
To this end, we denote by $\mathcal{R}_{\mesh\to\meshs}$ the set of refined elements of $\mesh$ to obtain $\meshs$ and let $\omega(\mathcal{R}_{\mesh\to\meshs})$ be any subset of $\mesh$ containing $\mathcal{R}_{\mesh\to\meshs}$. 
We observe that as $\meshs$ is a refinement of $\mesh$,  Assumption \ref{ass:constant-coeff} implies that for every $E_*\in\meshs$ with $E_*\subseteq E$, $E\in \mesh$ we have $A_{E_{*}}=A_{E}$, $c_{E_{*}}=c_{E}$ and  $f_{E_{*}}=f_{E}$.
The following lemma bounds the difference between a discrete solution and (the $V^0_\meshs$ part of) another discrete solution on a refined mesh. Such difference is bounded by the error estimator evaluated on a suitable neighbourhood of the refined elements, plus an additional term which nevertheless becomes ``negligible'' for $\gamma$ sufficiently large.
 \begin{lemma}[localized upper bound]\label{lemma:loc-up-bound} Let $\meshs$ be a refinement of $\mesh$ and let $\umesh\in \Vmesh$ and $\umeshs\in \Vmeshs$ be the corresponding discrete solutions of \eqref{def-Galerkin} with piecewise constant data.
Let $\umeshs=\umeshs^0+\umeshs^{\perp} \in \Vmeshs^0 \oplus \Vmeshs^{\perp}$ be the orthogonal decomposition of $\umeshs$ according to \eqref{orth-decomp}. Then, there exists a constant $C_{LU}$ only depending on the shape regularity of $\mesh$ so that 
\begin{equation}\label{LU-bound}
\vvvert \umeshs^0-\umesh \vvvert \leq C_{LU}
\left( \etamesh(\omega(\mathcal{R}_{\mesh\to \meshs});\umesh,\data) + \gamma^{-1} \etamesh(\umesh,\data)\right) \,.
\end{equation}
\end{lemma}\label{lm:LU-bound}
 \begin{proof}
Let us preliminarily proceed by steps and collect some instrumental results that will be employed in the sequel.

\noindent{\bf{Step 1}}. First, we observe that as $\Vmesh^0\subset \Vmeshs^0$ is made of continuous piecewise linear functions on $\mesh$ we have 
\begin{equation}\label{P:0}
v^0_\mesh=\Pi^\nabla_\mesh v^0_\mesh= \Pi^\nabla_\meshs v^0_\mesh.
\end{equation}

\noindent{\bf{Step 2}}.  There holds 
\begin{equation}\label{P:1}
\Bmeshs(u_\meshs^0,v_\meshs^0)= {\cal F}_{\meshs} (v_\meshs^0)\qquad \forall v_\meshs^0\in V^0_\meshs.
\end{equation}
Indeed, for any $v_\meshs^0\in  V_\meshs^0$ we have 
\begin{eqnarray}
\mathcal{F}_\meshs(v_\meshs^0)=\Bmeshs(\umeshs,v_\meshs^0)=\Bmeshs(\umeshs^{0}+\umeshs^{\perp},v_\meshs^0)=\Bmeshs(\umeshs^0,v_\meshs^0)
\end{eqnarray}
where in the last step we employed that $\Vmeshs^{\perp}$ is the orthogonal complement of
$\Vmeshs^0$ in $\Vmeshs$ with respect to $\Bmeshs(\cdot,\cdot)$.

\noindent{\bf{Step 3}}.  There holds
 \begin{equation}\label{P:2}
\mathcal{B}(u_\meshs^0-\umesh,v_\mesh^0)=0 \qquad \forall v_\mesh^0\in \Vmesh^0.
\end{equation}
Using \eqref{eq:consistency} and \eqref{P:0} we have
\begin{equation}
\mathcal{B}(u_\meshs^0-\umesh,v_\mesh^0)=
\mathcal{B}_\meshs(u_\meshs^0,v_\mesh^0)-\mathcal{B}_\mesh(\umesh,v_\mesh^0)= {\cal F}_{\meshs} (v_\mesh^0)- {\cal F}_{\mesh} (v_\mesh^0)
\end{equation}
where in the last step we employed \eqref{P:1}. From Assumption \ref{ass:constant-coeff}, \eqref{discr-rhs} and \eqref{P:0} we get \eqref{P:2}.

\noindent{\bf{Step 4}}. Let $e_*^0=u_\meshs^0-\umesh^0-v_\mesh^0$ with $v_\mesh^0\in V_\mesh^0$, where $u_\mesh^0=u_\mesh - u_\mesh^\perp \in \V_\mesh^0$.
There holds
 \begin{equation}\label{P:3}
 \Bmeshs(u_\meshs^0-\umesh,e_*^0)\lesssim  \vert \umeshs^0-\umesh^0\vert_{1,\Omega} (
 \etamesh(\omega(\mathcal{R}_{\mesh\to \meshs});\umesh,\data) + \gamma^{-1} \etamesh(\umesh,\data)).
 \end{equation}
 Indeed, we have 
 \[
 \begin{aligned}
 \Bmeshs(u_\meshs^0-\umesh,e_*^0)&=\mathcal{F}_\meshs (e_*^0)- \Bmeshs(\umesh,e_*^0)
 \\
 &= \left(
 \mathcal{F}_\meshs(e_*^0)-\Bmeshs(\Pi_\mesh^\nabla\umesh,e_*^0) \right) + \Bmeshs(\Pi_\mesh^\nabla\umesh-\umesh,e_*^0)= : I+II \,,
 \end{aligned}
 \]
where, with a slight abuse of notation, we extend the definition \eqref{eq:def-BT} of $\Bmeshs$ to $\mathbb{P}_1(\meshs)$.

 In the sequel we choose $v_\mesh^0=\widetilde{\cal I}_\mesh^0(u_\meshs^0-u_\mesh^0)$ in the definition of $ e_*^0$, where $\widetilde{\cal I}_\mesh^0: C^0(\bar{\Omega}) \to \Vmesh^0$ is the Cl\'ement quasi-interpolation operator on $\mesh^0$. We also notice that 
 $e_*^0$ vanishes outside the set $\omega(\mathcal{R}_{\mesh\to \meshs})$. 
 As $e_*^0\in V_\meshs^0$ and $\Pi^\nabla_\meshs e_*^0= e_*^0 $ we have 
 \[
 \begin{aligned}
 I&= \mathcal{F}_\meshs(e_*^0) -\sum_{E_*\in \meshs}
 \int_{E_*}
 \left(A_{E_*} \nabla \Pi^\nabla_\meshs (\Pi^\nabla_\mesh \umesh) \cdot \nabla \Pi^\nabla_\meshs e_*^0 +  c_{E_*}  \Pi^\nabla_\meshs (\Pi^\nabla_\mesh \umesh)  \Pi^\nabla_\meshs  e_*^0 \right) 
 \\
 &=
 \mathcal{F}_\mesh(e_*^0) -\sum_{E\in\mesh}\sum_{E_*\in \meshs, E_*\subseteq E}
 \int_{E_*}
\left( A_{E_*}  \nabla \Pi^\nabla_\mesh \umesh \cdot \nabla e_*^0 +  c_{E_*}  \Pi^\nabla_\mesh \umesh \, e_*^0  \right)
\\
  &=
 \sum_{E\in \omega(\mathcal{R}_{\mesh\to \meshs})}
 \int_{E}
 \left(f_E e_*^0 -
 A_{E}  \nabla \Pi^\nabla_\mesh \umesh \cdot \nabla e_*^0 -  
 c_{E}  \Pi^\nabla_\mesh \umesh \, e_*^0 \right)
 \end{aligned}
 \]
 where we employed the properties of the enhanced space 
 \eqref{vem:choice:2}. 
Integrating by parts, employing the Cauchy-Schwarz inequality together with the vanishing property of $e_*^0$ and the interpolation error estimate for ${\cal I}_\mesh^0$, we obtain
 \begin{equation}\label{P:3:step1}
 I\lesssim \vert \umeshs^0-\umesh^0\vert_{1,\Omega} \,
 \etamesh(\omega(\mathcal{R}_{\mesh\to \meshs});\umesh,\data).
 \end{equation}
 On the other hand, again  as $e_*^0\in V_\meshs^0$ and $\Pi^\nabla_\meshs e_*^0= e_*^0 $, we have
 \begin{equation}
 \label{P:3:step2}
 \begin{aligned}
 II&=\sum_{E_*\in \meshs}
 \int_{E_*}
 \left(A_{E_*} \nabla \Pi^\nabla_\meshs (\Pi^\nabla_\mesh-I) \umesh \cdot \nabla \Pi^\nabla_\meshs e_*^0 +  c_{E_*} \Pi^\nabla_\meshs (\Pi^\nabla_\mesh-I) \umesh  \, \Pi^\nabla_\meshs  e_*^0 \right)
 \\
 &=\sum_{E_*\in \meshs}
 \int_{E_*}
\left( A_{E_*} \nabla  (\Pi^\nabla_\mesh-I) \umesh \cdot \nabla e_*^0 +  c_{E_*} (\Pi^\nabla_\mesh-I) \umesh  \,  e_*^0 \right) 
\\
 &= \sum_{E\in \mesh}\sum_{E_*\in \meshs, E_*\subseteq E} \int_{E_*}
\left( A_{E_*} \nabla  (\Pi^\nabla_\mesh-I) \umesh \cdot \nabla e_*^0 + c_{E_*}  (\Pi^\nabla_\mesh-I) \umesh  \,  e_*^0 \right)
\\
  &= \sum_{E\in \mesh} \int_{E}
 \left( A_{E} \nabla  (\Pi^\nabla_\mesh-I) \umesh \cdot \nabla e_*^0 + c_{E}  (\Pi^\nabla_\mesh-I) \umesh  \,  e_*^0 \right)
 \\
&\lesssim  \Smesh(\umesh,\umesh)^{1/2}  \vert \umeshs^0-\umesh^0\vert_{1,\Omega}\lesssim \gamma^{-1} \etamesh(\umesh,\data) \vert \umeshs^0-\umesh^0\vert_{1,\Omega}  \,.
 \end{aligned}
 \end{equation}
 where in the last step we used \eqref{eq:bound-ST}. 
 The thesis follows combining \eqref{P:3:step1}-\eqref{P:3:step2}.
 
 \noindent{\bf{Step 5}}.  Let $\umesh=\umesh^0+\umesh^{\perp}$ be the orthogonal decomposition \eqref{orth-decomp}. There holds
 \begin{equation}\label{P:4}
 \vvvert u_\mesh^\perp\vvvert \lesssim \Smesh(\umesh,\umesh)^{1/2} \lesssim \gamma^{-1} \etamesh(\umesh,\data) \,.
 \end{equation}
 Indeed, we have 
 \begin{equation*}
 \begin{aligned}
 \Bmesh(u_\mesh^\perp,u_\mesh^\perp)&=\inf_{w_\mesh^0\in V_\mesh^0} \Bmesh(u_\mesh-w_\mesh^0,u_\mesh-w_\mesh^0)\leq \Bmesh(u_\mesh-I_\mesh^0 u_\mesh,u_\mesh-I_\mesh^0 u_\mesh)
 \lesssim  \Smesh(u_\mesh,u_\mesh)
 \end{aligned}
 \end{equation*}
 where in the last inequality we employed the continuity of $\Bmesh(\cdot,\cdot)$ in combination with 
 \cite[Prop. 3.2]{paperA}. 
 The coercivity of  $\Bmesh(\cdot,\cdot)$ together with \eqref{eq:stab-norm} and \eqref{norm:equiv}, and the bound \eqref{eq:bound-ST} yield the result.

\smallskip
At this point, we have collected all ingredients to prove \eqref{LU-bound}.
From the coercivity of $\mathcal{B}(\cdot,\cdot)$ and employing \eqref{P:2} we get
\begin{equation}
\label{Step:1}
\begin{aligned}
\vvvert u_\meshs^0-\umesh\vvvert^2 &\lesssim 
\mathcal{B}(u_\meshs^0-\umesh,u_\meshs^0-\umesh)=
\mathcal{B}(u_\meshs^0-\umesh,e_*^0)+
\mathcal{B}(u_\meshs^0-\umesh,v_\mesh^0-u_\mesh^\perp)\nonumber\\
&=\mathcal{B}(u_\meshs^0-\umesh,e_*^0)-\mathcal{B}(u_\meshs^0-\umesh,u_\mesh^\perp)=: III + IV.
\end{aligned}
\end{equation}
Employing the consistency of $\Bmeshs(\cdot,\cdot)$ (cf. \eqref{eq:consistency}) together with \eqref{P:3} we get
\begin{equation}\label{Step:2}
III=\Bmeshs(u_\meshs^0-\umesh,e_*^0)\lesssim 
\vert \umeshs^0-\umesh^0 \vert_{1,\Omega} (
 \etamesh(\omega(\mathcal{R}_{\mesh\to \meshs});\umesh,\data) + \gamma^{-1}\etamesh(\umesh,\data)).
\end{equation}
On the other hand, employing the continuity of $\mathcal{B}(\cdot,\cdot)$ in combination with \eqref{P:4} we obtain
\begin{equation}\label{Step:3}
IV \lesssim \vvvert u_\meshs^0-\umesh\vvvert \vvvert u_\mesh^\perp \vvvert \lesssim  \vvvert u_\meshs^0-\umesh\vvvert \, \gamma^{-1} \etamesh(\umesh,\data) \,.
\end{equation}
We now observe that $\vert \umeshs^0-\umesh^0 \vert_{1,\Omega} = \vert \umeshs^0-\umesh + u_\mesh^\perp \vert_{1,\Omega} \lesssim \vvvert u_\meshs^0-\umesh\vvvert + \gamma^{-1} \etamesh(\umesh,\data))$. 
Concatenating \eqref{Step:1}-\eqref{Step:3}, we easily conclude the proof of Lemma \ref{lemma:loc-up-bound}.
 \end{proof}
 
 \subsubsection{Optimal marking}
 We first recall two instrumental results that will be useful in the sequel. From \cite[Corollary 4.3]{paperA} we have the global error bound
\begin{equation}\label{global:lower:bound}
C_{GL} \eta^2_\mesh(\umesh,\data)\leq \vvvert u -\umesh\vvvert^2 + \vert \umesh - I_\mesh \umesh\vert_{1,\mesh}^2 = \qnt_\mesh^2(u,u_\mesh)\,.
\end{equation} 
Moreover, we observe that \eqref{eq:stab-norm} and \eqref{eq:bound-ST} yield
\begin{equation}\label{aux:3:OM}
\vert \umesh -I_\mesh \umesh\vert_{1,\mesh}^2 \leq \tilde{C}_B \gamma^{-2} \eta_\mesh^2(\umesh,\data).
\end{equation}
 In order to derive a quasi-optimal decay of the total error, we define
 $$
 \gamma^2_*:=\frac{2C_{LU} + \tilde{C}_B}{C_{GL}} \qquad \theta_*(\gamma):=\frac{C_{GL}-\gamma^{-2}(2C_{LU}+\tilde{C}_B)}{2C_{LU}}
 $$ 
 for $\gamma>\gamma_*$, where $C_{LU}$ is given by Lemma \ref{lemma:loc-up-bound}. Notice that $\gamma>\gamma_*$ yields  $\theta_*>0$ and if $C_{GL}<2C_{LU}$ then $\theta_*<1$. Moreover we make the following assumption.
\begin{assumption}[module \texttt{MARK}]
The set of marked elements produced by the module \texttt{MARK} has minimal cardinality and the marking parameter satisfies $\theta\in (0,\theta_*)$. 
\end{assumption}
In order to simplify the notation, we let $0<\mu<1/2$ be defined by 
 $$
 \mu(\gamma,\theta):=\frac{C_{GL}-\gamma^{-2}(2C_{LU}+\tilde{C}_B)}{2C_{GL}}(1-\frac{\theta}{\theta_*}) \qquad \forall\gamma>\gamma_*, \quad 0<\theta<\theta_*.
 $$
 We now prove the analogous of \cite[Lemma 5.4]{BonitoNochetto:10}.
 \begin{lemma}[optimal marking]\label{lm:optimal_marking}Let $\meshs$ be a refinement of $\mesh$ and and  let $\umesh\in V_\mesh$ and $\umeshs\in V_\meshs$ the corresponding discrete solutions of \eqref{def-Galerkin}. In addition, assume
 \begin{equation}\label{assumption:mu}
 \qnt_\mesh^2(u,u^0_\meshs) \leq \mu \, \qnt_\mesh^2(u,u_\mesh)
 \end{equation}
 where $\umeshs=\umeshs^0+\umeshs^{\perp}$ is the orthogonal decomposition \eqref{orth-decomp}. Then, for $\gamma>\gamma_*$ and $\theta\in (0,\theta_*(\gamma))$, the set $\omega(\mathcal{R}_{\mesh\to\meshs})$ satisfies a D\"orfler marking property
 $$ \eta^2_\mesh(\omega(\mathcal{R}_{\mesh\to\meshs});\umesh,\data)\geq \theta \eta^2_\mesh(\umesh,\data).$$
 \end{lemma}
 \begin{proof}
Since $0<\mu<1/2$, employing \eqref{global:lower:bound} and \eqref{assumption:mu} we get 
\begin{equation}\label{aux:1:OM}
(1-2\mu)C_{GL} \eta^2_\mesh(\umesh,\data) \leq 
\vvvert u - \umesh\vvvert^2 - 2 \vvvert u - u_\meshs^0\vvvert^2 + \vert \umesh -I_\mesh \umesh\vert_{1,\mesh}^2
\end{equation}
where we used $I_\meshs u_{\meshs}^0=u_{\meshs}^0$.
From the triangle inequality and \eqref{LU-bound} we obtain 
\begin{equation}\label{aux:2:OM}
\vvvert u - \umesh\vvvert^2 - 2 \vvvert u - u_\meshs^0\vvvert^2 \leq 2  \vvvert \umesh - u_\meshs^0\vvvert^2 \leq 2C_{LU}
\left( \etamesh^2(\omega(\mathcal{R}_{\mesh\to \meshs});\umesh,\data) + \gamma^{-2} \etamesh^2(\umesh,\data)\right).
\end{equation}
Combining \eqref{aux:1:OM}-\eqref{aux:3:OM} we get
\begin{equation*}
(1-2\mu) C_{GL} \eta_\mesh^2(\umesh,\data) \leq 2 C_{LU}
\etamesh^2(\omega(\mathcal{R}_{\mesh\to \meshs});\umesh,\data) + \gamma^{-2}(2C_{LU}+\tilde{C}_B) \eta_\mesh^2(\umesh,\data)
\end{equation*}
which implies, employing the definition of $\mu$ and $\theta_*$, the desired estimate
$$
\etamesh^2(\omega(\mathcal{R}_{\mesh\to \meshs});\umesh,\data) \geq
\frac{1}{2C_{LU}}\left( (1-2\mu)C_{GL} - \gamma^{-2}(2C_{LU}+ \tilde{C}_B)\right) \eta^2_\mesh(\umesh,\data)\geq
\theta  \eta^2_\mesh(\umesh,\data).
$$ 
The proof is concluded.
\end{proof}
 
\subsection{Complexity of {\tt GALERKIN}}\label{sec:compl-gal}
In this section, we rely on the notation introduced in Sect. \ref{sec:complex-Gal}, in particular those in \eqref{eq:def-many-functions}.
We assume that the pair $(\widehat{\mesh}_k,\widehat{\data}_k)$ transferred by {\tt DATA} to {\tt GALERKIN} at iteration $k$ satisfies 
$$
\eta_{\widehat{\mathcal{T}}_k}(u_{k,0},\widehat{\data}_k)=:\widehat{\varepsilon}_k>\varepsilon_k\, ,
$$
for otherwise {\tt GALERKIN} is skipped. 
On the other hand, combining \eqref{norm:equiv} with the stabilization free a posteriori error estimates 
\eqref{apost:stab-free},  we can write
\begin{equation}
\widehat{C}_L^2 \, \eta_{\widehat{\mathcal{T}}_k}^2(u_{k,0},\widehat{\data}_k)   \leq    \vvvert \widehat{u}_k - u_{k,0}  \vvvert^2  \leq \widehat{C}_U^2 \, \eta_{\widehat{\mathcal{T}}_k}^2(u_{k,0},\widehat{\data}_k)
\end{equation} 
with ${\widehat{C}}_L^2:=c_\mathcal{B} C_L$ and 
${\widehat{C}}_U^2:=c^\mathcal{B} C_U$.
Therefore, we get the lower bound 
$$
\vvvert \widehat{u}_k - u_{k,0}   \vvvert \geq 
\widehat{C}_L \widehat{\varepsilon}_k>\widehat{C}_L \varepsilon_k.
$$
On the other hand,   from \eqref{pert_estimate:2} and  \eqref{data:approximation} it follows that  {\tt DATA} provides a perturbed  exact solution $\widehat{u}_k\in H^1_0(\Omega)$ satisfying 
$$
\vvvert u- \widehat{u}_k \vvvert \leq D \omega \varepsilon_k = \frac{D\omega}{\widehat{C}_L}{\widehat{C}_L} \varepsilon_k
$$
for a suitable constant $D>0$.  Let 
$$
\omega:= \frac{\sqrt{\mu} \widehat{C}_L}{2D} \,,
$$
which implies 
$$
\vvvert u- \widehat{u}_k\ \vvvert \leq
\frac{\sqrt{\mu}}{2} \widehat{C}_L \varepsilon_k.
$$
In view of Proposition \ref{prop:compl_gal} (computational cost of {\tt GALERKIN}) the module {\tt GALERKIN} performs a number of iterations $J_k$ bounded uniformly in $k$ by $J$.  For each such iteration $j$ we have a mesh $\mathcal{T}_{k,j}$ and a Galerkin solution $u_{k,j}\in \mathbb{V}_{\mathcal{T}_{k,j}}$ so that for $0\leq j<J_k$
\begin{eqnarray}
&&\mathcal{T}_{k,0}=\widehat{\mathcal{T}}_k \,, \nonumber\\
&&  \eta_{\mathcal{T}_{k,j}}(u_{k,j},\widehat{\data}_k)>\varepsilon_k \,, \nonumber\\
&& \qnt_{\mathcal{T}_{k,j}}(\widehat{u}_k,u_{k,j})\geq 
\vvvert \widehat{u}_k - u_{k,j}\vvvert \geq \widehat{C}_L 
\eta_{\mathcal{T}_{k,j}}(u_{k,j},\widehat{\data}_k)>\widehat{C}_L \varepsilon_k\,.\nonumber
\end{eqnarray}
Let $\mathcal{M}_{k,j}$ be the marked set within $\mathcal{T}_{k,j}$ using the D\"orfler strategy.
\begin{lemma}[cardinality of marked sets]
If $u\in{\ACu}_s$ and $\omega=\frac{\sqrt{\mu}\widehat{C}_L}{2D}$, then there exists a constant $C_0>0$ such that 
$$
\#\mathcal{M}_{k,j}\leq C_0 \vert u\vert_{{\ACu}_s}^{\frac{1}{s}}\varepsilon^{-\frac{1}{s}}\,, \qquad  0\leq j <J_k.
$$
\end{lemma}
\begin{proof}
Fix $0\leq j < J_k$ and set 
$$
\delta:= \sqrt{\mu}\qnt_{\mathcal{T}_{k,j}}(\widehat{u}_k,u_{k,j})=\sqrt{\mu}
\left(
\vvvert \widehat{u}_k  - u_{k,j}\vvvert +
\vert u_{k,j}- \mathcal{I}_{\mathcal{T}_{k,j} u_{k,j}}\vert_{1,\Omega}
\right)
$$
whence
$$
\delta \geq \sqrt{\mu}\widehat{C}_L\varepsilon_k. 
$$
Since $\vvvert u-\widehat{u}_k \vvvert \leq \frac{\sqrt{\mu}}{2} \widehat{C}_L \varepsilon_k$,
we deduce that $\widehat{u}_k $ is an
$\sqrt{\mu} \widehat{C}_L \varepsilon_k$-approximation of order $s$ to $u$.  Therefore, there exist an admissible mesh $\mathcal{T}_\delta$ such that 
$$ 
\qnt_{\mathcal{T}_{\delta}}(\widehat{u}_k,
u_{\mathcal{T}_\delta}^0)\leq \delta \qquad 
\#\mathcal{T}_\delta \lesssim \vert u \vert_{{\ACu}_s}^{\frac{1}{s}}\delta^{-\frac{1}{s}}
$$
where $u_{\mathcal{T}_\delta}^0\in \mathbb{V}_{\mathcal{T}_\delta}^0$ because ${\ACu}_s={\ACu}_s^0$.  This implies 
$$
\vvvert \widehat{u}_k -  u_{\mathcal{T}_\delta}^0 \vvvert= 
 \qnt_{\mesh_\delta}(\widehat{u}_k, u_{\mathcal{T}_\delta})\leq \delta.
$$
In order to compare with $\mathcal{T}_{k,j}$ we consider the overlay 
$\mathcal{T}_*=\mathcal{T}_{k,j}\oplus\mathcal{T}_\delta$, which satisfies
$$
\# \mathcal{T}_*\leq \mathcal{T}_{k,j} + \# \mathcal{T}_\delta - \#\mathcal{T}_0.
$$
Consider now $u_{\mathcal{T}_*}^0\in \mathbb{V}_{\mathcal{T}_*}^0$,  the Galerkin solution on the subspace of continuous piecewise linears  $\mathbb{V}_{\mathcal{T}_*}^0$.  Exploiting the monotonicity 
$$ 
\vvvert \widehat{u}_k - u_{\mathcal{T}_*}^0 \vvvert \leq \vvvert  \widehat{u}_k - u_{\mathcal{T}_\delta}^0 \vvvert \,,
$$ 
because $\mathcal{T}_*$ is a refinement of $\mathcal{T}_\delta$, we see that 
\begin{equation}
\qnt_{\mathcal{T}_*}( \widehat{u}_k,u_{\mathcal{T}_*}^0)= \vvvert  \widehat{u}_k - u_{\mathcal{T}_*}^0 \vvvert 
\leq 
\vvvert  \widehat{u}_k - u_{\mathcal{T}_\delta}^0 \vvvert\leq \delta = \sqrt{\mu}\qnt_{\mathcal{T}_{k,j}}(\widehat{u}_k,u_{k,j})\,.
\end{equation}
Applying Lemma \ref{lm:optimal_marking} (optimal marking) to $\mathcal{T}_*$ and $\mathcal{T}_{k,j}$ we infer that the refined set $R_{k,j}=R_{\mathcal{T}_{k,j}\to\mathcal{T}_*}$ satisfies D\"orfler marking with parameter $0<\theta<\theta^*$ and stabilization constant $\gamma>\gamma_*$.  In addition, 
$$
\#R_{k,j}=\#\mathcal{T}_*-\#\mathcal{T}_{k,j}.
$$
Since our D\"orfler marking involves a minimal set 
$\mathcal{M}_{k,j}$,  we deduce  
$$
\#\mathcal{M}_{k,j}\leq \# R_{k,j} \leq
\# \mathcal{T}_\delta - \#\mathcal{T}_0
\lesssim \vert u \vert_{{\ACu}_s}^{\frac{1}{s}}\delta^{-\frac{1}{s}}\lesssim
 \vert u \vert_{{\ACu}_s}^{\frac{1}{s}}\varepsilon_k^{-\frac{1}{s}}.
$$
This concludes the proof.
\end{proof}
\begin{corollary}[complexity of {\tt GALERKIN}] If $u\in {\ACu}_s$ and $\omega=\frac{\sqrt{\mu}\widehat{C}_L}{2D}$,  the number of marked elements $\mathcal{M}_k$ within a call to {\tt GALERKIN} satisfies
$$ 
\#\mathcal{M}_k\leq J C_0 \vert u \vert_{{\ACu}_s}^{\frac{1}{s}}\varepsilon_k^{-\frac{1}{s}}.
$$ 
\end{corollary}
\begin{proof}
Use that $\# \mathcal{M}_k=\sum_{j=0}^{J_k-1} \#\mathcal{M}_{k,j}$ and the previous lemma.
\end{proof}

\subsection{Quasi-optimality of {\tt AVEM}}\label{subsec:optimality-AVEM}
We finally address the quasi-optimality 
of the $2$-loop method {\tt AVEM}, by proving the announced bound \eqref{eq:optimality bound}.

\begin{theorem}[quasi-optimality of {\tt AVEM}]\label{Texact:optimality-AVEM}
  Let Assumptions \ref{ass:approx-u}, \ref{ass:approx-data}, and \ref{ass:optim-data} hold true.
Then,  there exist constants $\theta_*, \omega_* < 1$ and $\gamma_*\ge 1$
  such that for all $\theta < \theta_*$, $\omega < \omega_*$, and $\gamma\ge\gamma_*$
  there holds
  \begin{equation*}
   \vvvert u -u_k \vvvert \leq C(u,\data) \big(\#\mesh_k\big)^{-s}  \quad 1\leq k \leq K+1,
  \end{equation*}
  where $0 < s = \min\{s_u, s_\data\} = \min\{s_u, s_A, s_c, s_f\} \leq\frac12$.
  \end{theorem}
\begin{proof}
We know that the number of marked elements $N_k(u)$ within {\tt GALERKIN} satisfies 
$$
N_k(u)\lesssim \vert u \vert^{\frac{1}{s_u}}_{{\ACu}_{s_u}}\varepsilon_k^{-\frac{1}{s_u}}
$$ 
with $s_u\leq \frac{1}{2}$.  Moreover,  by Assumption  \ref{ass:optim-data} the number of marked elements $N_k(\data)$ within {{\tt DATA}} satisfies 
$$
N_k(\data)\lesssim \vert \data \vert_{\mathbb{A}_{s_\data}}^{\frac{1}{s_\data}} \, \varepsilon_k^{-\frac{1}{s_\data}} 
$$ 
with $s_\data \leq \frac12$.  Upon termination,  {\tt DATA} and {\tt GALERKIN} give 
\begin{equation*}
\begin{split}
\vvvert u-\widehat{u}_k \vvvert  & \leq D \omega \varepsilon_k = D\frac{\sqrt{\mu}\widehat{C}_L}{2D}\varepsilon_k < \widehat{C}_U \varepsilon_k \,, \\[5pt]
\vvvert \widehat{u}_k-u_{k+1}\vvvert
& \leq \widehat{C}_U \eta_{\mathcal{T}_{k+1}}(u_{k+1},\data_k)\leq \widehat{C}_U \varepsilon_k \,,
\end{split}
\end{equation*}
because $\mu<1$.  This implies by triangle inequality 
\begin{equation}\label{aux:1:opt-vem}
\vvvert u - u_{k+1} \vvvert \leq 2 \widehat{C}_U \varepsilon_k.
\end{equation}
In addition, the total number of marked elements in the $j$-th loop of {\tt AVEM}
is
$$
N_j(\data)+N_j(u)\leq C_1(\vert u \vert_{{\ACu}_{s_u}}^{\frac{1}{s_u}} +\vert \data \vert_{\mathbb{A}_{s_\data}}^{\frac{1}{s_\data}}
)\, \varepsilon_j^{-\frac{1}{s}} \,.
$$
Therefore,  the total amount of elements created by $k$ loops of {\tt AVEM}, besides those in $\mathcal{T}_0$,  obey the expression
\begin{eqnarray*}
\#\mathcal{T}_{k+1}-\#\mathcal{T}_{0}
\leq C_0 \sum_{j=0}^{k-1} \big( N_j(\data) + N_j(u) \big) \leq C_0 C_1 
(\vert u \vert_{{\ACu}_{s_u}}^{\frac{1}{s_u}} +\vert \data \vert_{\mathbb{A}_{s_\data}}^{\frac{1}{s_\data}}
)\sum_{j=0}^{k-1}\varepsilon_j^{-\frac{1}{s}}.
\end{eqnarray*}
Since $\varepsilon_0=1$,  $\varepsilon_j=2^{-j}$ and
$$
\sum_{j=0}^{k-1}(2^{-\frac{1}{s}})^j \leq \frac{1}{1-2^{-1/s}}
$$
we deduce
\begin{equation}\label{aux:2:opt-vem}
\#\mathcal{T}_{k+1}-\#\mathcal{T}_0 \leq C
(\vert u \vert_{{\ACu}_{s_u}}^{\frac{1}{s_u}} +\vert \data \vert_{\mathbb{A}_{s_\data}}^{\frac{1}{s_\data}}
)\, \varepsilon_k^{-\frac{1}{s}}
\end{equation}
with $C=\frac{C_0C_1}{1-2^{-1/s}}$.  Since the first refined mesh satisfies $\#\mathcal{T}_1 \geq c_0 \, \#\mathcal{T}_0$ for some $c_0>1$, it holds $\#\mathcal{T}_{k+1} \leq \frac{c_0}{c_0-1}( \#\mathcal{T}_{k+1} -\#\mathcal{T}_0)$. Combining this with \eqref{aux:2:opt-vem} and \eqref{aux:1:opt-vem}  yields the thesis.
\end{proof}

\begin{remark} {\rm
The thresholds $\theta_*, \omega_*$ play no role in Proposition \ref{P:convergence-AVEM} but are critical
in Theorem \ref{Texact:optimality-AVEM}. The former takes care of the gap between $C_L$ and $C_U$ in the a posteriori
bounds \eqref{apost:stab-free}, and is well documented in the optimality analysis of AFEMs
\cite{BonitoNochetto:10,NochettoVeeser:12,NSV:09,Stevenson2007}. The latter guarantees that the perturbation
error \eqref{eq:pert-estimate} is much smaller than $\varepsilon_k$ and enables {\tt GALERKIN} to learn the
regularity of $u$ from $\widehat{u}_{\widehat{\mesh}_k}$ \cite{BonitoDeVoreNochetto,Stevenson2007}.
}
\end{remark}

\section{Data approximation: cardinality properties}\label{sec:approx-data}
In this section, we provide sufficient regularity conditions for data $\data=(A,c,f)$ to belong to the approximation classes introduced earlier and present algorithms for their approximation.

\subsection{Greedy algorithm: definition and performance}
We start with a constructive approximation estimate for a generic function $g: \Omega \to \mathbb{R}$ of class $W^s_p(\Omega)$ and next apply it to $\data$.

Let $1 \leq p,q \leq \infty$, $0 \leq s \leq 1$ be so that
$$
\textsf{sob}(W^s_p(\Omega))=s-\frac{2}{p} \ \geq \ \textsf{sob}({L^q}(\Omega))=0-\frac{2}{q}, 
$$
whence
\begin{equation}\label{eq:sob}
s-\frac2p+\frac2q \geq 0\,.
\end{equation}
Let $E \in \mesh$ be a generic element, and let 
$$
g_E :=\frac1{|E|} \int_E g
$$
denote the mean value of $g$ on $E$. Polynomial approximation theory yields
\begin{equation*}
\Vert g-g_E \Vert_{L^q(E)} \lesssim h_E^{s-2/p+2/q} \vert g \vert_{W^s_p(E)}\,.
\end{equation*}
In view of the application to $\data$, it is convenient to consider the weighted $L^q(E)$-norm instead, namely for $0 \leq t \leq 1$
\begin{equation}\label{eq:error-estimate-g}
\zeta_\mesh(E;g) :=h_E^t \Vert g-g_E \Vert_{L^q(E)} \lesssim h_E^r \vert g \vert_{W^s_p(E)}\,, \qquad \text{with \ } r:=t+s-\frac2p+\frac2q \,.
\end{equation}

Given a tolerance $\delta >0$, we consider the algorithm
\begin{algotab}
  \>  $[\mesh] = \texttt{GREEDY}(\mesh, \delta)$

  \\
  \>  $\quad \text{while }$ 
  $\mathcal{M}=\{E\in\mesh:\,  \zeta_{\mesh}(E; g)>\delta \}\not=\emptyset$
  \\
  \>  $\quad \quad  \mesh={\tt REFINE}(\mesh, \mathcal{M}) $ 
   \\
  \>  $\quad  \text{end while }$
  \\
  \>  $\quad  \text{return}(\mesh)$
\end{algotab}
\medskip
The following properties are valid for the global weighted error
$$
\zeta_\mesh(g) := \left( \sum_{E \in \mesh} \zeta_\mesh^q(E;g) \right)^\frac1q \, ,
$$
with the usual interpretation $\zeta_\mesh(g) := \max_{E \in \mesh} \zeta_\mesh(E;g)$ for $q=\infty$. 

\begin{proposition}[performance of $\texttt{GREEDY}$]\label{prop:greedy}
If $r>0$, then $\tt{GREEDY}$ terminates in a finite number of steps. The output partition $\mesh$ satisfies the estimates
\begin{align}
\zeta_\mesh(g) & \leq \delta \, (\#\mesh)^{\frac1q} \,, \label{eq:greedy-A} \\
\delta &\lesssim \vert g\vert_{W^s_p(\Omega)} (\#\mesh - \#\mesh_0)^{-\frac{1}{q}-\frac{t+s}2}\,. \label{eq:greedy-B}
\end{align}
\end{proposition}
\begin{remark}[error decay in $\texttt{GREEDY}$]\label{rem:greedy}
 {\rm
Assuming $\#\mesh \geq c_0 \#\mesh_0$ for some $c_0>1$, and concatenating \eqref{eq:greedy-A} and \eqref{eq:greedy-B} yields
\begin{equation}\label{eq:error-decay-greedy}
\zeta_\mesh(g) \ \lesssim \ \vert g\vert_{W^s_p(\Omega)} (\#\mesh)^{-\frac{t+s}2}\,. 
\end{equation}
}
\end{remark}

\bigskip
\noindent {\em Proof of Proposition \ref{prop:greedy}}. We proceed in several steps.
\begin{itemize}
\item[(i)] {\em Termination}. Since $r>0$,  {\tt GREEDY} stops in finite steps $k$, producing $k$ subsequent refinements $\mesh_1, \dots, \mesh_k$ of $\mesh$.  Upon termination, it holds $\zeta_{\mesh_k}(E;g) \leq \delta$ for all $E \in \mesh_k$, whence \eqref{eq:greedy-A} follows.

\item[(ii)] {\em Counting}. Let
$\mathcal{M}=\mathcal{M}_0\cup\ldots\cup\mathcal{M}_{k-1}$ be the set of marked elements.  We reorganize $\mathcal{M}$ by size: let $\mathcal{P}_j$ be the set of elements $E\in\mathcal{M}$ such that 
$$ 
2^{-(j+1)}\leq \vert E\vert < 2^{-j} \,, \qquad \text{namely} \qquad  2^{-\frac{j+1}{2•}}\leq h_E < 2^{-\frac{j}{2}}. 
$$ 
Since $\tt{REFINE}$ uses bisection, the elements of $\mathcal{P}_j$ are disjoint, whence 
$$
2^{-(j+1)} \#\mathcal{P}_j \leq \vert \Omega\vert \qquad \text{i.e., } \qquad \#\mathcal{P}_j\leq \vert \Omega \, \vert 2^{j+1}.
$$
On the other hand,  $E\in \mathcal{P}_j$ (with $E \in \mesh_i$ for some $i$) implies
$$ 
\delta < \zeta_{\mesh_i}(E; g) \lesssim h_E^r \vert g\vert_{W^s_p(E)}\leq 2^{-\frac{jr}{2}}\vert g\vert_{W^s_p(E)}\,.
$$
In view of the summability of the right-hand side, we now accumulate these inequalities in the $\ell^p$ norm
$$
\delta^p \#\mathcal{P}_j 
\lesssim 2^{-\frac{j rp}{2}}\vert  g\vert_{W^s_p(\Omega)}^p.
$$
This gives an alternative bound 
$$
\#\mathcal{P}_j\lesssim \delta^{-p} 2^{-\frac{jrp}{2}}\vert g\vert^p_{W^s_p(\Omega)}.
$$
\item[(iii)] {\em Summing up}. Adding over $j$ we obtain 
$$
\#\mathcal{M}=\sum_j \#\mathcal{P}_j\lesssim \sum_{j\leq j_0} \vert \Omega\vert \, 2^{j+1} + \sum_{j>j_0} \delta^{-p} 2^{-\frac{jrp}{2}} \vert g\vert_{W^s_p(\Omega)}^p \,,
$$
where $j_0$ corresponds to the crossover of the two series, namely
$$  
\vert \Omega\vert \, 2^{j_0+1} \simeq \delta^{-p} 2^{-\frac{j_0rp}{2}} \vert g\vert_{W^s_p(\Omega)}^p \,.
$$
This implies 
$$
2^{j_0(1+\frac{rp}{2})}\simeq \vert \Omega\vert^{-1} \vert g\vert_{W^s_p(\Omega)}^p \delta^{-p} \,,
$$
and
$$
1+\frac{rp}{2}=1+\frac{p}{2} \left(t+s-\frac2p+\frac2q \right) = \frac{p}2(t+s) +\frac{p}{q} = p \, w\,,  \qquad \text{with } \quad w:= \frac12(t+s)+\frac1q \,. 
$$
We thus deduce
$$
2^{j_0}\simeq \vert\Omega\vert^{-\frac{1}{pw}}\vert g\vert_{W^s_p(\Omega)}^{\frac{1}{w}} \delta^{-\frac1w}
$$
and the two series amount to the same sum
$$
\#{\cal M} \lesssim \vert\Omega\vert^{1-\frac{1}{pw}}\vert g\vert_{W^s_p(\Omega)}^{\frac{1}{w}} \delta^{-\frac1w} \,.
$$
\item[(iv)] {\em Complexity}.  Apply finally the estimate \eqref{eq:complexity-REFINE} that controls the number of elements in $\mesh_k$ in terms of $\mathcal{M}$:
$$
\#\mesh_k-\#\mesh_0\lesssim \#\mathcal{M}\lesssim  \vert\Omega\vert^{1-\frac{1}{pw}}\vert g\vert_{W^s_p(\Omega)}^{\frac{1}{w}} \delta^{-\frac1w} \,.
$$
This in turn yields
$$
\delta \lesssim \vert\Omega\vert^{w-\frac{1}{p}}\vert g\vert_{W^s_p(\Omega)} ( \#\mesh_k-\#\mesh_0)^{-w} \,,
$$
which is the asserted inequality \eqref{eq:greedy-B} in view of the definition of $w$. This concludes the proof.  \qquad \qquad $\square$
\end{itemize}

We now apply Proposition \ref{prop:greedy} to data $\data=(A,c,f)$, starting with $A$. In this case, we have
$$
t=0, \quad q=\infty, \quad r= s-\frac2p >0 \,.
$$
This allows for $s=1, p>2$ (i.e., $A \in (W^1_p(\Omega))^{2 \times 2}$), or $s>0, p=\infty$ (i.e., $A \in (W^s_\infty(\Omega))^{2 \times 2}=(C^{0,s}(\bar{\Omega}))^{2 \times 2}$, the space of H\"older-continuous tensor fields of exponent $s$).
\begin{corollary}[approximation of $A$]\label{cor:approx-A} If $A \in (W^s_p(\Omega))^{2 \times 2}$ with $0 < s \leq 1$ and $p >\frac2s$, then
\begin{equation} \label{eq:approx-A}
\Vert A - \widehat{A}_\mesh \Vert_{L^\infty(\Omega)} \lesssim \vert\Omega\vert^{\frac{s}2-\frac{1}{p}}\vert A \vert_{W^{s}_p(\Omega)} ( \#\mesh)^{-\frac{s}2} \,. 
\end{equation}
Thus, $A$ belongs to the approximation class $\ACA_{\frac{s}2}$, and the {\tt{GREEDY}} algorithm provides a quasi-optimal approximation of $A$.
\end{corollary}

We next consider the reaction term $c$, for which we have 
$$
t=1, \quad q=\infty, \quad r= s-\frac2p +1 >0  \,.
$$
The latter inequality is surely satisfied if  condition \eqref{eq:sob} holds. Thus, we may take $s=1$, $p=2$ (i.e., $c \in H^1(\Omega)$), or $0 \leq s \leq 1$, $p=\infty$ (i.e., $c \in W^s_\infty(\Omega)$).
\begin{corollary}[approximation of $c$]\label{cor:approx-c} If $c \in W^s_p(\Omega)$ with $0 \leq s \leq 1$ and $p \geq \frac2{s}$, then
\begin{equation} \label{eq:approx-c}
\Vert \hh (c - \widehat{c}_\mesh) \Vert_{L^\infty(\Omega)} \lesssim \vert\Omega\vert^{\frac{1+s}2-\frac1p}\vert c \vert_{W^{s}_p(\Omega)} ( \#\mesh)^{-\frac{1+s}2} \,. 
\end{equation}
Thus, $c$ belongs to the approximation class $\ACc_{\frac{1+s}2}$, and the {\tt{GREEDY}} algorithm provides a quasi-optimal approximation of $c$.
\end{corollary}

We conclude with the forcing term $f$, for which we have
$$
t=1, \quad q=2, \quad r=s-\frac2p +2 > 0 \,.
$$
Again, the latter inequality is implied by \eqref{eq:sob}. Admissible cases are $0 \leq s \leq 1$, $p=2$ (i.e., $f \in H^s(\Omega)$), or $s=1$, $p=1$ (i.e., $f \in W^1_1(\Omega)$).
\begin{corollary}[approximation of $f$]\label{cor:approx-f} If $f \in W^s_p(\Omega)$ with $0 \leq s \leq 1$ and $p \geq \frac2{s+1}$, then
\begin{equation} \label{eq:approx-f}
\Vert \hh (f - \widehat{f}_\mesh) \Vert_{L^2(\Omega)} \lesssim \vert\Omega\vert^{\frac{s}2+1-\frac1p}\vert f \vert_{W^{s}_p(\Omega)} ( \#\mesh)^{-\frac{1+s}2} \,. 
\end{equation}
Thus, $f$ belongs to the approximation class $\ACf_{\frac{1+s}2}$, and the {\tt{GREEDY}} algorithm provides a quasi-optimal approximation of $f$.
\end{corollary}

\begin{remark}[rates of convergence]\label{rem:rates} {\rm
We see that the most critical data term is $A$, whose approximation error decays, according to \eqref{eq:approx-A}, with rate $-\frac{s}2$ ($0 < s \leq 1$) provided $A \in W^s_p(\Omega)$. If $s=1$, $p>2$, we get the best possible rate $-\frac12$. 

On the other hand, data $c$ and $f$ lead to a rate $- \frac{1+s}2 < - \frac12$ for any regularity $c, f \in W^s_p(\Omega)$ with $0 < s \leq 1$. This is observed in the numerical experiments of Sect. \ref{sec:experiments}.
If instead, $c$ and $f$ have minimal regularity for our {\tt AVEM} to make sense, namely $c\in L^\infty(\Omega)$, $f \in L^2(\Omega)$, then the convergence rates are $-\frac12$ for both data (i.e., $s=0$).
}
\end{remark}

\subsection{A pseudo-greedy strategy for $f$}

Since the local error estimators $\zeta_{ \mesh}(E; f) = h_E \| f-\widehat{f} \|_{L^2(E)}$ accumulate in $\ell^2$, the  
threshold $\delta$ of $\texttt{GREEDY}$ is not directly related to the desired tolerance $\varepsilon$. In fact,
all $\zeta_{ \mesh}(E; f)$ could be rather small relative to $\varepsilon$ and yet $\zeta_{ \mesh}(f) = \| \hh (f-\widehat{f}) \|_{L^2(\Omega)} > \frac13 \varepsilon$. A practical choice is $\delta = \max_{T\in\mesh}\zeta_{ \mesh}(E; f)$, but the ensuing algorithm is inefficient. We propose a minor modification of $\texttt{GREEDY}$ with similar properties as D\"orfler's algorithm that hinges on the maximum strategy. We describe the algorithm for a generic function $f \in W^s_p(\Omega)$ in the general setting presented at the beginning of this section, then we restrict the result to the forcing $f$ of Corollary \ref{cor:approx-f}.

\smallskip
Given $\theta \in (0,1)$ and a tolerance $\delta >0$, consider the algorithm
\begin{algotab}
  \>  $[\mesh] = \texttt{P-GREEDY}(\mesh, \delta)$
  \\
  \>  $\quad \text{while } \zeta_{ \mesh}(f) > \delta $ 
  \\
  \> $\quad \quad \mathcal{M}=\{E\in\mesh:\,  \zeta_{\mesh}(E; f) \geq \theta \, \displaystyle{\max_{E' \in \mesh} \zeta_{\mesh}(E'; f)} \} $
  \\
  \>  $\quad \quad  \mesh={\tt REFINE}(\mesh, \mathcal{M}) $ 
   \\
  \>  $\quad  \text{end while }$
  \\
  \>  $\quad  \text{return}(\mesh)$
\end{algotab}
\medskip

The following statement is the counterpart of Proposition \ref{prop:greedy} and Remark \ref{rem:greedy} for $\texttt{P-GREEDY}$.

\begin{proposition}[performance of $\texttt{P-GREEDY}$]\label{prop:p-greedy}
Let $r$ be defined in \eqref{eq:error-estimate-g}, and suppose $r>0$. Then $\tt{P-GREEDY}$ terminates in a finite number of steps. The output partition $\mesh$ satisfies the estimates
\begin{equation}\label{eq:p-greedy}
\zeta_\mesh(f)  \leq \delta \qquad \text{and } \qquad \zeta_\mesh(f) \ \lesssim \ \vert f \vert_{W^s_p(\Omega)} (\#\mesh)^{-\frac{t+s}2}\,. 
\end{equation}
\end{proposition}
\begin{proof}
Since the proof is similar to that of Proposition \ref{prop:greedy}, we only report the new ingredients. Let $\mesh_1, \dots, \mesh_k$ be the sequence of refinements produced by $\texttt{P-GREEDY}$, and  ${\cal M}_1, \dots {\cal M}_k$ be the sequence of marked elements, with ${\cal M}= {\cal M}_1 \cup \dots \cup {\cal M}_k$. Set 
$$
\mu_i := \max \{ \zeta_{\mesh_i}(E;f) : E \in \mesh_i \} \  (1 \leq i \leq k) \qquad \text{and } \qquad \mu :=\mu_{k-1} \,.
$$
Then, it holds
\begin{equation}\label{eq:pgreedy-1}
\zeta_{\mesh_k}(f) \leq \delta < \zeta_{\mesh_{k-1}}(f) \leq \mu (\# \mesh_{k-1})^\frac1q \leq \mu (\# \mesh_{k})^\frac1q .
\end{equation}
On the other hand, since $\texttt{REFINE}$ does not increase the element estimators, one has $\mu_i \geq \mu$ for any $i$, whence
$$
\zeta_{\mesh_i}(E,f) \geq \theta \, \mu_i \geq \theta \, \mu \qquad \forall E \in {\cal M}_i, \quad \forall i\,.
$$
Let us introduce the partition of ${\cal M}$ into disjoint subsets ${\cal P}_j$ as in the proof of Proposition \ref{prop:greedy}. If $E \in {\cal P}_j$, denoting by $i$ the index such that $E \in {\cal M}_i$, we get
$$
\theta \, \mu \leq \zeta_{\mesh_i}(E,f) \lesssim h^r_E \vert f \vert_{W^s_p(\Omega)} \leq 2^{-\frac{jr}2} \vert f \vert_{W^s_p(\Omega)}\,,
$$
whence
$$
\#\mathcal{P}_j\lesssim \theta^{-p} \mu^{-p}  2^{-\frac{jrp}{2}}\vert f\vert^p_{W^s_p(\Omega)}.
$$
As in the proof of Proposition \ref{prop:greedy}, this yields
$$
\mu \lesssim \vert f \vert_{W^s_p(\Omega)} ( \#\mesh_k-\#\mesh_0)^{-w} \,,
$$
and we conclude using \eqref{eq:pgreedy-1} and the bound $\#\mesh_k \geq c_0 \#\mesh_0$ for $c_0>1$.
\end{proof}

If the forcing $f \in W^s_p(\Omega)$ with $0 \leq s \leq 1$ and $p \geq \frac2{s+1}$, as in Corollary \ref{cor:approx-f}, then \eqref{eq:p-greedy} reads $\zeta_\mesh(f) \ \lesssim \ \vert f \vert_{W^s_p(\Omega)} (\#\mesh)^{-\frac{1+s}2}$, i.e, $\texttt{P-GREEDY}$ provides a quasi-optimal approximation of $f$ with convergence rate $- \frac{1+s}{2}$. In particular, if $f\in L^2(\Omega)$, then $\zeta_\mesh(f) \ \lesssim \ \Vert f \Vert_{L^2(\Omega)} (\#\mesh)^{-\frac{1}2}$.

\section{Numerical results}\label{sec:experiments}

In this section we present a numerical experiment to confirm the convergence and optimality properties of the {$2$-step algorithm \texttt{AVEM}.
We consider problem \eqref{eq:pde} in the L-shaped} domain $\Omega = (-1, 1)^2\setminus([0, 1]\times [-1, 0])$,  with diffusion term $A = a I$, where
\[
\begin{aligned}
a(x, y) &= 1 + \exp\bigl(-50 ((x + 0.5)^2 + (y+0.5)^2) \bigr)
+ \exp\bigl(-50 ((x + 0.5)^2 + (y-0.5)^2) \bigr) \,,
\end{aligned}
\]
and reaction term
\[
\begin{aligned}
c(x, y) &= 1 + \exp\bigl(-50 ((x + 0.5)^2 + y^2) \bigr)
+ \exp\bigl(-50 (x^2 + (y-0.5)^2) \bigr) \,;
\end{aligned}
\]
 {note that the Gaussians in the definition of $a$ and $c$ have the same intensity but are located
in different places within $\Omega$ (see Figures \ref{fig:A} and \ref{fig:c}).}
The load term $f$ and the Dirichlet
boundary conditions are chosen in accordance with the analytical solution
\[
u(x, y) = r^{\frac{2}{3}} \sin \bigl(  2\alpha/3\bigr) + 
\exp\bigl(-1000 ((x - 0.5)^2 + (y-0.5)^2) \bigr) \,,
\]
where $(r, \alpha)$ are the polar coordinates around the origin.
Notice that the exact solution $u$ is singular at the reentrant corner:  {it belongs to the Sobolev spaces $H(\Omega)^{\frac{5}{3}-\epsilon}$ with $\epsilon > 0$ and $W^2_p(\Omega)$ with $p>1$. It also exhibits a rapid transition of order $10^{-3/2}$ around the point $(0.5,0.5)$ due to the presence of a very narrow Gaussian. The three Gaussians are meant to test the performance of the module {\tt DATA}.

We utilize the following parameters in the numerical test
\[
\gamma = \texttt{1}, \quad \Lambda = \texttt{10}, \quad \theta_{\text{D{\"o}rfler}} = \texttt{0.5}, \quad \omega = \texttt{1},
\quad \theta_{\text{p-greedy}} = \texttt{sqrt(0.75)},
\quad \texttt{tol}=\texttt{0.125},
\]
where $\gamma$ is parameter of the \texttt{dofi-dofi} stabilization \eqref{eq:stab-dofidofi}, $\Lambda$ is
the bound for the global index of non-conforming partitions in Definition \ref{def:Lambda-partitions},
$\theta_{\text{D{\"o}rfler}}$ is the D{\"o}rfler marking parameter \eqref{eq:dorfler},
$\omega$ is the safety input parameter of {\tt DATA},
$\theta_{\text{p-greedy}}$ is the pseudo-greedy marking parameter \eqref{eq:pseudogreedy-data},
and {\tt tol} is the target tolerance of {\tt AVEM}.
We implement algorithm \texttt{AVEM} with a uniform structured triangular mesh $\mesh_0$
with diameter $h=\texttt{0.125}$ and initial tolerance $\epsilon_0=\texttt{1}$.
}

In order to estimate the VEM error between the exact solution $u$ and
the VEM solution $u_\mesh$, we consider the computable $H^1$-like error quantity:
\[
\texttt{H\textasciicircum 1-error} := 
\frac{ |u - \Pimesh u_\mesh|_{1, \mesh}}
{|u|_{1, \Omega}} \,.
\]
In Fig. \ref{fig:avem} (left) we display the estimator $\etamesh(\umesh,\data)$, 
the data error $\zeta_{\widehat \mesh}(\data)$ and the 
$\texttt{H\textasciicircum 1-error}$
obtained with algorithm \texttt{AVEM}.
In Fig. \ref{fig:avem} (right) we exhibit the data error $\zeta_{\widehat \mesh}(\data)$ and the addends
$\zeta_{\widehat \mesh}(A)$, $\zeta_{\widehat \mesh}(c)$,
$\zeta_{\widehat \mesh}(f)$ (cf. \eqref{eq:zeta-data} and \eqref{eq:zeta-Acf}). 
Notice that the number of iterations of the algorithm \texttt{AVEM} is $K=\log_2(\epsilon_0/\texttt{tol}) = 3$.

\begin{figure}[!htb]
\begin{center}
\includegraphics[scale=0.225]{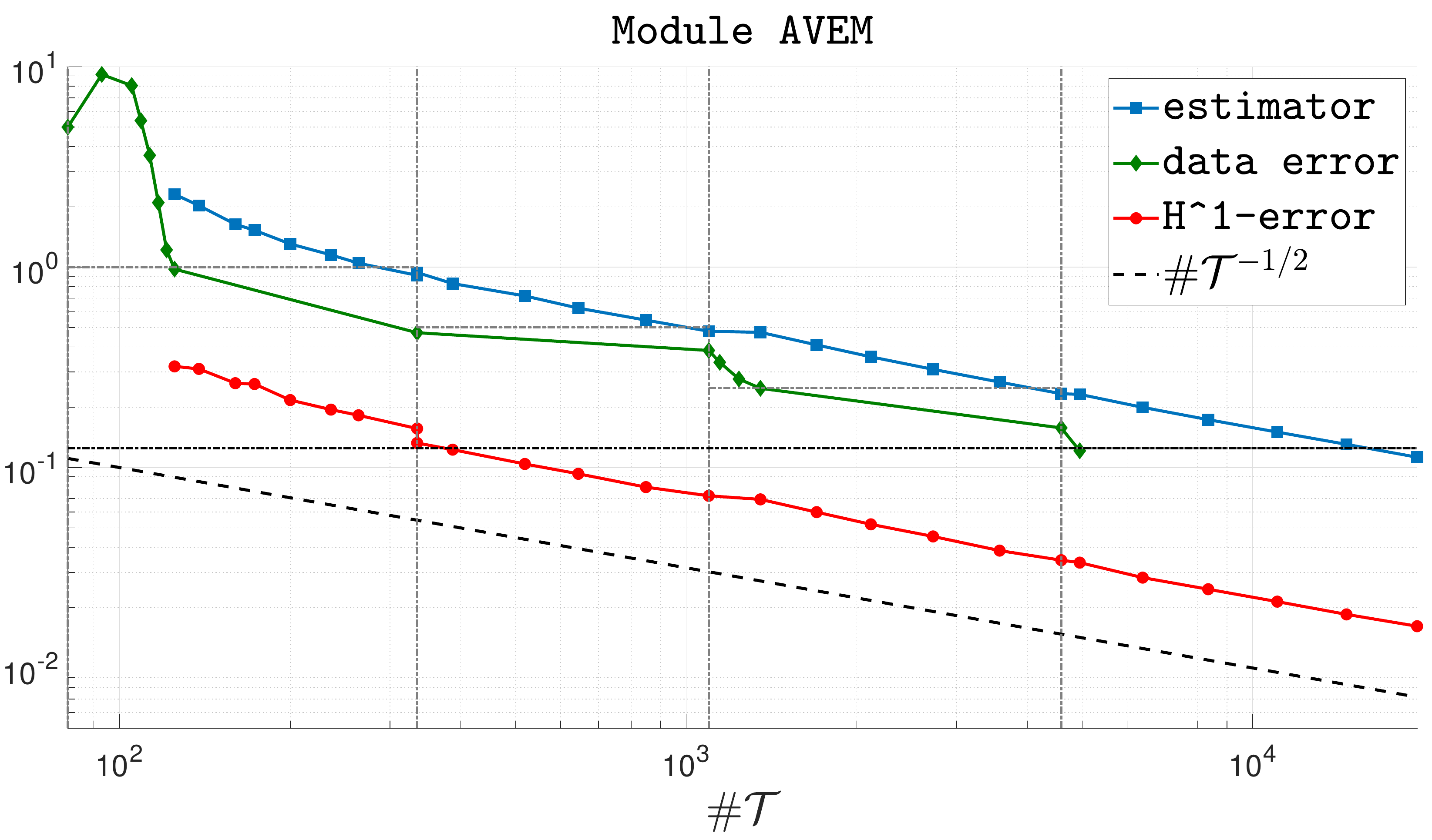}
\includegraphics[scale=0.225]{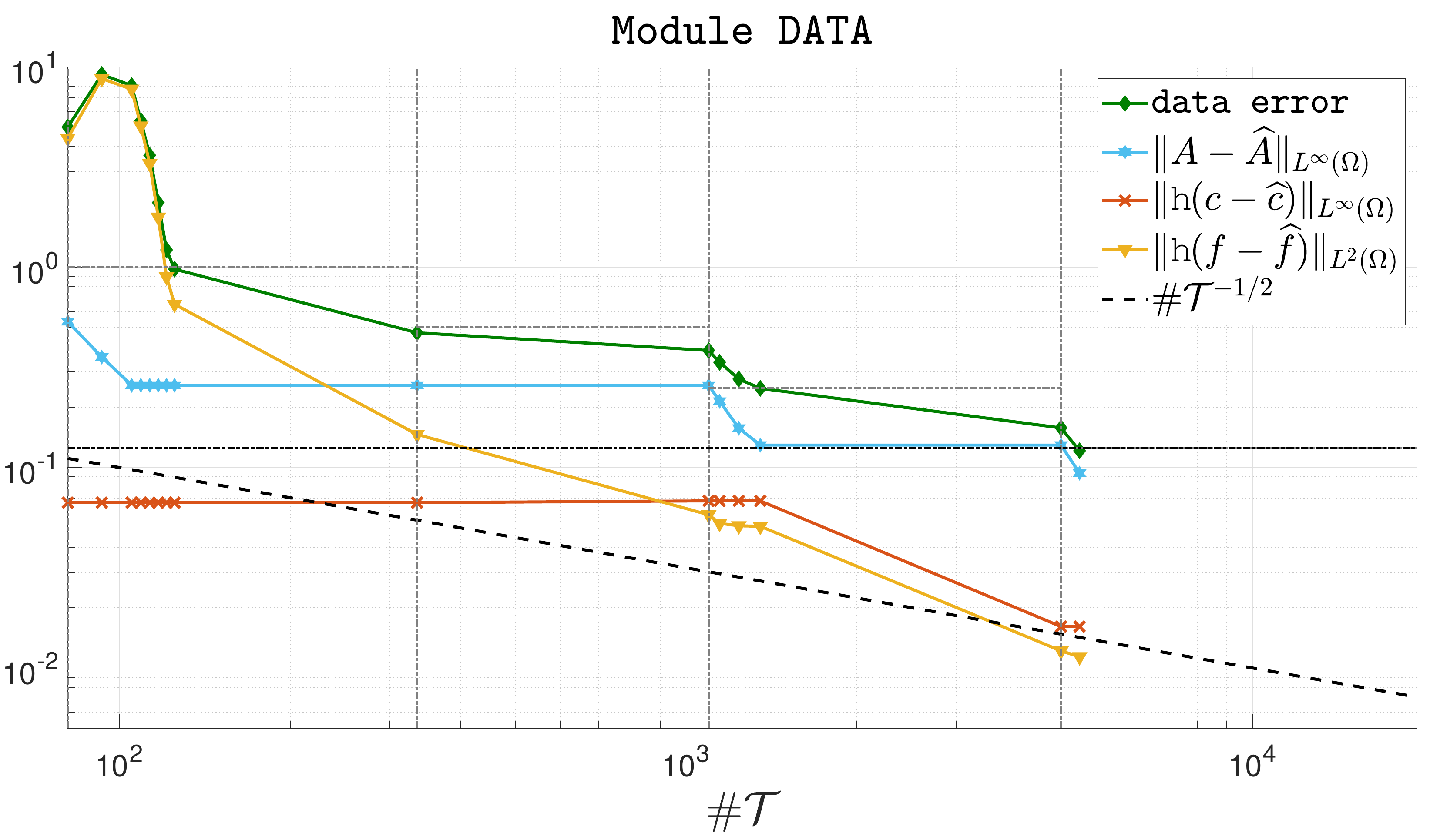}
\end{center}
\caption{{Left: 
estimator $\etamesh(\umesh,\data)$,
data error $\zeta_{\widehat\mesh}(\data)$,
$\texttt{H\textasciicircum 1-error}$
obtained with the algorithm \texttt{AVEM}. 
Right: data error $\zeta_{\widehat\mesh}(\data)$,
tensor error $\zeta_{\widehat\mesh}(A)$,
reaction error $\zeta_{\widehat\mesh}(c)$,
load error $\zeta_{\widehat\mesh}(f)$,
obtained with the algorithm \texttt{AVEM}. 
In both figures the optimal decay is indicated by the dashed line with slope $\texttt{-0.5}$.}}
\label{fig:avem}
\end{figure}

The predictions of  {Theorem \ref{prop:convergence-GALERKIN} (contraction property of {\tt GALERKIN})}
are confirmed: both the estimator $\etamesh(\umesh,\data)$ and the  $\texttt{H\textasciicircum 1-error}$ converge to zero and
the decay rate reaches asymptotically the theoretical optimal value $\# \mesh^{-1/2}$;
this corresponds to $s=1/2$ in Theorem \ref{Texact:optimality-AVEM} (optimality of {\tt AVEM}).  
Concerning data approximation, we observe from Fig. \ref{fig:avem} (right) that 
$\zeta_{\widehat \mesh}(\data)$ decays with rate $\# \mesh^{-1/2}$
dictated by $\zeta_{\widehat\mesh}(A)$, as predicted by Corollary  \ref{cor:approx-A} ,
while $\zeta_{\widehat \mesh}(c)$ and $\zeta_{\widehat \mesh}(f)$
exhibit a faster decay rate. This is due to regularity of $(c,f)$ beyond $L^\infty(\Omega)\times L^2(\Omega)$,
as predicted by Corollaries \ref{cor:approx-c} and \ref{cor:approx-f}.
We finally notice from Fig. \ref{fig:avem} that the module \texttt{DATA} is active for all $k$ except
$k=1$ because $\zeta_{\mesh_1}(\data) < \epsilon_1$.

\begin{figure}[!htb]
\begin{center}
\includegraphics[scale=0.225]{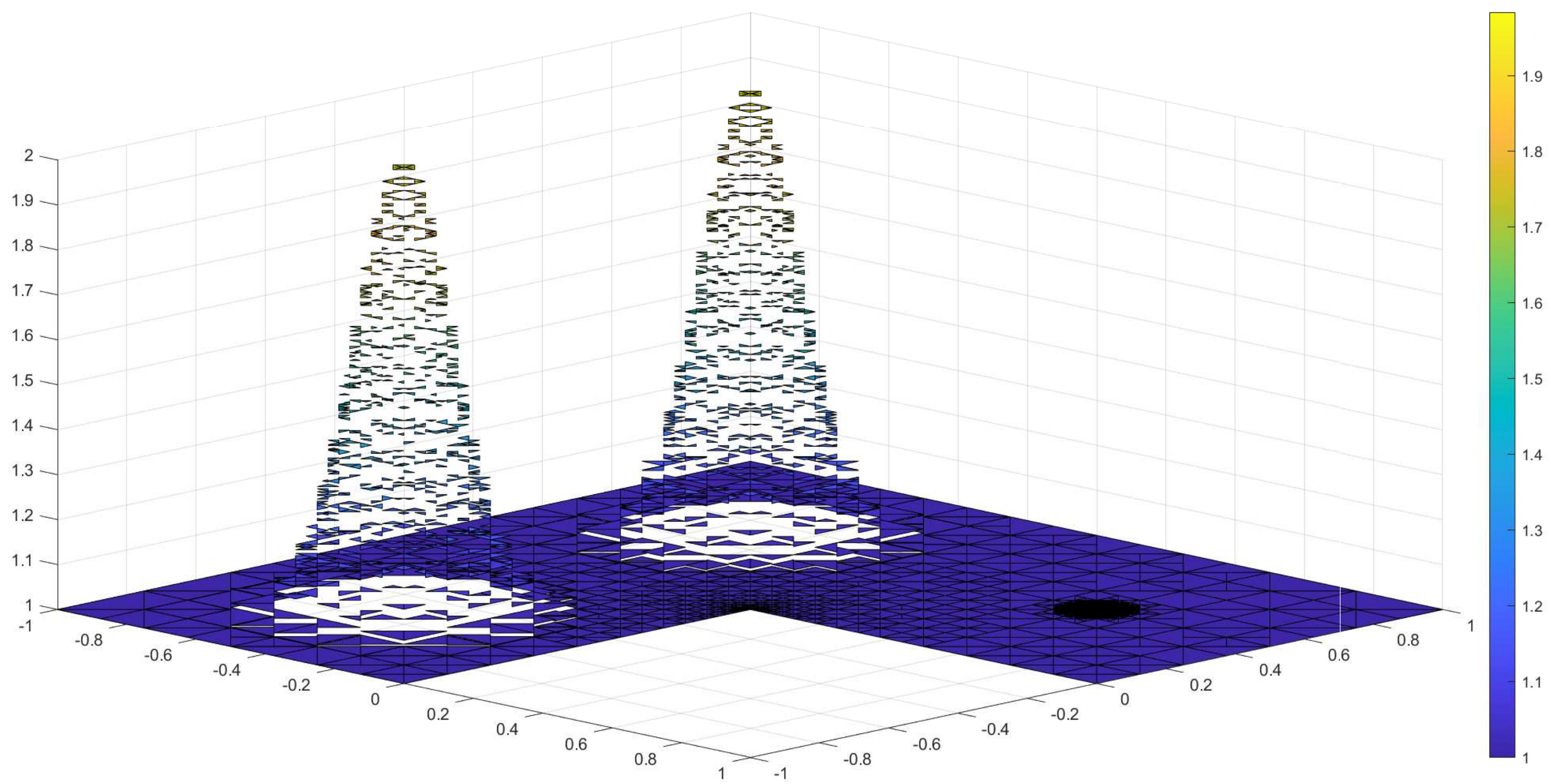}
\includegraphics[scale=0.225]{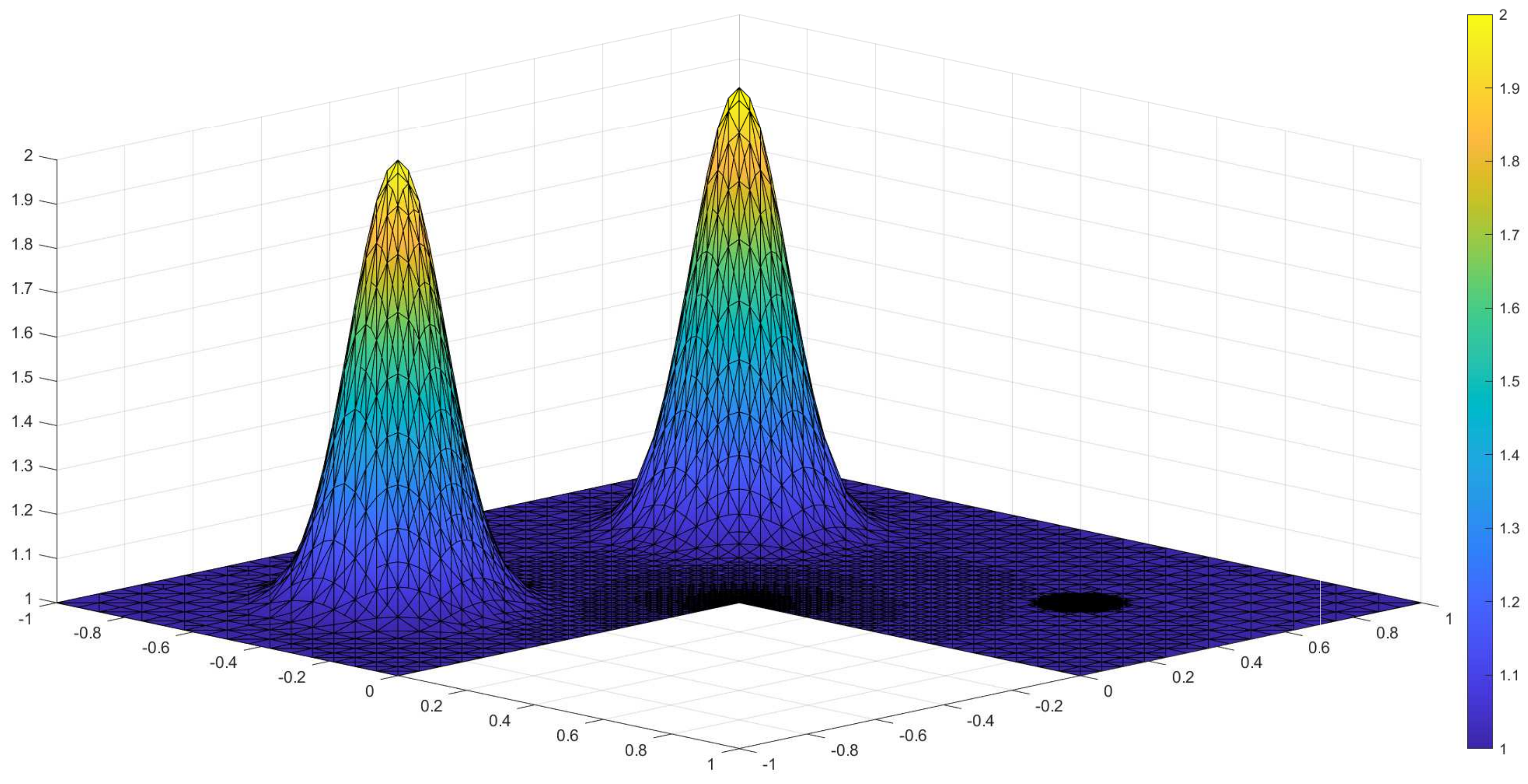}
\end{center}
\caption{ {Left: graph of  the piecewise constant approximation $\widehat{a}$ of $a$ (w.r.t. 
$\widehat\mesh_K$). 
Right: graph of the piecewise linear interpolant of $a$ (w.r.t. 
$\mesh_{K+1}$).}}
\label{fig:A}
\end{figure}

\begin{figure}[!htb]
\begin{center}
\includegraphics[scale=0.225]{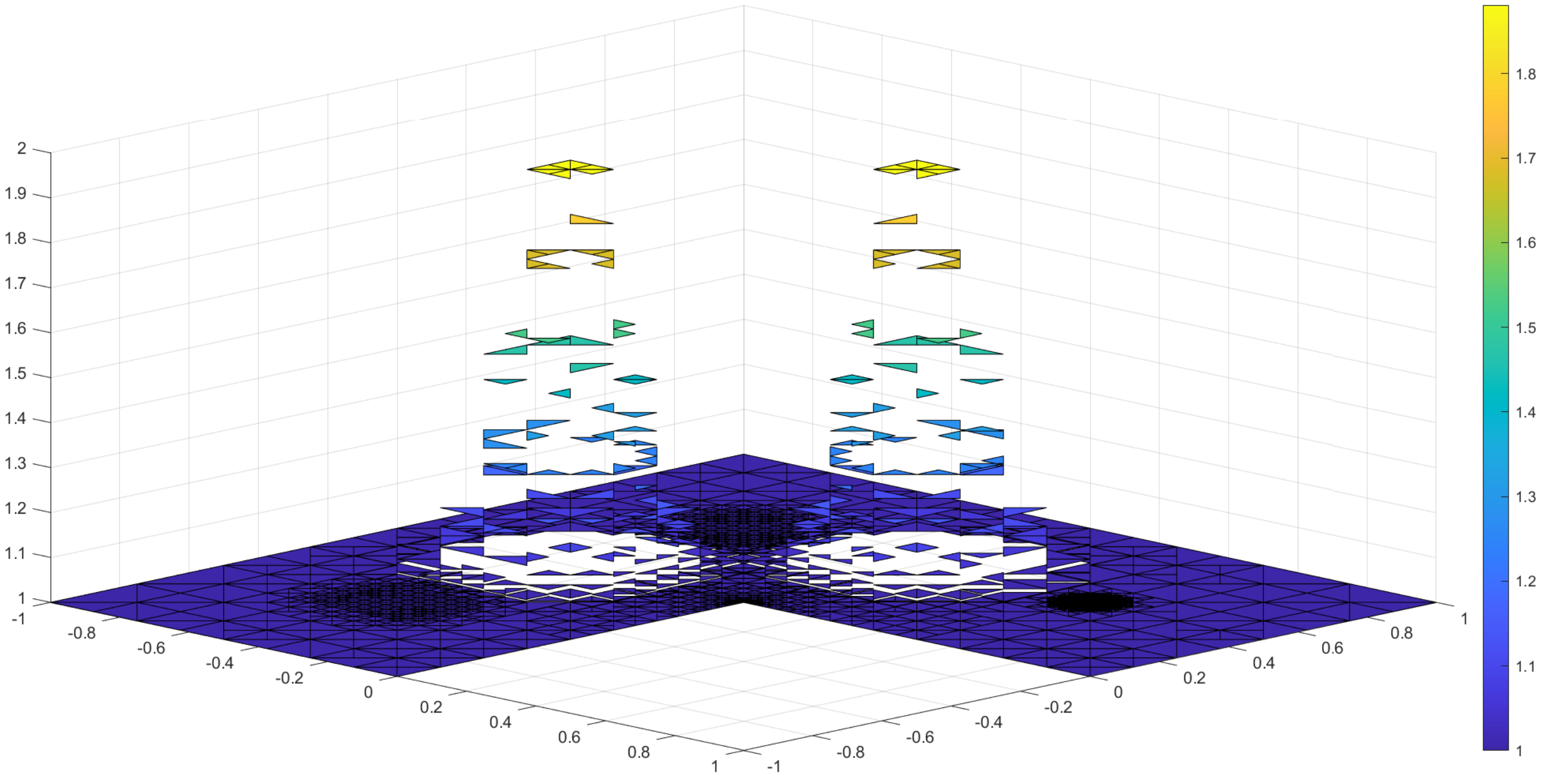}
\includegraphics[scale=0.225]{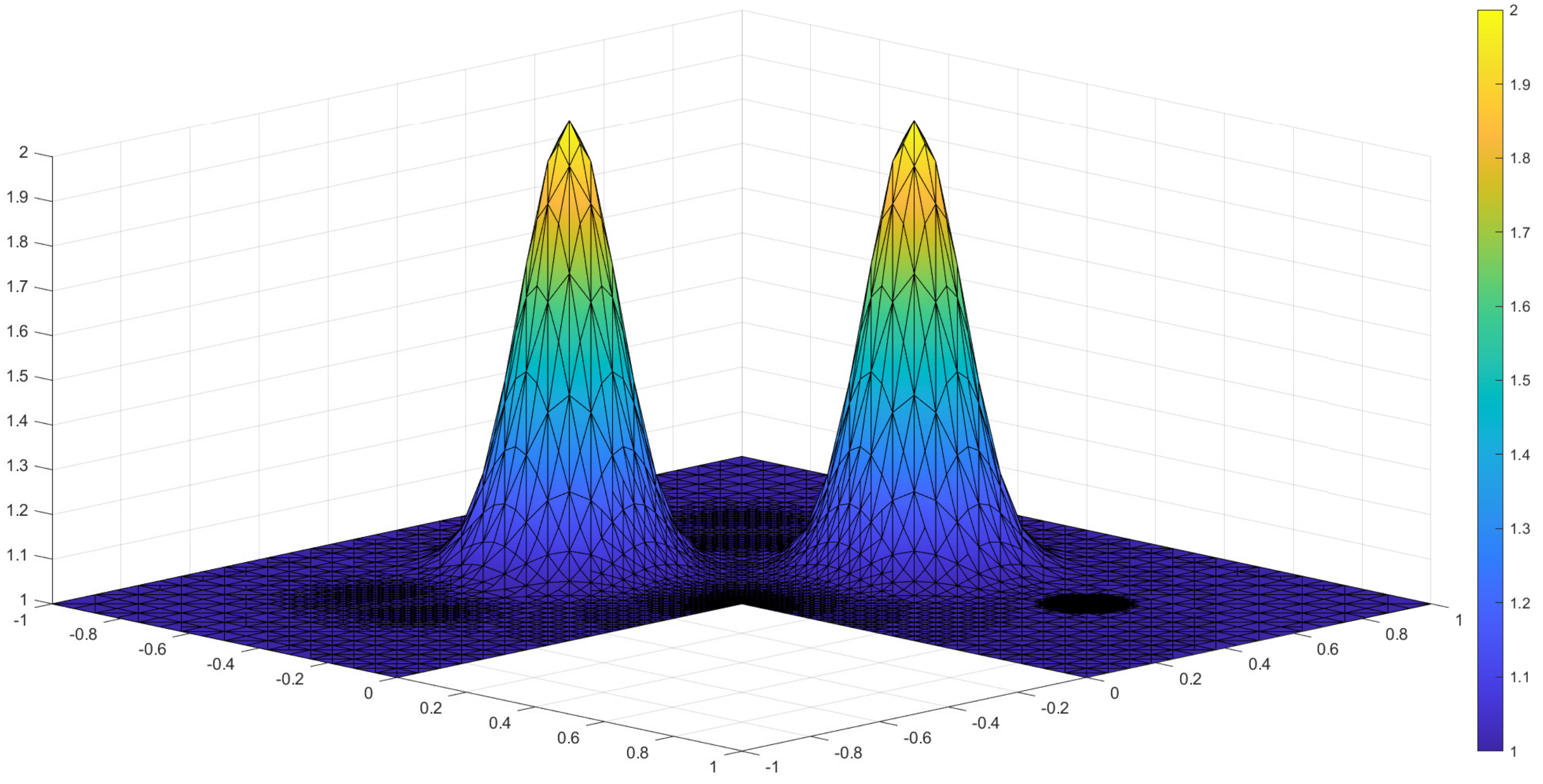}
\end{center}
\caption{ {Left: graph of the piecewise constant approximation  $\widehat{c}$ of $c$ (w.r.t. 
$\widehat\mesh_K$). 
Right: graph of the piecewise linear interpolant of $c$ (w.r.t. 
$\mesh_{K+1}$). Notice much coarser resolution than in Fig. \ref{fig:A}.}}
\label{fig:c}
\end{figure}

 {
  In order to highlight the different level of approximation of data $\data = (A,c,f)$ required by {\tt AVEM},
we display in Figs. \ref{fig:A}, \ref{fig:c} and \ref{fig:f} 
the graphs of the piecewise constant approximations $\widehat\data = (\widehat A,\widehat c,\widehat f)$ with
respect to the mesh $\widehat \mesh_K$ (left), and of the continuous piecewise linear counterparts 
with respect to the mesh $\mesh_{K+1}$ (right). Since the Gaussians in $a$ and $c$ are located in non-overlapping
subregions of $\Omega$, it is possible to see that {\tt AVEM} imposes a much finer resolution of $a$ than of $c$
in both meshes $\widehat \mesh_K$ and $\mesh_{K+1}$; this is due
to the extra factor $\hh$ in the definition \eqref{eq:zeta-Acf} of $\zeta_{\widehat\mesh}(c)$.
}

\begin{figure}[!htb]
\begin{center}
\includegraphics[scale=0.225]{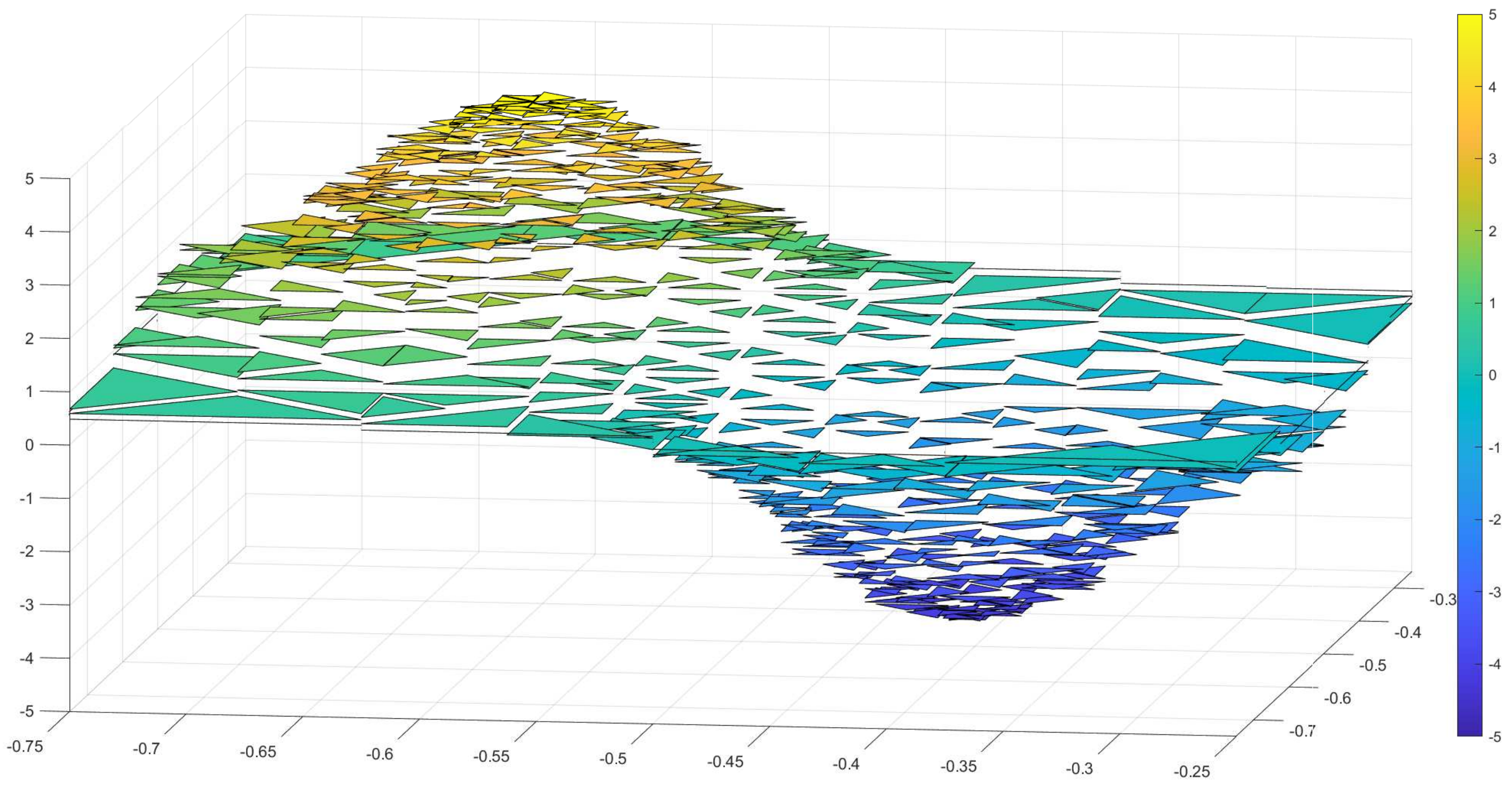}
\includegraphics[scale=0.225]{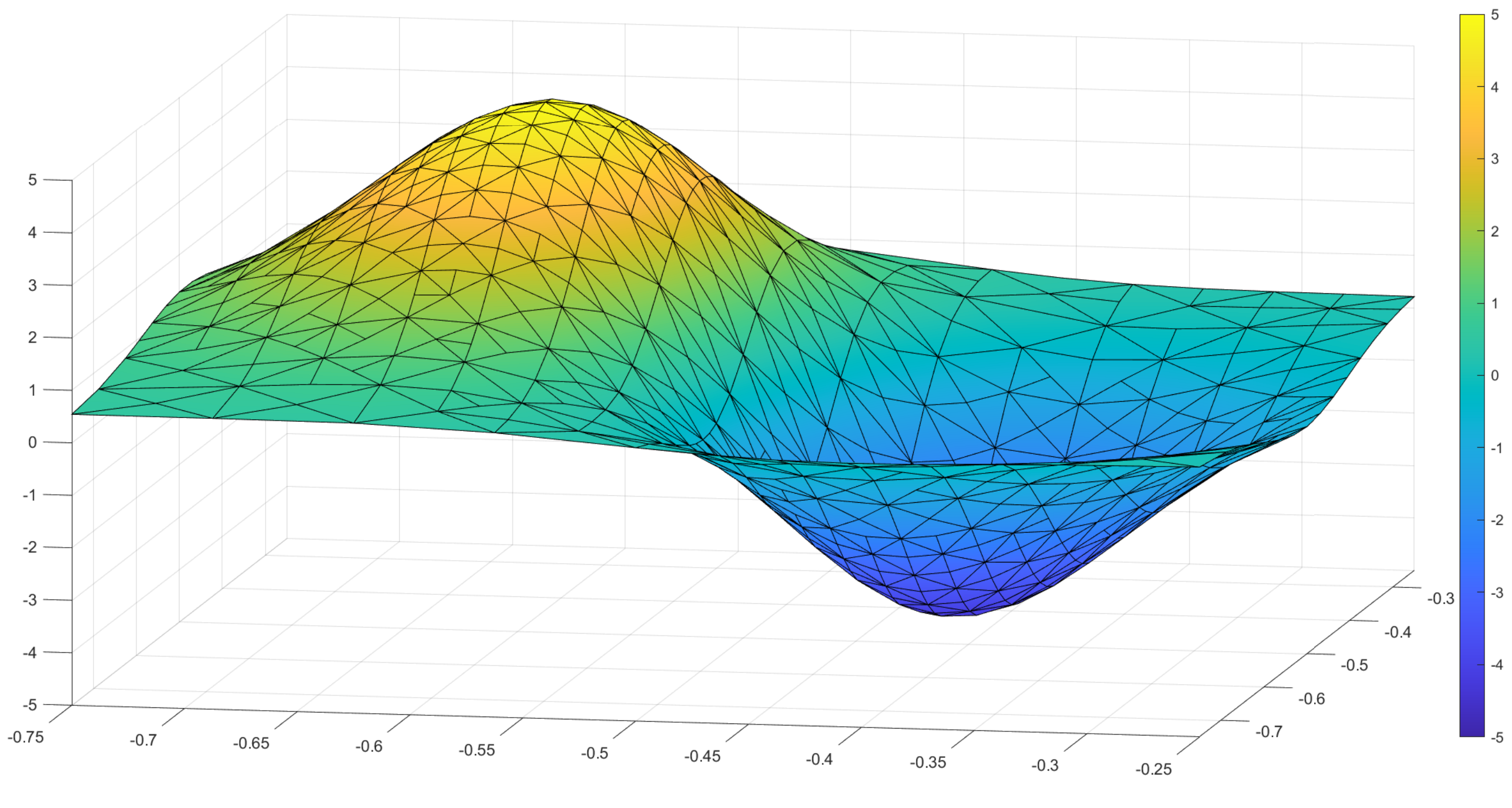}
\end{center}
\caption{ {Zoom to $(-0.75, -0.25)^2$ for the load term $f$. Left: graph of the piecewise constant
    approximation $\widehat{f}$ of $f$ (w.r.t. $\widehat\mesh_K$). 
Right: graph of the piecewise linear interpolant of $f$ (w.r.t. 
$\mesh_{K+1}$).}}
\label{fig:f}
\end{figure}

Finally in Figs. \ref{fig:mesh}, \ref{fig:mesh-gaussiana} and \ref{fig:mesh-origine} we compare
the grids $\widehat \mesh_K$ and $\mesh_{K+1}$ generated by the modules \texttt{DATA} and \texttt{GALERKIN}
upon termination of {\tt AVEM}.
The heat map on the rightmost pictures shows, for each element $E \in \mesh_{K+1}$, the number of newest-vertex bisections needed to create $E$ starting from $\widehat \mesh_K$, according to the colorbar in Fig. \ref{fig:colorbar}.
The number of nodes \texttt{N\_vertices} and elements \texttt{N\_elements} are 
\begin{align*}
& \texttt{N\_vertices}(\widehat \mesh_K) = \texttt{5030},
&&\texttt{N\_elements}(\widehat \mesh_K) = \texttt{9236},
\\
& \texttt{N\_vertices}(\mesh_{K+1}) = \texttt{19676},
&&\texttt{N\_elements}(\mesh_{K+1}) = \texttt{37244}.
\end{align*}
Furthermore, the number of polygons in $\widehat{\mesh}_K$ (elements with more than three vertices) is 
\texttt{730}: \texttt{723} quadrilaterals, \texttt{2} pentagons, \texttt{5} hexagons;
the number of polygons in $\mesh_{K+1}$ is 
\texttt{1920}: \texttt{1908} quadrilaterals, \texttt{16} pentagons, \texttt{4} hexagons.
In Fig. \ref{fig:mesh-gaussiana} we plot a zoom to $(0.35, 0.65)^2$ of the meshes
$\widehat \mesh_K$ and $\mesh_{K+1}$. 
We highlight for both meshes the presence of hexagons in this subregion. Moreover, looking at the vertices having maximum global index $\lambda$ sitting on the hexagons, we realize that the global indices are $\Lambda_{\widehat \mesh_K} = 2$ and  $\Lambda_{\mesh_{K+1}} = 3$.
It is worth noting that the threshold $\Lambda = \texttt{10}$ is never reached by \texttt{AVEM};
therefore, the condition of $\Lambda$-admissibility is not restrictive in practice.
We further notice that the Gaussian in $(0.5, 0.5)$ associated with $f$ is sufficiently resolved by \texttt{DATA}. 
In Fig. \ref{fig:mesh-origine} we present a zoom to $(-10^{-2}, 10^{-2})^2$ to examine
mesh refinement at the origin.
We see that the mesh $\mesh_{K+1}$ exhibits a rather strong grading at the reentrant corner, in accordance with the singularity of the exact solution.
Elements in $\mesh_{K+1}$ in this region need up to five newest-vertex bisection refinements relative to $\widehat \mesh_K$.

We close this section with the following observation.
From Figs. \ref{fig:mesh}, \ref{fig:mesh-gaussiana} and \ref{fig:mesh-origine} it can be appreciated how the presence of hanging nodes allows for quite abrupt and `steep' refinements where needed in order to approximate the data and the solution singularity. In this respect, a direct comparison with AFEM in terms of generated meshes can be found in \cite{paperA}. Such numerical results suggest that, although as shown in Remark \ref{rem:avem-afem} the approximation classes of AVEM and AFEM are the same, this added flexibility may be an important asset in adaptivity, especially in situations with more complex geometry. This aspect is worth further investigation, but is not within the scopes of the present contribution.

\begin{figure}
\begin{center}
\begin{overpic}[scale=0.2]{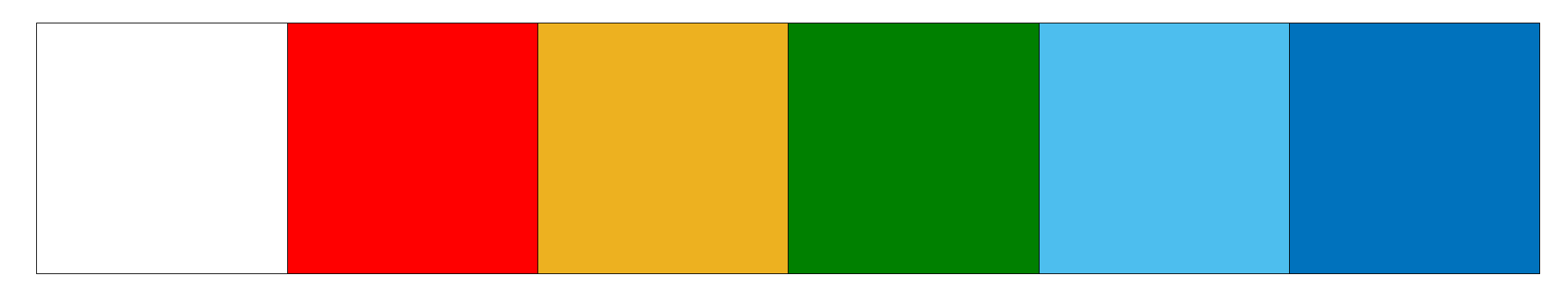}
\put(12,20){Number of bisections}
\put(8, -6){0}
\put(24,-6){1}
\put(40,-6){2}
\put(56,-6){3}
\put(72,-6){4}
\put(88,-6){5}
\end{overpic}
\end{center}
\caption{{Colorbar for the heat map in Figures \ref{fig:mesh}, \ref{fig:mesh-gaussiana} and \ref{fig:mesh-origine}.}}
\label{fig:colorbar}
\end{figure}
 
\begin{figure}[!htb]
\begin{center}
\includegraphics[scale=0.275]{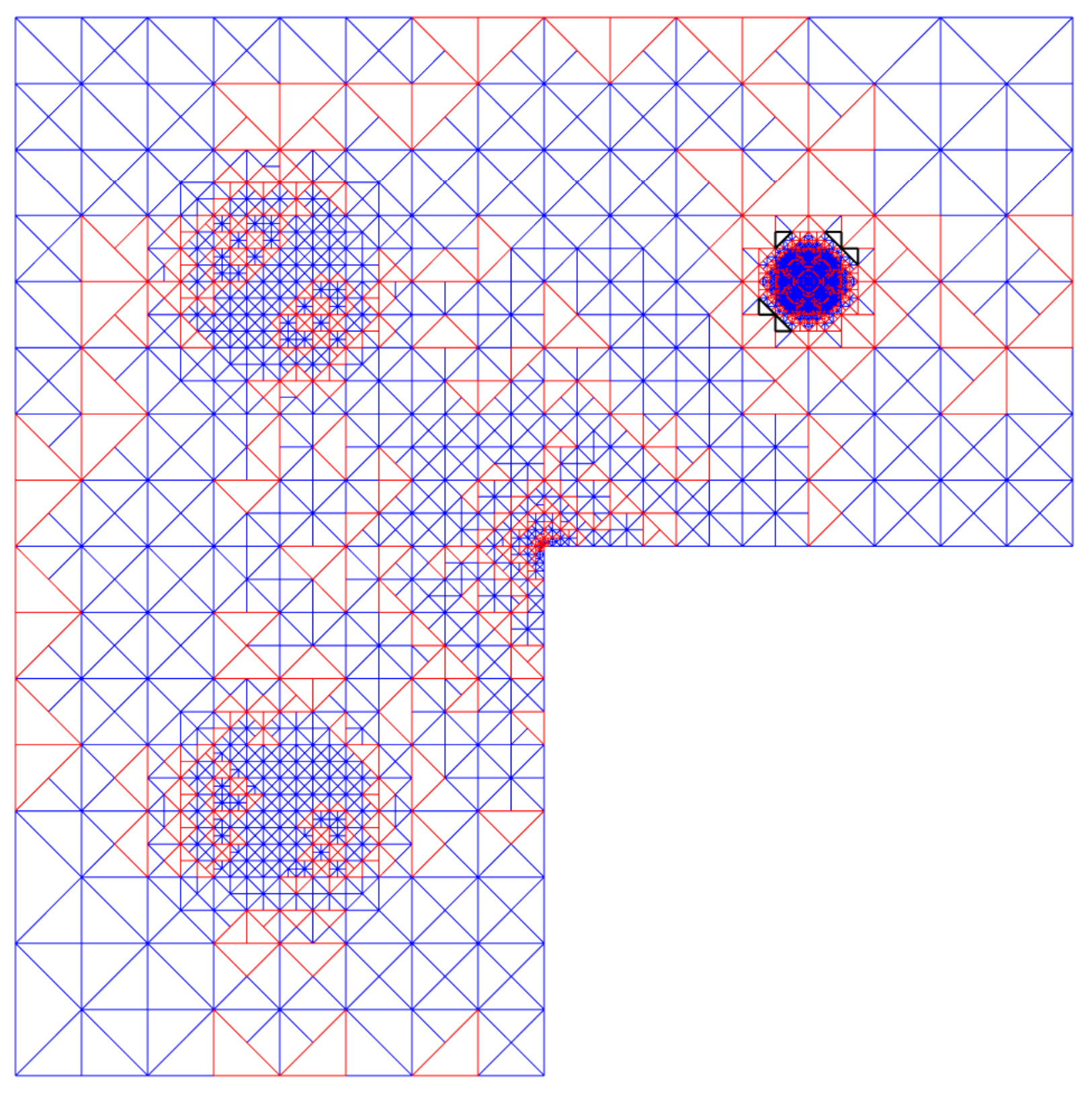}
\quad
\includegraphics[scale=0.275]{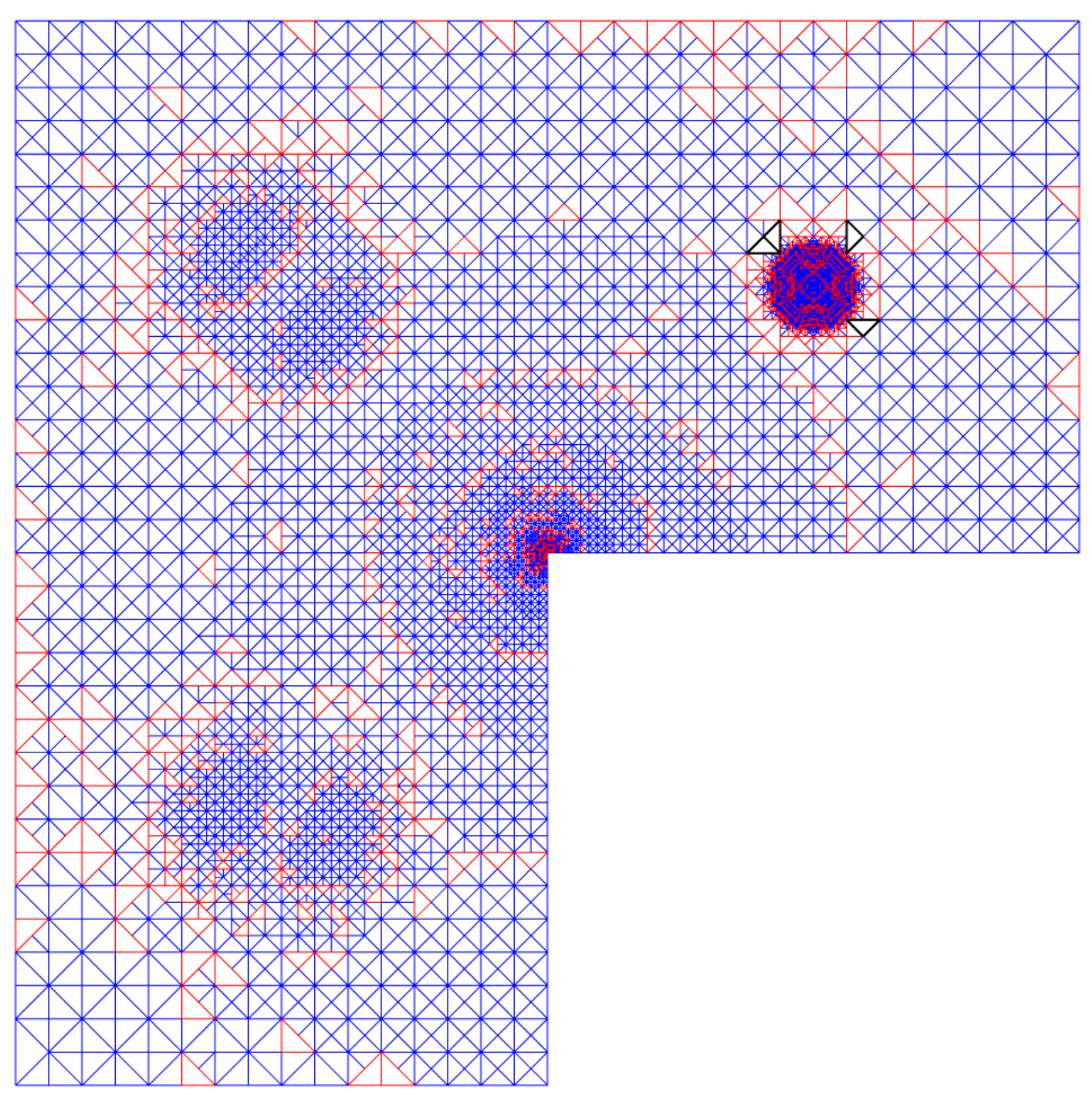}
\quad
\includegraphics[scale=0.275]{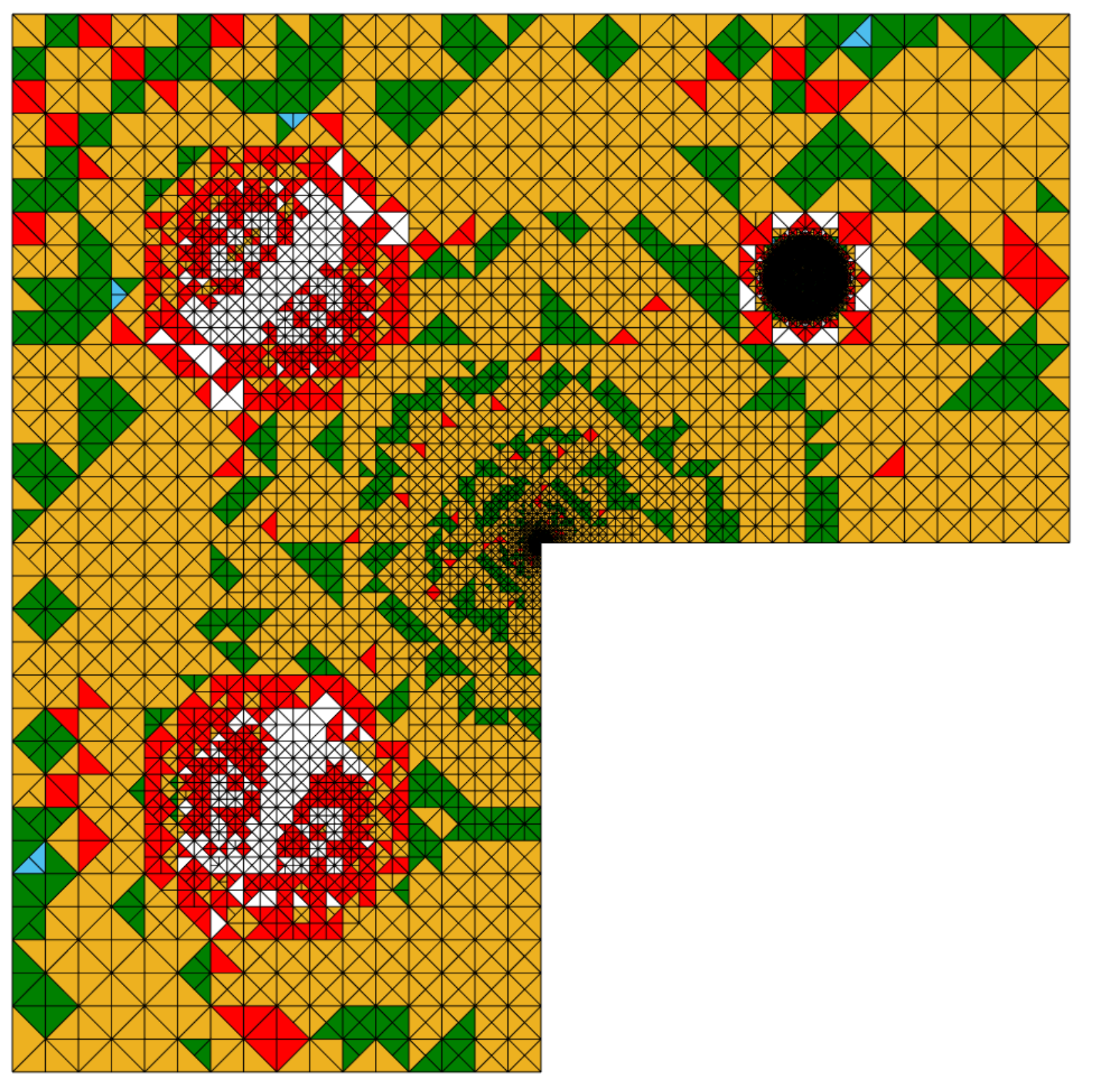}
\end{center}
\caption{ {Left: final grid  $\widehat\mesh_K$ generated by \texttt{DATA}.
Middle: final grid  $\mesh_{K+1}$ generated by \texttt{GALERKIN}. 
Mesh elements having more than three vertices (polygons) are drawn in red.
Right: heat map representing for each $E \in \mesh_{K+1}$ the number of newest-vertex bisection needed to generate $E$ starting from the mesh $\widehat\mesh_K$ (colorbar in Fig. \ref{fig:colorbar}).}}
\label{fig:mesh}
\end{figure}

\begin{figure}[!htb]
\begin{center}
\includegraphics[scale=0.275]{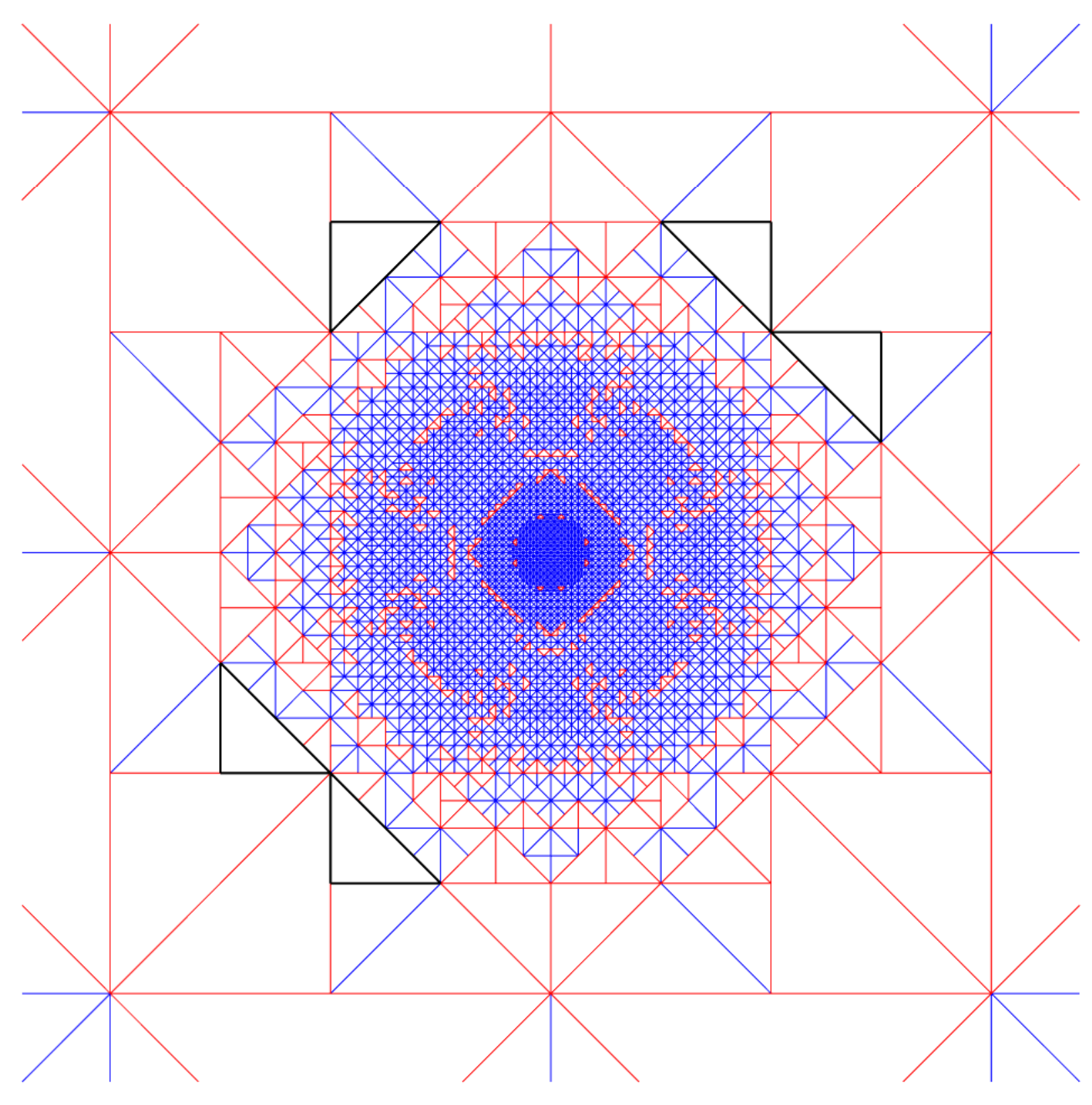}
\quad
\includegraphics[scale=0.275]{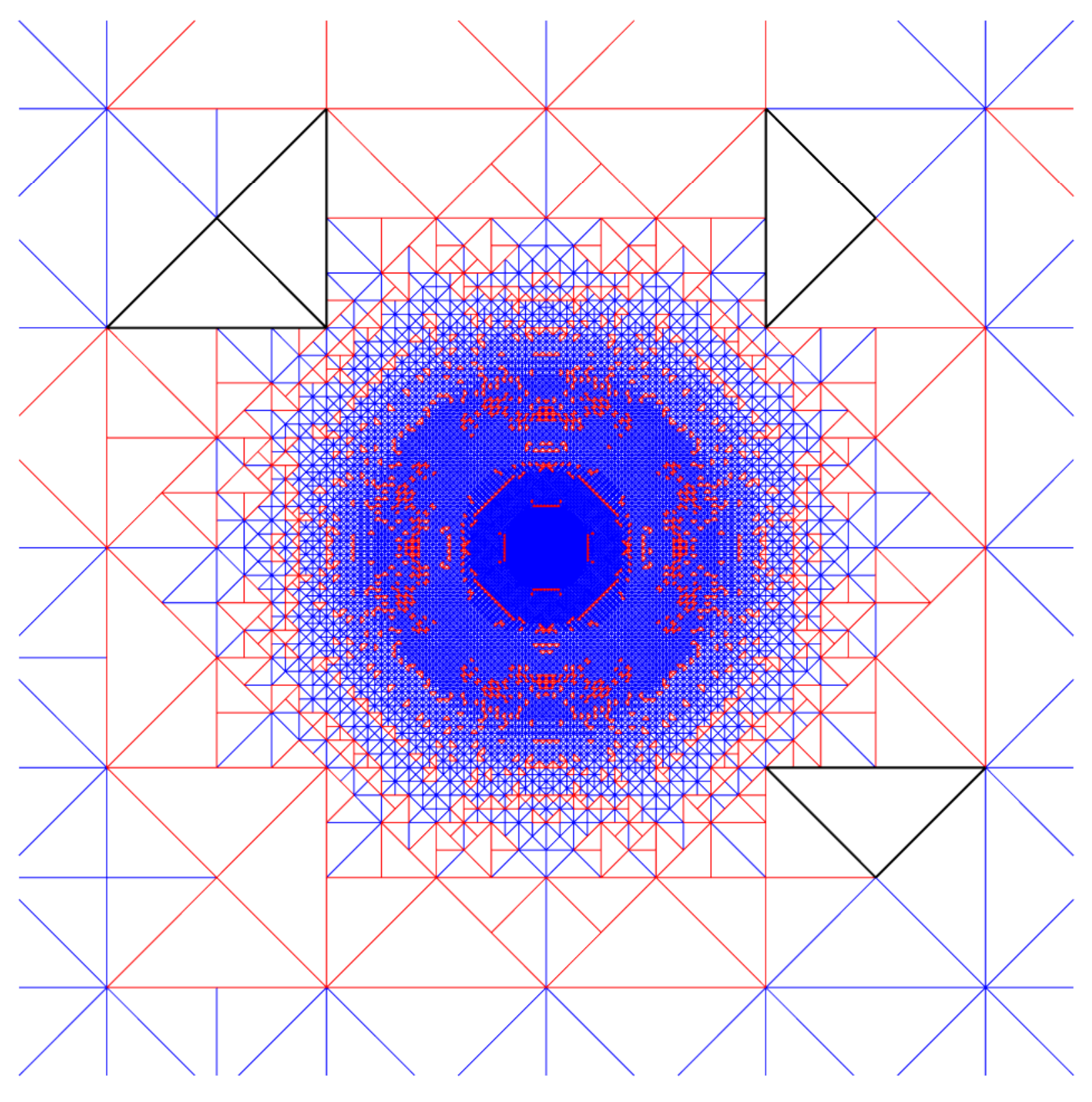}
\quad
\includegraphics[scale=0.275]{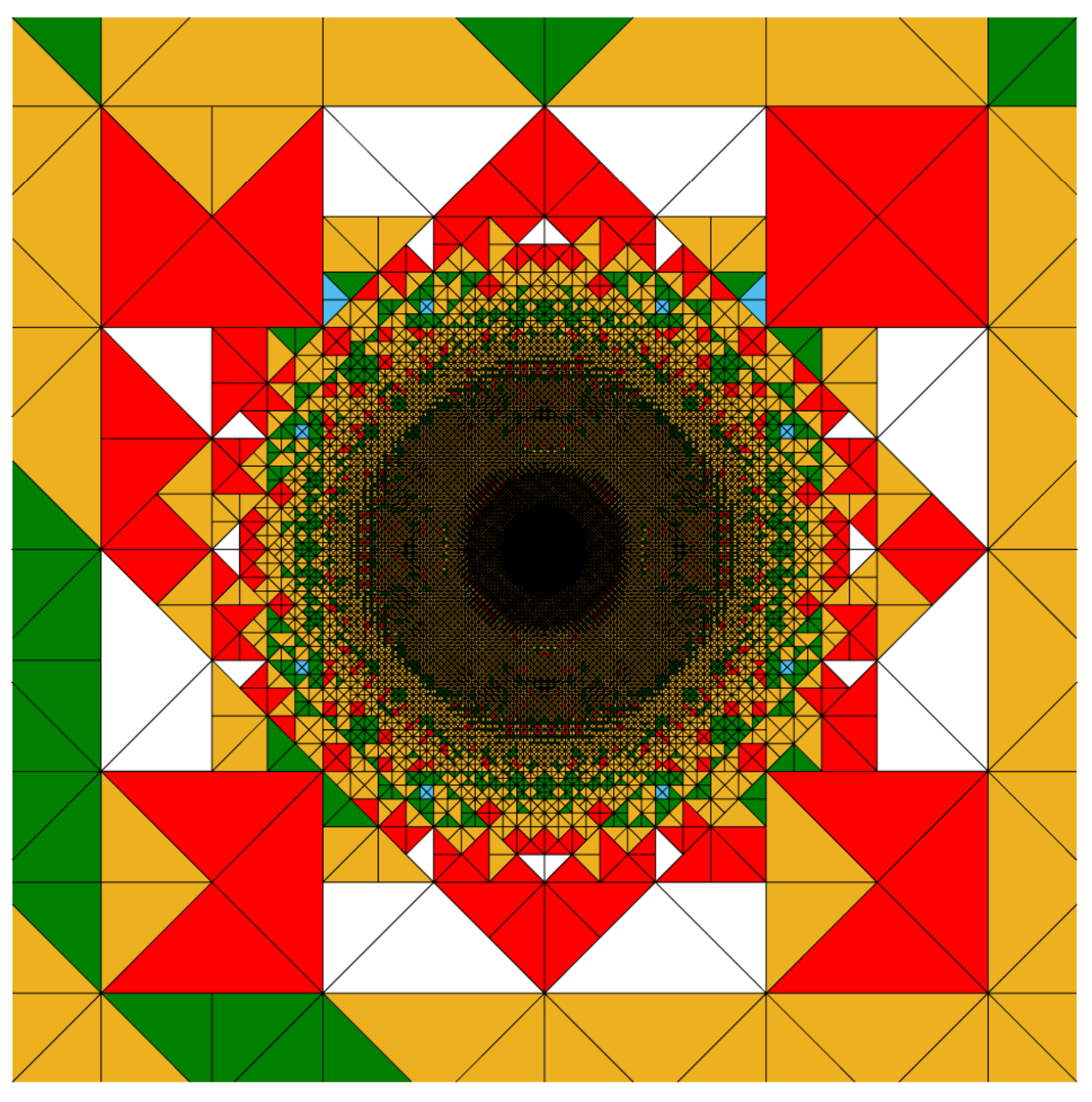}
\end{center}
\caption{ {Zoom to $(0.35, 0.65)^2$ related to $f$.
Left: final grid  $\widehat\mesh_K$ generated by \texttt{DATA}.
Middle: final grid  $\mesh_{K+1}$ generated by \texttt{GALERKIN}. 
Elements having more than three vertices (polygons) are drawn in red; elements drawn in black are hexagons.
Right: heat map representing for each $E \in \mesh_{K+1}$ the number of newest-vertex bisection needed to generate $E$ starting from the mesh $\widehat\mesh_K$ (colorbar in Fig. \ref{fig:colorbar}).}}
\label{fig:mesh-gaussiana}
\end{figure}

\begin{figure}
\begin{center}
\includegraphics[scale=0.275]{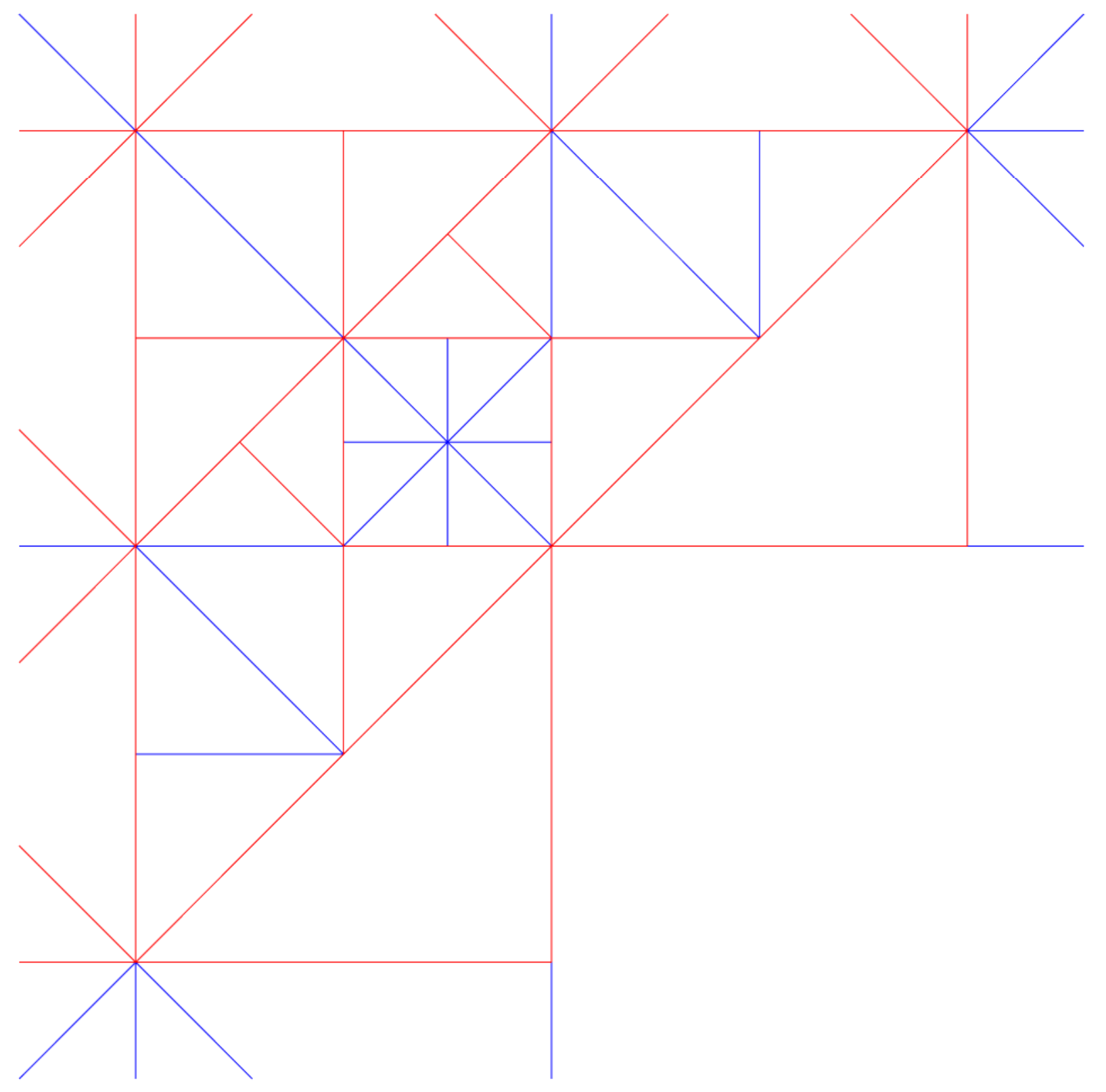}
\quad
\includegraphics[scale=0.275]{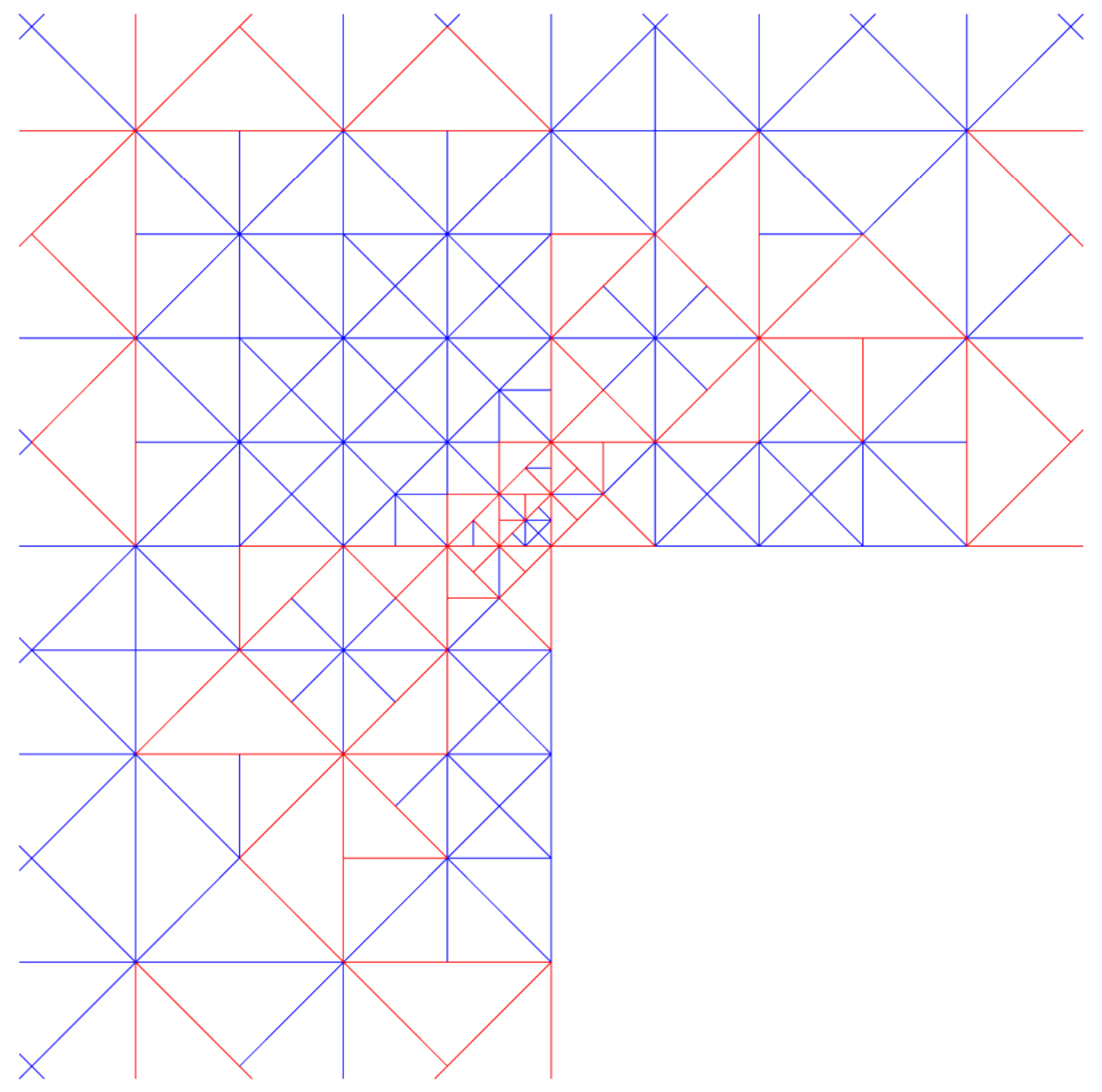}
\quad
\includegraphics[scale=0.275]{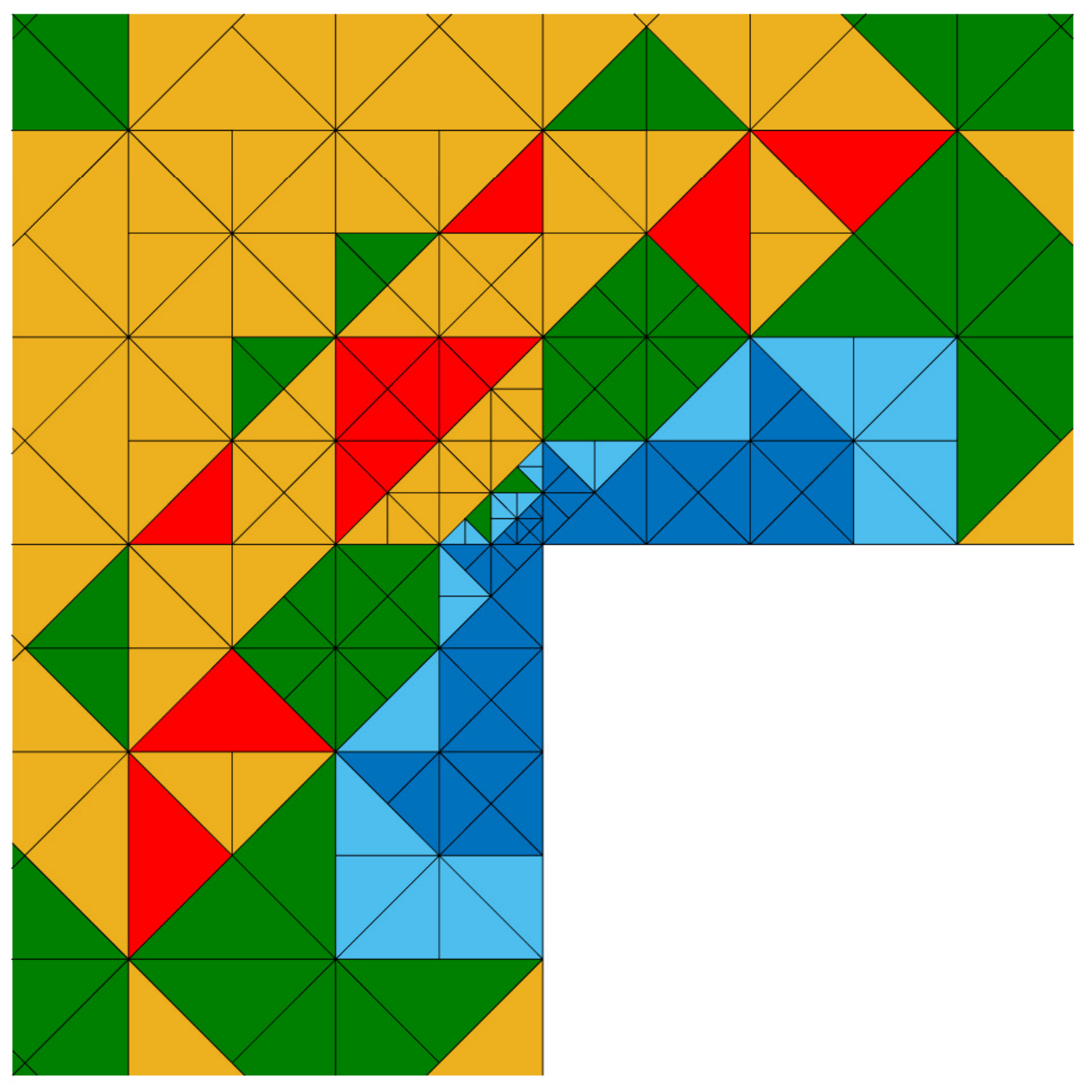}
\end{center}
\caption{ {Zoom to $(-10^{-2}, 10^{-2})^2$ to examine the origin.
Left: final grid  $\widehat\mesh_K$ generated by \texttt{DATA}.
Middle: final grid  $\mesh_{K+1}$ generated by \texttt{GALERKIN}. 
Elements having more than three vertices (polygons) are drawn in red.
Right: heat map representing for each $E \in \mesh_{K+1}$ the number of newest-vertex bisection needed to generate $E$ starting from the mesh $\widehat\mesh_K$ (colorbar in Fig. \ref{fig:colorbar}).}}
\label{fig:mesh-origine}
\end{figure}

\section{$\Lambda$-admissibility} \label{sec:admissible}

Our theory of {\tt AVEM} relies on the $\Lambda$-admissibility condition in Definition \ref{def:Lambda-partitions}.
Hereafter, we establish two results related to this concept: in Sect. \ref{sec:admissible-construction}, we show how to preserve $\Lambda$-admissibility while refining a mesh, whereas in Sect. \ref{sec:admissible-overlay} we prove that the overlay of two $\Lambda$-admissible meshes remains $\Lambda$-admissible.

\subsection{ {$\Lambda$-admissible mesh refinement}} \label{sec:admissible-construction}

In this section we introduce a constructive procedure that enforces $\Lambda$-admissibility at every stage of
{\tt AVEM} and study its complexity.
If $\mesh$ is a $\Lambda$-admissible refinement of $\mesh_0$ by newest-vertex bisection, the {\em level} of
an element $E \in \mesh$, denoted by $\ell(E)$, is the number of successive bisections
needed to generate $E$ from $\mesh_0$. Given $E \in \mesh$ marked for refinement, the procedure
$$
[\mesh_*] = \texttt{CREATE\_ADMISSIBLE\_CHAIN}(\mesh, E, \Lambda)
$$
generates a $\Lambda$-admissible refinement $\mesh_*$ of $\mesh$ upon bisecting $E$ and at most $\ell(E)+1$ other elements. To describe and analyze this procedure, we need some auxiliary notation and results.

Given any $E \in \mesh$, let us denote its newest vertex by $\bm{nv}(E)$, the edge opposite to $\bm{nv}(E)$ by ${\bm{oe}}(E)$, and the midpoint of ${\bm{oe}}(E)$ by $\bm{moe}(E)$. Furthermore, two elements $E', E'' \in \mesh$ are said {\em adjacent} if $e=E'\cap E''$ is an edge for at least one element, and are said {\em compatible} if they are adjacent and neither $\bm{nv}(E')$ nor $\bm{nv}(E'')$ belong to the line containing $e$ (see Fig. \ref{fig:compatible}, cases A and B).

\begin{figure}[h!]
\begin{center}
\begin{overpic}[scale=0.25]{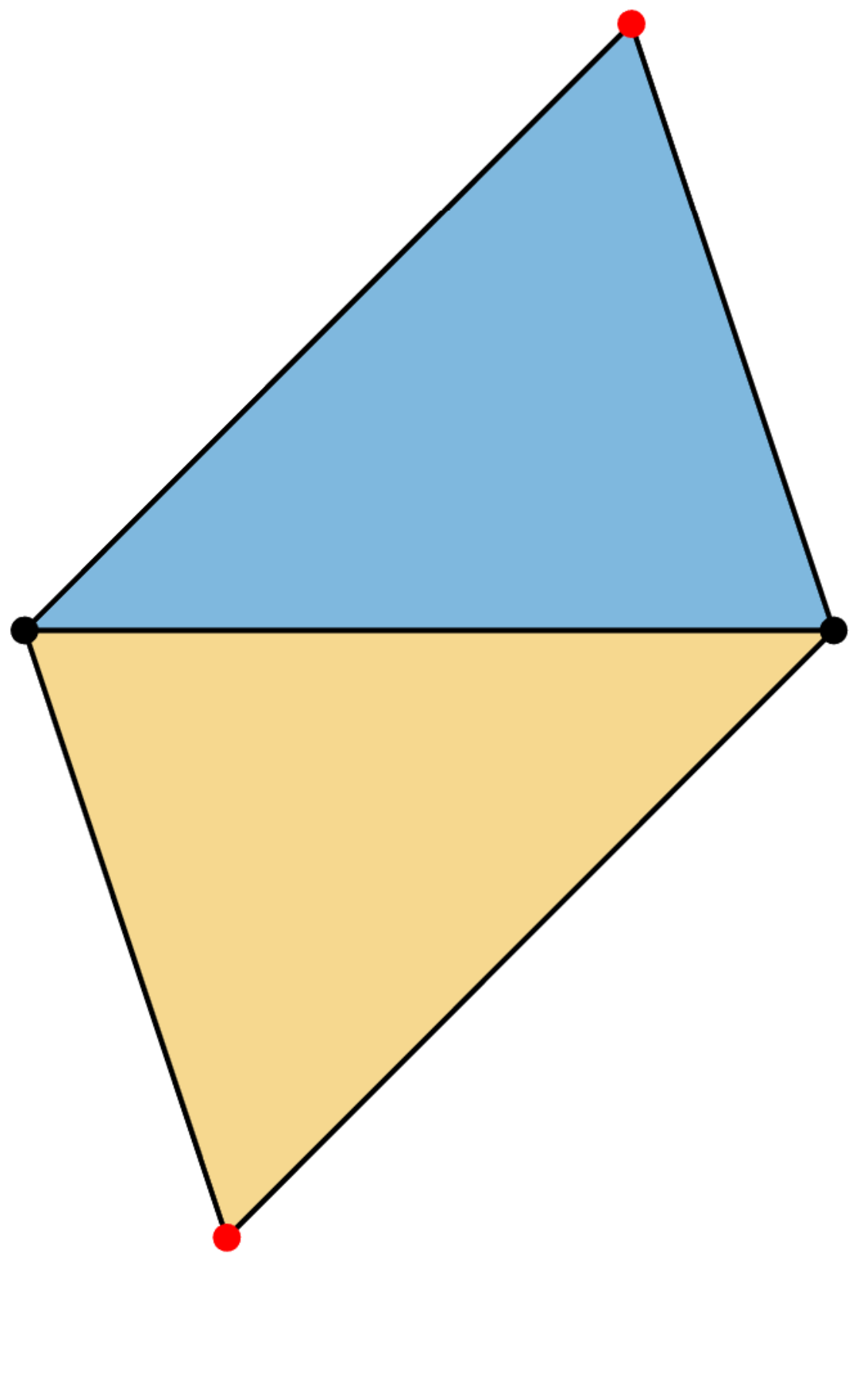}
\put(24,40){$E''$}
\put(30,70){$E'$}
\put(20,10){$\bm{nv}(E'')$}
\put(50,95){$\bm{nv}(E')$}
\put(20,0){\texttt{case A}}
\end{overpic}
\qquad \quad
\begin{overpic}[scale=0.25]{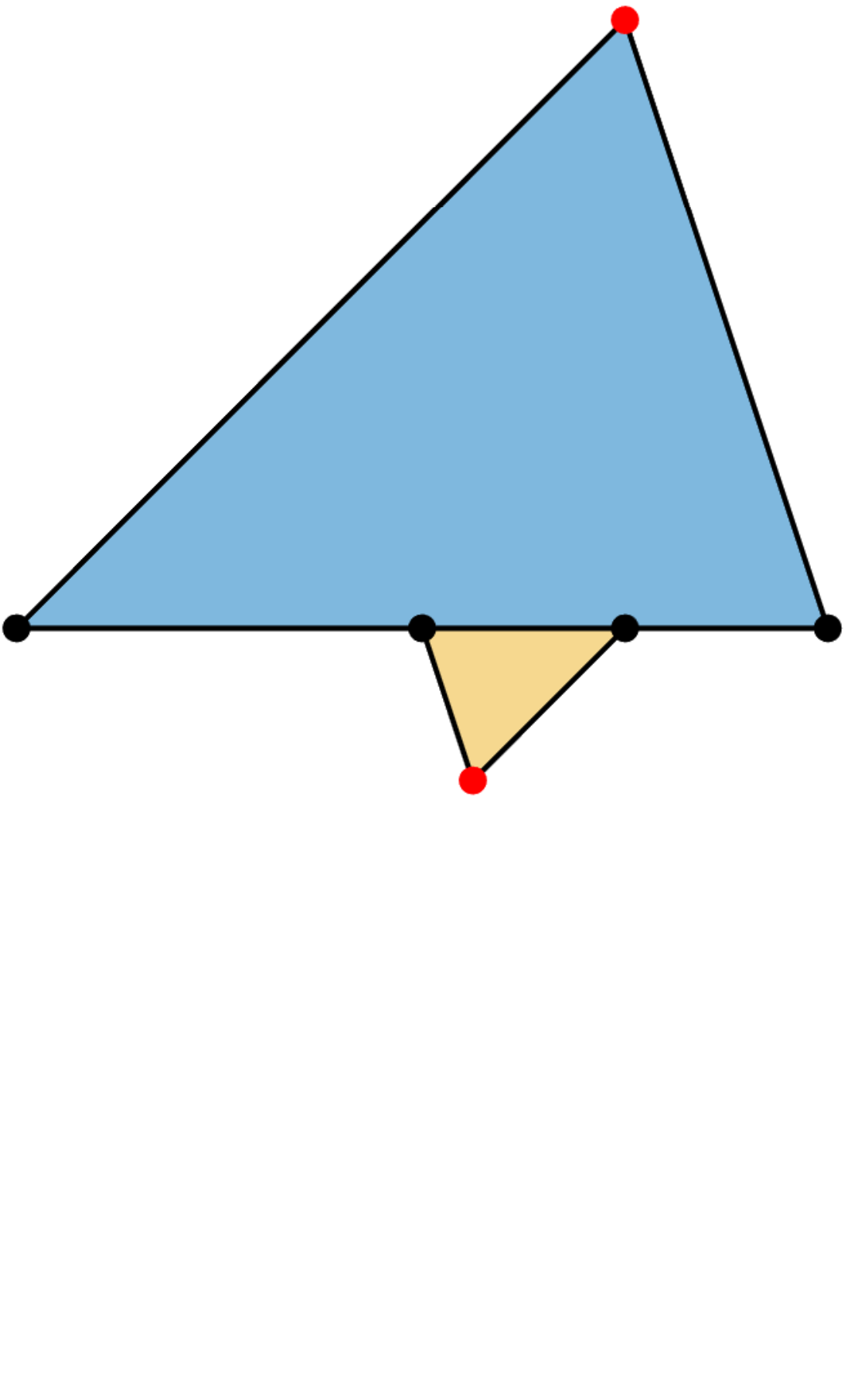}
\put(21,47){$E''$}
\put(30,70){$E'$}
\put(38,42){$\bm{nv}(E'')$}
\put(50,95){$\bm{nv}(E')$}
\put(20,0){\texttt{case B}}
\end{overpic}
\qquad \quad
\begin{overpic}[scale=0.25]{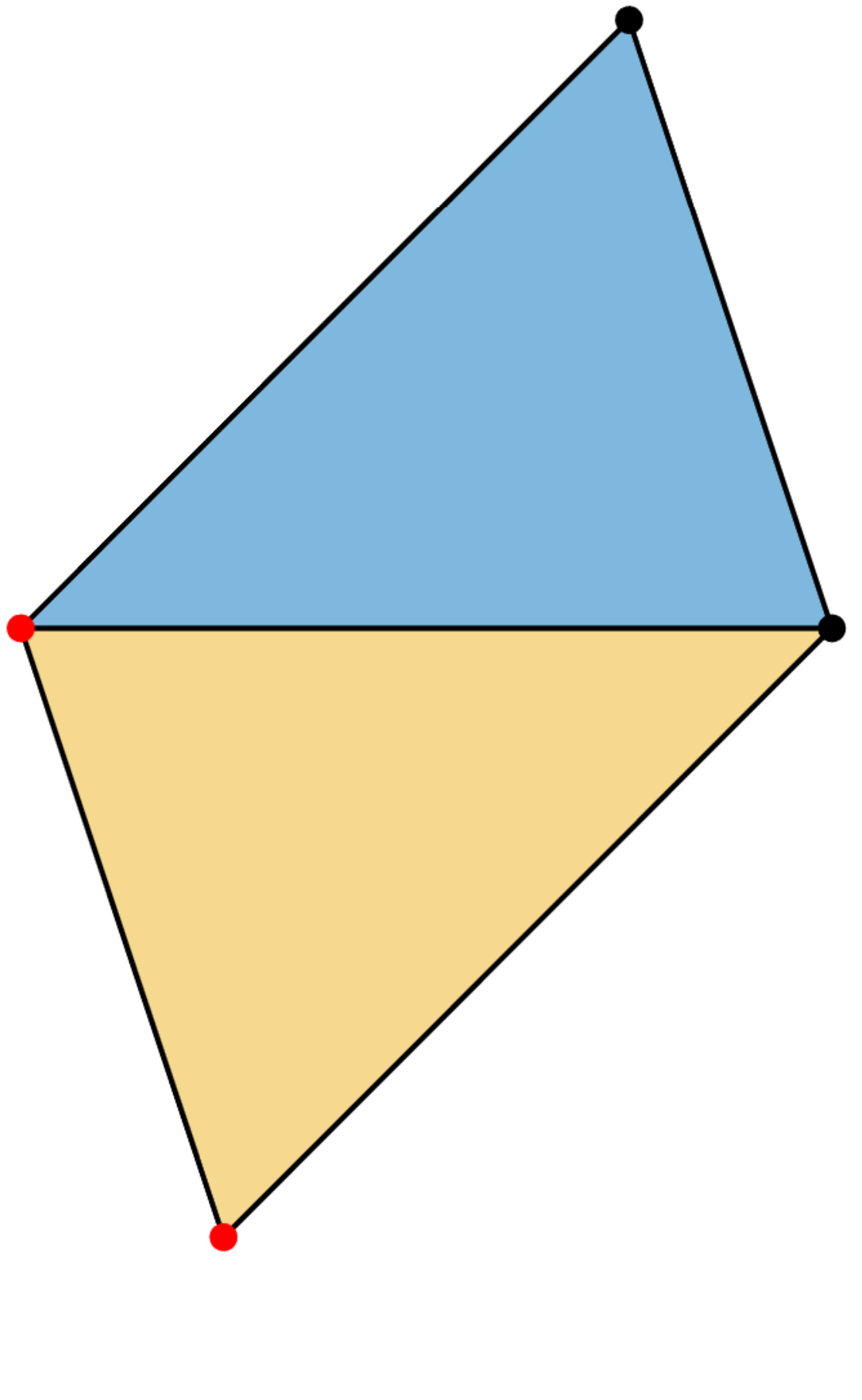}
\put(24,40){$E''$}
\put(30,70){$E'$}
\put(20,10){$\bm{nv}(E'')$}
\put(-17,62){$\bm{nv}(E')$}
\put(20,0){\texttt{case C}}
\end{overpic}
\qquad \quad
\begin{overpic}[scale=0.25]{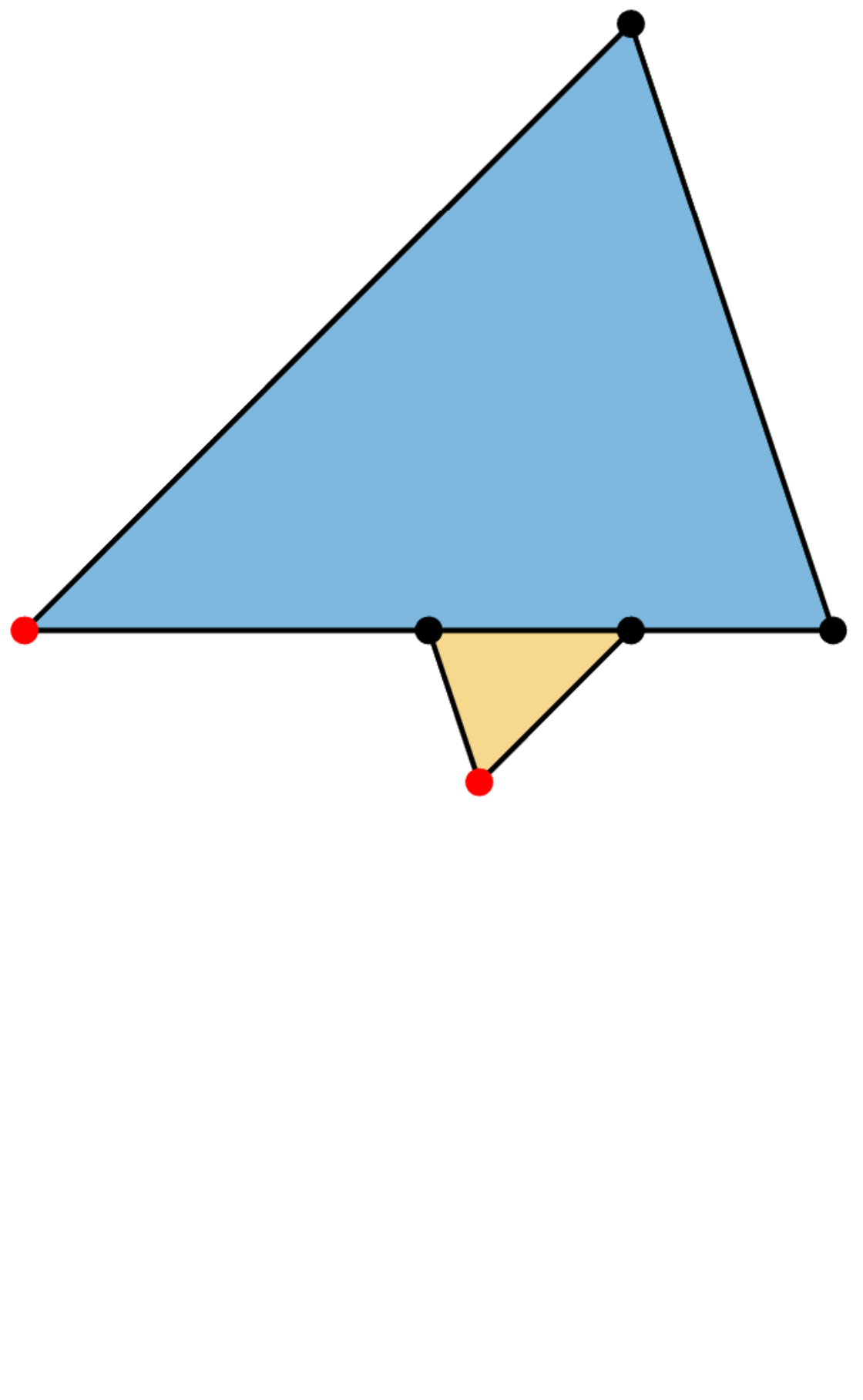}
\put(21,47){$E''$}
\put(30,70){$E'$}
\put(38,42){$\bm{nv}(E'')$}
\put(-17,62){$\bm{nv}(E')$}
\put(20,0){\texttt{case D}}
\end{overpic}
\end{center}
\caption{ {The elements $E'$ and $E''$ are adjacent in cases A to D. They are compatible in cases A and B,
    and non-compatible in cases C and D.}}
\label{fig:compatible}
\end{figure}
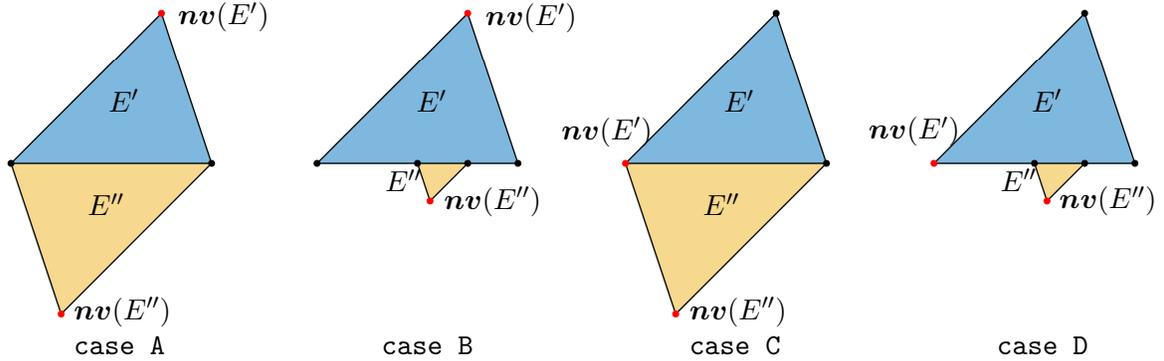

Denote by $\mathbb{T}$ the infinite tree obtained by successive bisections of the root partition $\mesh_0$. The following result is  {well-known \cite{BDD:04,BonitoNochetto:10,NSV:09,NochettoVeeser:12,Stevenson:08}.}

 {
\begin{lemma}[levels of elements sharing a full edge]\label{lem:levels}
Assume that $E, E' \in \mathbb{T}$ share a full edge $e=E\cap E'$. Then
$$
| \ell(E)-\ell(E') | \leq 1\,.
$$
\end{lemma}
\begin{proof}
 If neither $\bm{nv}(E)$ nor $\bm{nv}(E')$ belong to $e$, or both $\bm{nv}(E)$ and $\bm{nv}(E')$ belong to $e$, then $\ell(E)=\ell(E')$. On the other hand, if $\bm{nv}(E) \in e$ but $\bm{nv}(E') \not \in e$, then $\ell(E') = \ell(E)+1$, since $E'$ is generated by bisecting an element $\widetilde{E}$ of the same level as $E$.
\end{proof}
}

\begin{lemma}[global index of a hanging node]\label{lem:global-index}
Consider an edge $[\bm{x}',\bm{x}'']$ of the partition $\mesh$. If $\bm{x} \in {\cal H} \cap {\rm int\, }e$ is generated by $m \geq 1$ bisections of $e$, then its global index $\lambda(\bm{x})$ satisfies
$$
\lambda(\bm{x})= \max (\lambda(\bm{x}'), \lambda(\bm{x}'')) +m \,.
$$
\end{lemma}
\begin{proof} If $m=1$, $\bm{x}=\bm{x}_M$ is the midpoint of $e$, and the formula is just the  {Definition \ref{def:node-index}} of global index. If $m >1$, then $\bm{x}$ is generated by bisecting some interval $[\bm{z}',\bm{z}''] \subset e$, and $\lambda(\bm{x})=\max (\lambda(\bm{z}'), \lambda(\bm{z}''))+1$. Exactly one between $\bm{z}',\bm{z}''$ has been generated by $m-1$ bisections, whereas the other one has been
generated by less than $m-1$ bisections. Hence, one concludes by induction.  \end{proof}

\begin{proposition}[reducing the global index of hanging nodes]\label{prop:reduce-globindex}
Let ${\cal H} \cap {\rm int\, }e$ contain at least the midpoint $\bm{x}_M$ of $e$. Assume that a bisection of some element in $\mesh$ transforms $\bm{x}_M$ into a proper node, and let $\lambda_{\rm new}$ denote the new global-index mapping of the nodes in ${\cal H} \cap {\rm int\, }e$ after the bisection. Then there holds
$$
\lambda_{\rm new}(\bm{x}) \leq \lambda(\bm{x}) -1 \qquad \forall \bm{x} \in {\cal H} \cap {\rm int\, }e \,.
$$
\end{proposition}
\begin{proof}
If $\bm{x}=\bm{x}_M$, then trivially $\lambda_{\rm new}(\bm{x}) = 0 \leq \lambda(\bm{x}) -1$. If $\bm{x} \in {\cal H} \cap {\rm int\, }e$ is contained, say, in $(\bm{x}',\bm{x}_M)$ and has been generated by $m>1$ successive bisections of $e$, then it is generated by $m-1$ successive bisections of $[\bm{x}',\bm{x}_M]$. Thus, by Lemma \ref{lem:global-index}
\begin{eqnarray*}
\lambda_{\rm new}(\bm{x}) &\leq& \max(\lambda_{\rm new}(\bm{x}'), \lambda_{\rm new}(\bm{x}_M) ) + m-1 \\
&=& \max(\lambda(\bm{x}'), 0 ) + m-1 \ = \ \lambda(\bm{x}') + m-1 \\
&\leq& \max((\lambda(\bm{x}'), \lambda(\bm{x}''))+m-1 \ = \ \lambda(\bm{x})-1\,.  \qquad \qquad 
\end{eqnarray*} 
 {This gives the desired estimate.}
\end{proof}

The result just established is the motivation for the proposed refinement strategy. Indeed, it assures that in order to reduce the global index of a hanging node sitting on an edge, it is enough to transform the midpoint of the edge into a proper node.

The following remark will be useful in the sequel.
\begin{remark}[facing element]\label{rem:facing}
{\rm
 {Given a $\Lambda$-admissible mesh $\mesh$ and $E \in \mesh$,} let $\bm{x}=\bm{moe}(E)$ and suppose that $\lambda(\bm{x}) > \Lambda$. Then $\bm{x}$ is not a node of $\mesh$,  {whence the edge ${\bm{oe}}(E)$} cannot contain any hanging node in its interior.  We conclude that there exists a unique adjacent element $\widetilde{E} \in \mesh$, $\widetilde{E} \not = E$, such that  {$E \cap \widetilde{E} = {\bm{oe}}(E)$.} This element will be called the element {\em facing} $E$.
}
\end{remark}

\medskip

Given an element $E \in \mesh$ which has been marked for refinement, we are ready to identify those elements  {in $\mesh$} that need be bisected with $E$ in order to create a $\Lambda$-admissible refinement of $\mesh$.

\begin{definition}[chain of elements to be refined]\label{def:ref-chain}
Define by recurrence the chain of elements 
$$
\RC(E)=\{E_0, E_1, \dots, E_K\}
$$
for some $K \geq 0$, as follows:  {set first $E_0=E$ and, assuming to have defined $E_k$ for $k \geq 0$, then}
\begin{enumerate}[(i)]
\item if $\lambda(\bm{moe}(E_k)) \leq \Lambda$, set $K=k$ and stop;
\item if $\lambda(\bm{moe}(E_k)) = \Lambda+1$ and the facing element $\widetilde{E}_{k}$ is compatible with $E_k$, set $E_{k+1}=\widetilde{E}_{k}$, $K=k+1$ and stop;
\item if $\lambda(\bm{moe}(E_k)) = \Lambda+1$ and the facing element $\widetilde{E}_{k}$ is not compatible with $E_k$, set $E_{k+1}=\widetilde{E}_{k}$ and continue.
\end{enumerate}

\end{definition}

\begin{lemma}[properties of the chain of refinement]\label{L:chain}
 {The chain $\RC(E)$ has at most $K \leq \ell(E)+1$ elements.
Furthermore, the sequence of element levels $\{ \ell(E_k) \}_{k = 0}^K$} is not increasing.
\end{lemma}
\begin{proof}
 {We claim that step (iii) in Definition \ref{def:ref-chain} reduces the level by at least one. In fact, $E_k$ coincides with or is a refinement of a triangle $E \in \mathbb{T}$ sharing with $E_{k+1}$ a full edge; thus $\ell(E_k)\ge\ell(E)$. Such triangle $E$ satisfies $\ell(E) = \ell(E_{k+1})+1$ according to Lemma \ref{lem:levels}, whence 
\begin{equation}\label{eq:bound-levels}
\ell(E_{k+1}) = \ell(E) -1 \leq \ell(E_k)-1.
\end{equation}
Therefore, for as long as case (iii) is active, i.e. for all $j<K$, we have $\ell(E_j) \leq \ell(E_0) - j$ and
$$
0 \leq \ell(E_{K-1}) \leq \ell(E_0)-(K-1) \, ,
$$
which gives the first part of the Lemma. The monotonicity of $\{ \ell(E_k) \}_{k = 0}^K$
follows from \eqref{eq:bound-levels} and the fact that $\ell(E_{K-1})=\ell(E_K)$ in case (ii).
}
\end{proof}

\medskip
We are now ready to define the procedure 
$$
[\mesh_*] = \texttt{CREATE\_ADMISSIBLE\_CHAIN}(\mesh, E, \Lambda)
$$
The partition $\mesh_*$ is obtained from $\mesh$ by refining only the elements in $\RC(E)$.
 {More precisely, starting from $E_K$, one goes traverses the chain backwards and, for $K \geq k \geq 1$,
considers the cases}
\begin{itemize}
\item if $E_k$ and $E_{k-1}$ are compatible, then $E_k$ is bisected once (see Fig. \ref{fig:catena}, cases A or B);
\item if $E_k$ and $E_{k-1}$ are not compatible, then $E_k$ is bisected twice  {and, after the first bisection,} the sibling that is facing $E_{k-1}$ is further bisected (see Fig. \ref{fig:catena}, cases C or D);
\item finally, $E_0=E$ is bisected once.
\end{itemize}

\begin{figure}[h!]
\begin{center}
\begin{overpic}[scale=0.25]{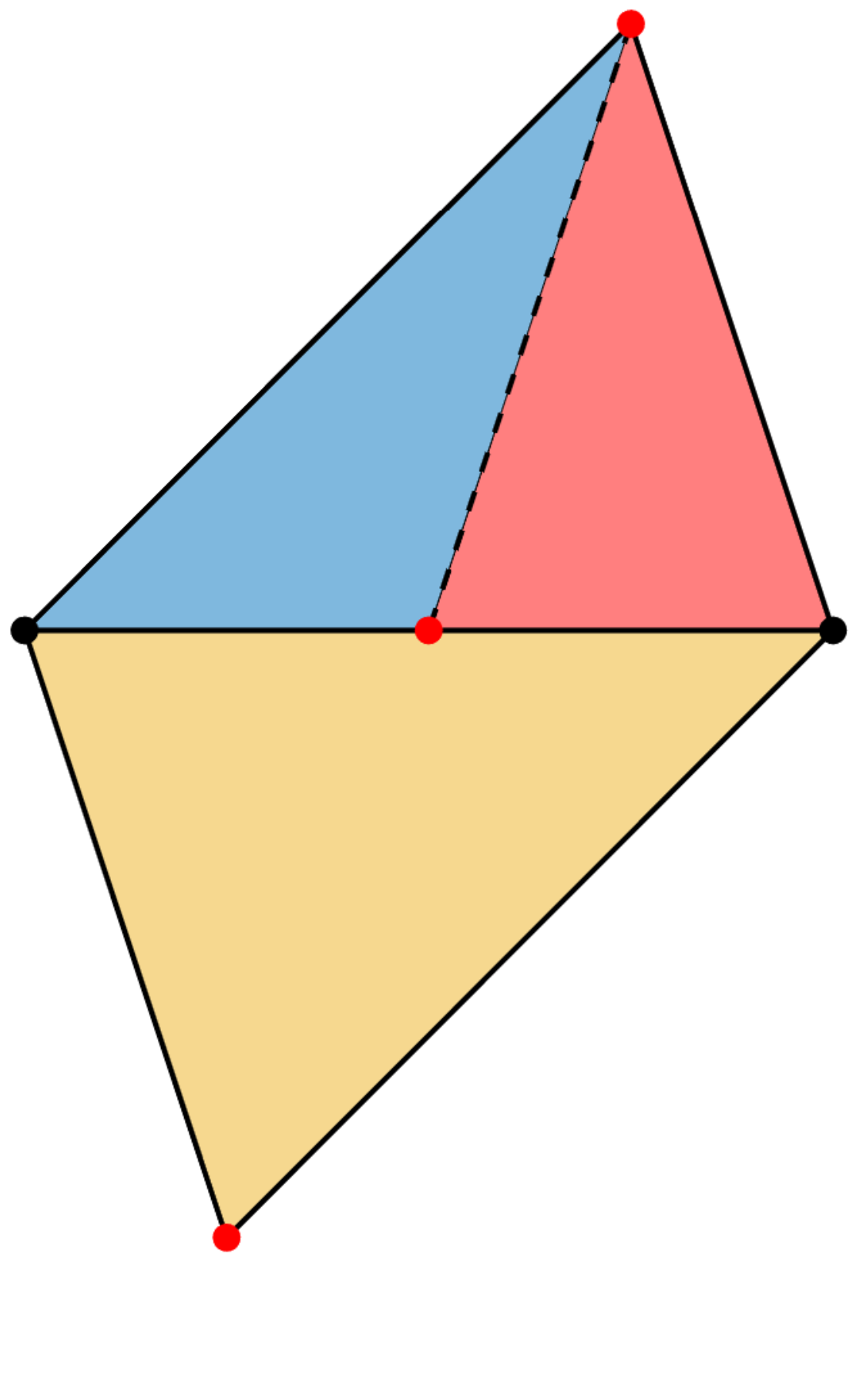}
\put(24,40){$E_{k-1}$}
\put(30,70){$E_k$}
\put(20,10){$\bm{nv}(E_{k-1})$}
\put(50,95){$\bm{nv}(E_k)$}
\put(20,0){\texttt{case A}}
\end{overpic}
\qquad \quad
\begin{overpic}[scale=0.25]{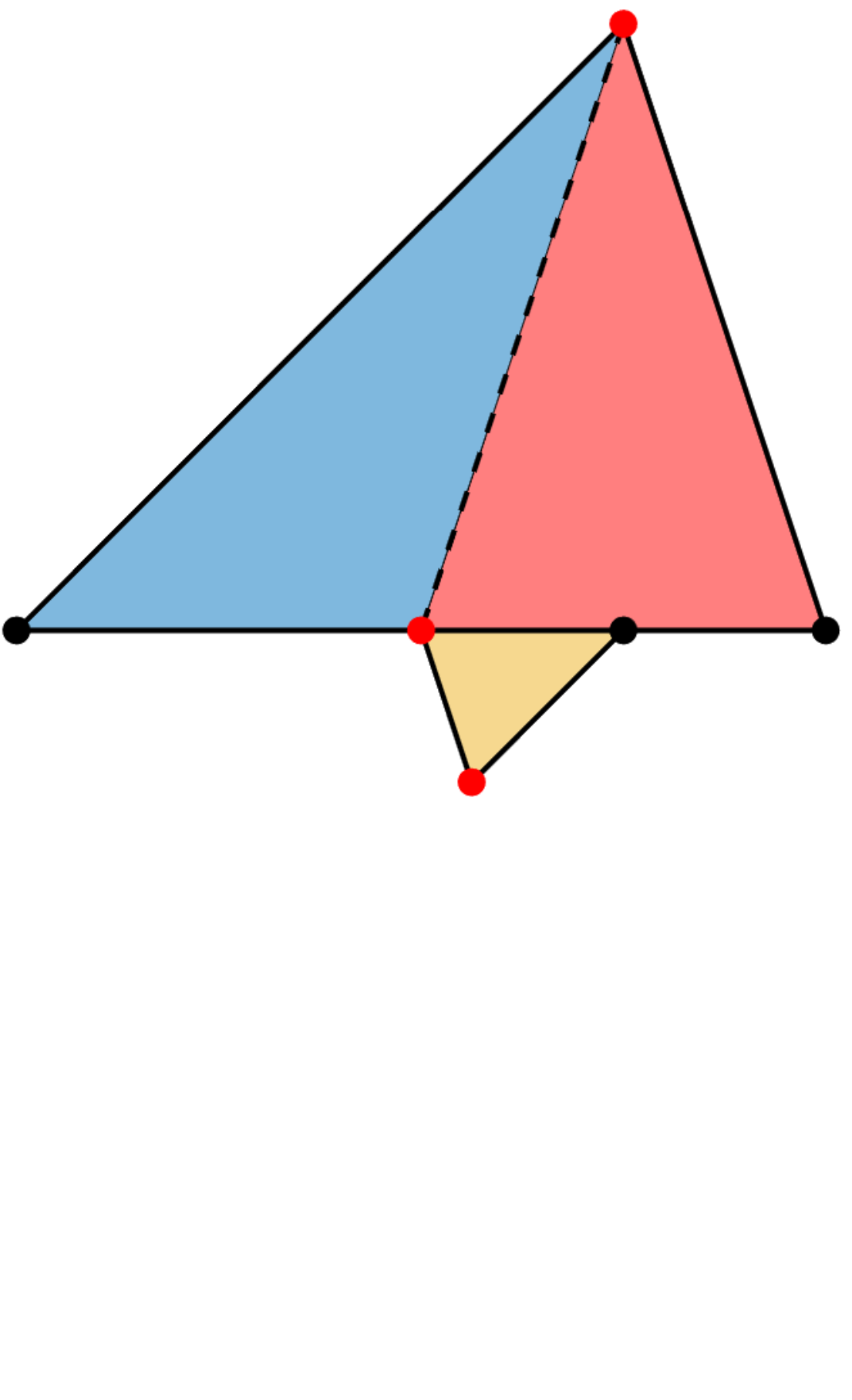}
\put(15,47){$E_{k-1}$}
\put(30,70){$E_k$}
\put(38,42){$\bm{nv}(E_{k-1})$}
\put(50,95){$\bm{nv}(E_k)$}
\put(20,0){\texttt{case B}}
\end{overpic}
\qquad \quad
\begin{overpic}[scale=0.25]{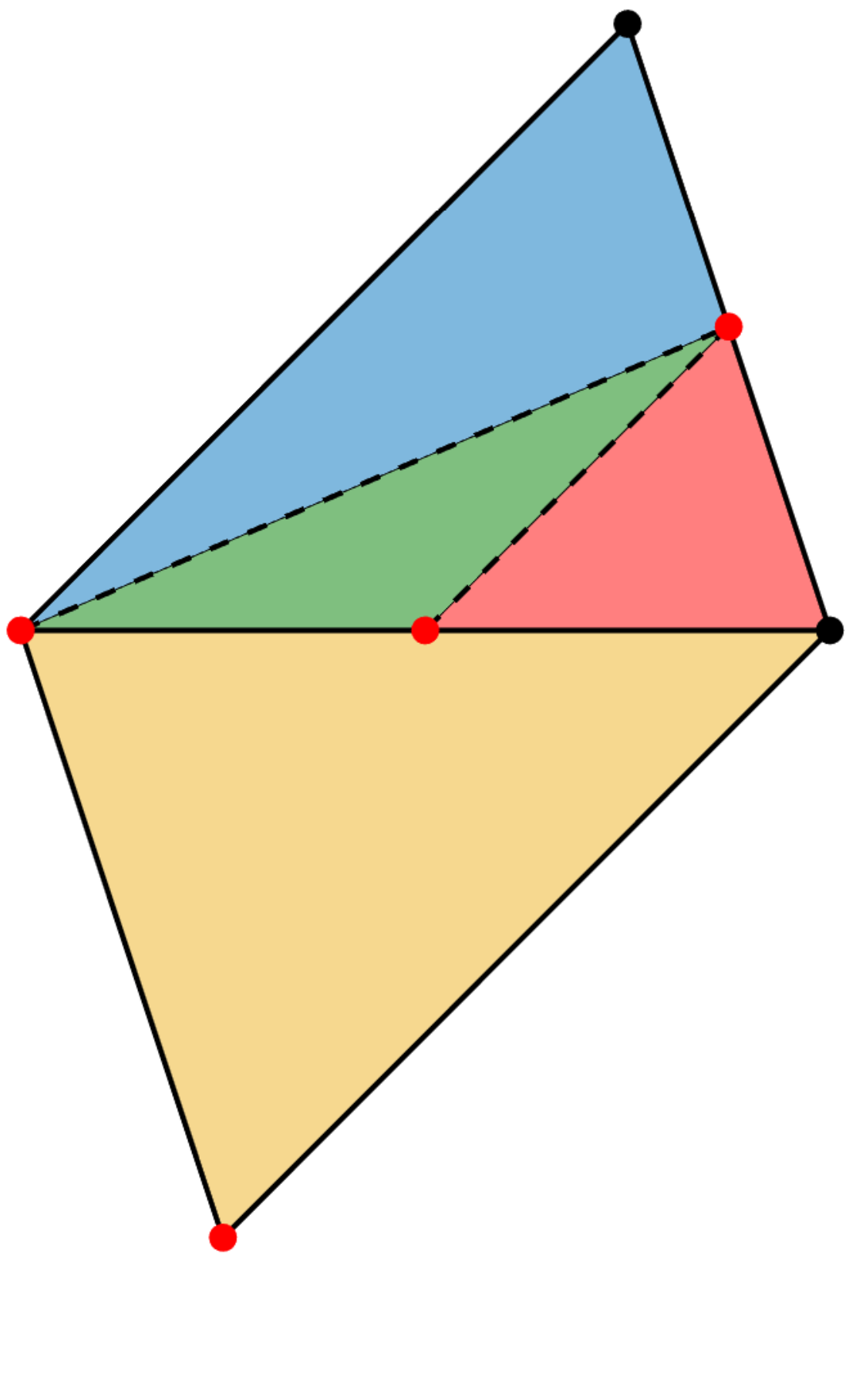}
\put(24,40){$E_{k-1}$}
\put(30,70){$E_k$}
\put(20,10){$\bm{nv}(E_{k-1})$}
\put(-18,62){$\bm{nv}(E_k)$}
\put(20,0){\texttt{case C}}
\end{overpic}
\qquad \quad
\begin{overpic}[scale=0.25]{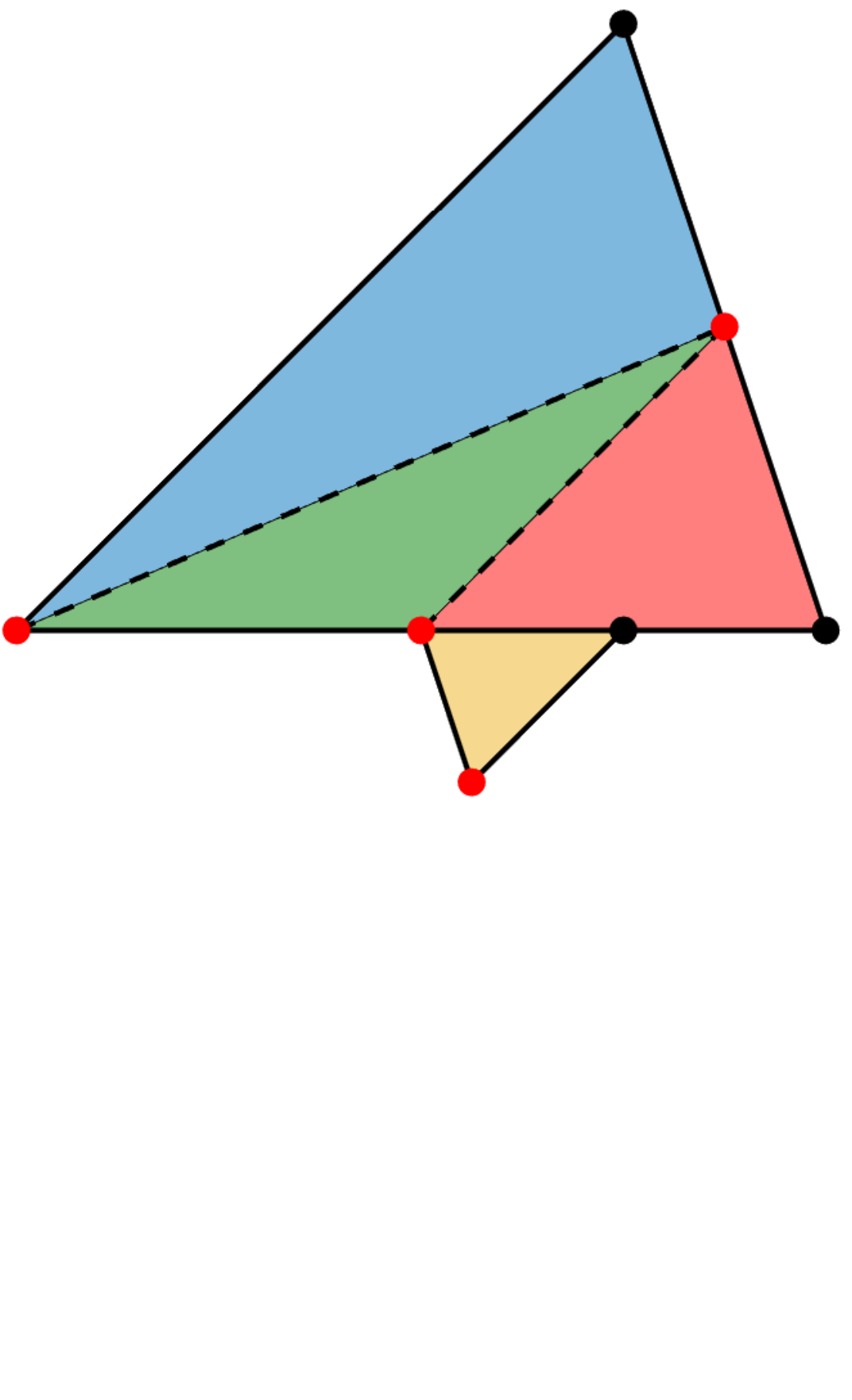}
\put(15,47){$E_{k-1}$}
\put(30,70){$E_k$}
\put(38,42){$\bm{nv}(E_{k-1})$}
\put(-18,62){$\bm{nv}(E_k)$}
\put(20,0){\texttt{case D}}
\end{overpic}
\end{center}
\caption{Two elements $E_{k-1}$ and $E_k$ in the chain $\RC(E)$: $E_{k-1}$ can be bisected in a $\Lambda$-admissible way, only after $E_k$ is refined once (cases A and B), or twice (cases C and D)}
\label{fig:catena}
\end{figure}
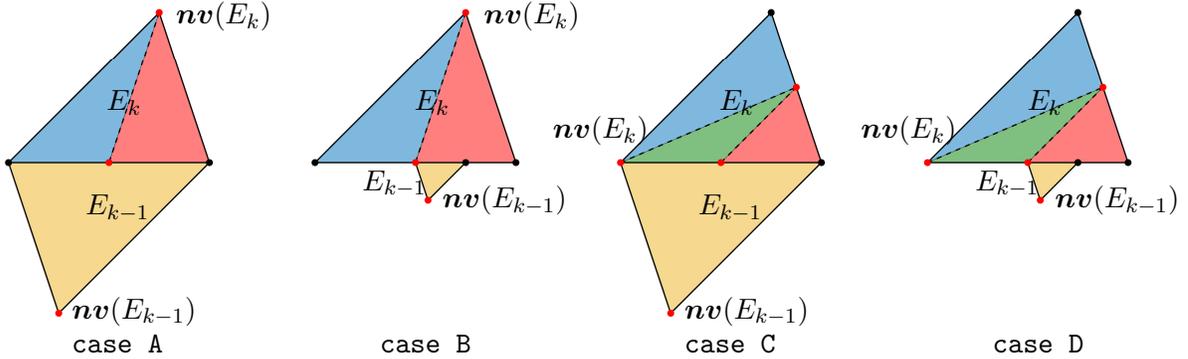

 {
\begin{proposition}[properties of \texttt{CREATE\_ADMISSIBLE\_CHAIN}]\label{prop:level-bound}
   {If $\mesh$ is $\Lambda$-admissible, then the call $[\mesh_*] =$ {\tt CREATE\_ADMISSIBLE\_CHAIN} $(\mesh,E, \Lambda)$
    bisects once or twice the elements of the chain $\RC(E)$, whose cardinality is at most $\ell(E)+1$,
    and produces a $\Lambda$-admissible mesh $\mesh_*$ with $E$ bisected once.
    Moreover, every element $E'\in\mesh_*$ generated by this call satisfies}
\begin{equation}\label{eq:levels-chain}
\ell(E') \leq \ell(E)+1 \,.
\end{equation}
\end{proposition}
\begin{proof}
Let $\RC(E) = \{E_k\}_{k=0}^K$ and
observe that, for $k \geq 1$, one or two bisections of $E_k$ convert the midpoint of the edge $e$ of $E_k$
shared with $E_{k-1}$ into a proper node. Therefore, Proposition \ref{prop:reduce-globindex} (reducing the
global index of hanging nodes) implies that
the global indices of all interior nodes to $e$ decrease by at least $1$, and makes the
bisection of $E_{k-1}$ $\Lambda$-admissible as desired. This procedure creates $\mesh_*$ upon partitioning
at most $\ell(E)+1$ elements, namely those of $\RC(E)$, according to Lemma \ref{L:chain} \
(properties of the chain refinement).
  
Moreover, to prove \eqref{eq:levels-chain} we take $k \geq 1$ and consider the following two mutually
exclusive cases.
If $E_k$ and $E_{k-1}$ are compatible, then $E_k$ is replaced by two elements $E'\in\mesh_*$ of level
\[
\ell(E') = \ell(E_k)+1 \leq  \ell(E)+1,
\]
according to Lemma \ref{L:chain}. On the other hand, if $E_k$ and $E_{k-1}$ are not compatible, then $E_k$ is replaced by one element of level $\ell(E_k)+1$ and two elements $E'\in\mesh_*$ of level
\[
\ell(E') = \ell(E_k)+2 \leq\ell(E_{k-1})+1 \leq \ell(E)+1
\]
because of \eqref{eq:bound-levels}. Finally, the element $E_0=E$ is replaced by two elements of level $\ell(E)+1$.
\end{proof}

In view of Proposition \ref{prop:level-bound} a bound of the form $\#\mesh_* - \#\mesh \leq C_0$ with
a universal constant $C_0$
is false because $C_0$ may depend on $\ell(E)$  in general. This obstruction to optimal
complexity of {\tt REFINE} was
tackled by Binev, Dahmen and DeVore in their seminal paper \cite{BDD:04}, and further studied in
\cite{BonitoNochetto:10,NSV:09,NochettoVeeser:12,Stevenson:08}. In fact, the cumulative effect
of bisection on conforming meshes obeys the weaker, but yet optimal, equation \eqref{eq:complexity-REFINE}.
The extension of  this to $\Lambda$-admissible non-conforming partitions is precisely guaranteed by the stated Theorem \ref{T:complexity-REFINE}, whose proof follows.

\medskip
{\it Proof of Theorem \ref{T:complexity-REFINE} (complexity of {\tt REFINE})}. We follow
\cite[Section 6.3]{NochettoVeeser:12}, which explains the basic ingredients to derive 
\eqref{eq:complexity-REFINE}. It turns out that two crucial properties of \texttt{CREATE\_ADMISSIBLE\_CHAIN} as
required. The first is \eqref{eq:levels-chain}. The second one relates the level of elements and their
distance to $E$, namely
\[
\textrm{dist }(E,E') \leq C 2^{\frac{\ell(E')}{2}} \quad\forall \, E'\in\mesh_*\backslash\mesh \,;
\]
such property is valid for bisection grids regardless of $\Lambda$-admissibility \cite[Lemma 18]{NochettoVeeser:12}.
This completes the proof. $\square$

}

\subsection{Mesh Overlay and $\Lambda$-admissibility}\label{sec:admissible-overlay}
Given two partitions $\mesh_A$ and $\mesh_B$,   denote by $\mesh_A \oplus \mesh_B$  the {\em overlay} of $\mesh_A$ and $\mesh_B$, i.e., the partition whose associated tree is the union of the trees of $\mesh_A$ and $\mesh_B$.  The following property holds.
 \begin{proposition}
 If $\mesh_A$ and $\mesh_B$ are $\Lambda$-admissible, then $\mesh_A \oplus \mesh_B$ remains $\Lambda$-admissible.
 \end{proposition}
 \begin{proof}
 Denote here by ${\cal N}$ the set of all nodes obtained by newest-vertex bisection from the root partition $\mesh_0$. Let ${\cal N}_0$, ${\cal N}_A$, ${\cal N}_B$, ${\cal N}_{A + B}$, resp., be the set of nodes of the partitions  $\mesh_0$, $\mesh_A$, $\mesh_B$,  $\mesh_A \oplus \mesh_B$, resp.. It is easily seen that for each $\bm{x} \in {\cal N}\setminus {\cal N}_0$ there exists a unique ${\cal B}(\bm{x}) = \{\bm{x}', \bm{x}''\} \subset {\cal N}$ such that $\bm{x}$ is generated by the bisection of the segment $[\bm{x}', \bm{x}'']$. Furthermore, if $\bm{x} \in {\cal N}_{A+B}$ is a proper node of $\mesh_A$ (of $\mesh_B$, resp.), then it is also a proper node of $\mesh_A \oplus \mesh_B$.
 
 Let us denote by $\lambda_A$, $\lambda_B$, $\lambda_{A+B}$, resp., the global-index mappings defined on ${\cal N}_A$, ${\cal N}_B$, ${\cal N}_{A + B}$, resp.. It is convenient to extend the definition of $\lambda_A$ and $\lambda_B$ to the whole ${\cal N}_{A + B}$ by setting 
$$
\lambda_A(\bm{x}) = +\infty \quad \text{if } \bm{x} \in {\cal N}_{A + B}\setminus {\cal N}_A \,, \qquad 
\lambda_B(\bm{x}) = +\infty \quad \text{if } \bm{x} \in {\cal N}_{A + B}\setminus {\cal N}_B \,. 
$$
With these notations at hand, we are going to prove the inequality
\begin{equation}\label{eq:overlay-1}
\lambda_{A+B}(\bm{x}) \leq \min ( \lambda_A(\bm{x}),  \lambda_B(\bm{x}) ) \qquad \forall \bm{x} \in {\cal N}_{A + B}\,,
\end{equation} 
from which the thesis immediately follows.

We proceed by induction on $k=\lambda_{A+B}(\bm{x})$, $\bm{x} \in {\cal N}_{A + B}$. If $k=0$, the inequality is trivial since $\lambda_A(\bm{x}),  \lambda_B(\bm{x}) \geq 0$. So suppose \eqref{eq:overlay-1} hold up to some $k\geq 0$. If $\bm{x} \in {\cal N}_{A + B}$ satisfies $\lambda_{A+B}(\bm{x})= k+1 >0$, then it is a hanging node of $\mesh_A \oplus \mesh_B$ by definition of global index, hence, it is a hanging node of $\mesh_A$ or $\mesh_B$; wlog, suppose it is a hanging node of $\mesh_A$.
If $\bm{x}$ is generated by the bisection of the segment $[\bm{x}', \bm{x}'']$, then again by definition of global index it holds
$$
k+1 = \lambda_{A+B}(\bm{x}) = \max( \lambda_{A+B}(\bm{x'}), \lambda_{A+B}(\bm{x}'') ) + 1 \,,
$$ 
which implies 
$$
\lambda_{A+B}(\bm{x'}) \leq k \,, \qquad  \lambda_{A+B}(\bm{x''}) \leq k \,.
$$
By induction,
$$
\lambda_{A+B}(\bm{x'}) \leq \min ( \lambda_A(\bm{x'}),  \lambda_B(\bm{x'}) ) \,, \qquad \lambda_{A+B}(\bm{x''}) \leq \min ( \lambda_A(\bm{x''}),  \lambda_B(\bm{x''}) ) \,, 
$$
from which we obtain
$$
\lambda_{A+B}(\bm{x}) \leq \max(\lambda_A(\bm{x'}), \lambda_A(\bm{x''}) ) +1 = \lambda_A(\bm{x})
$$ 
since $\bm{x}$ is a hanging node of $\mesh_A$. On the other hand, either $\bm{x} \in {\cal N}_B$ or $\bm{x} \not \in {\cal N}_B$. In the latter case, $\lambda_B(\bm{x}) = +\infty$, and \eqref{eq:overlay-1} is proven. In the former case, necessarily $\bm{x}$ is a hanging node of $\mesh_B$, hence as above
$$
\lambda_{A+B}(\bm{x}) \leq \max(\lambda_B(\bm{x'}), \lambda_B(\bm{x''}) ) +1 = \lambda_B(\bm{x}) \,,
$$
and the thesis  is proven.    
\end{proof}

\section{Conclusions}\label{S:conclusions}

This paper introduces and studies a two-step adaptive virtual element method ({\tt AVEM})
  of lowest order over triangular meshes with hanging nodes in
  2d,  which are treated as polygons. {\tt AVEM} applies to linear symmetric elliptic problems with
  variable data. The main achievements of the paper can be summarized as follows:
 \begin{enumerate}[$\bullet$]

 \item {\tt AVEM} concatenates two modules, {\tt DATA} and {\tt GALERKIN}.
 The former approximates data by piecewise constants to a desired accuracy, while the latter handles
 the adaptive approximation of the problem  
 with piecewise constant data, as described in \cite{paperA}. {\tt AVEM} converges (Proposition \ref{P:convergence-AVEM});   \looseness=-1
   
 \item {\it Complexity of {\tt{GALERKIN}}:} the number of sub-iterations inside the call to {\tt{GALERKIN}} at iteration $k$ of {\tt AVEM} is bounded independently of $k$ (Proposition  \ref{prop:compl_gal});
   
 \item {\it Complexity of {\tt{DATA}}:} the module {\tt{DATA}} is quasi-optimal in terms of accuracy versus mesh cardinality,  under  suitable regularity conditions on the data (Sect. \ref{sec:approx-data});
   
 \item {\it Complexity of {\tt AVEM}:} {\tt AVEM} is quasi-optimal in terms
   of error decay versus degrees of freedom, for solutions and data belonging to appropriate approximation classes (Theorem \ref{Texact:optimality-AVEM});
   
 \item {\it Numerical experiments}: they illustrate
   the interplay between the modules {\tt DATA} and {\tt GALERKIN} and provide computational evidence of the optimality of {\tt AVEM} (Section \ref{sec:experiments}).

 \item {\it Mesh admissibility}: Section \ref{sec:admissible} designs a procedure to keep
   the global index of meshes uniformly bounded for all steps $k$, and proves its optimality in terms
   of degrees of freedom.
   
\end{enumerate}

Although in Remark \ref{rem:avem-afem} we observed that, in the presence of a bound on the maximal index of hanging nodes, the equivalence classes of AVEM and AFEM are the same, the numerical results in Section \ref{sec:experiments} and in \cite{paperA} suggest that the flexibility of VEM may lead to more efficient meshes in complex situations, at least in terms of the involved constants. A deeper investigation of this aspect at the theoretical level may require a more advanced VEM approach, for instance taking inspiration from the a-priori analysis in \cite{BdV-Vacca:2022}.


\bigskip
\begin{center}
{\bf Acknowledgements}
\end{center}
\noindent
LBdV, CC and MV where partially supported by the Italian MIUR through the PRIN grants n. 201744KLJL and n.  20204LN5N5 (LBdV, MV) and n. 201752HKH8 (CC).
 RHN has been supported in part by NSF grant DMS-1908267. These supports are gratefully acknowledged. LBdV, CC, MV and GV are members of the INdAM research group GNCS.


\bibliographystyle{plain}
\bibliography{biblio}

\end{document}